\documentclass[a4paper,11pt]{article}
\usepackage{graphicx}
\usepackage{amssymb}
\usepackage{enumerate}
\usepackage{array}
\usepackage{geometry}
\usepackage{enumitem}
\usepackage{tikz-cd}
\usepackage{fancyhdr}
\usepackage{amsmath}
\usepackage{verbatim}
\usepackage{url}
\usepackage{theoremref}
\usepackage{amscd}
\usepackage{mathrsfs}
\usepackage{color}

\DeclareMathOperator{\rank}{rank}

\DeclareMathOperator{\stab}{Stab}

\DeclareMathOperator{\CZ}{CAT(0)}

\newtheorem{theorem}{Theorem}[section]
\newtheorem{lemma}[theorem]{Lemma}
\newtheorem{proposition}[theorem]{Proposition}
\newtheorem{corollary}[theorem]{Corollary}
\newtheorem{definition}[theorem]{Definition}
\newtheorem{remark}[theorem]{Remark}
\newtheorem{example}[theorem]{Example}

\newenvironment{proof}[1][Proof]{\begin{trivlist}
\item[\hskip \labelsep {\bfseries #1}]}{\end{trivlist}}

\begin{document}

\title{Accessibility of partially acylindrical actions}
\author{Michael Hill}
\date{\today}
\maketitle

%Weidmann's Remark
%
%It seems that the ideas presented in this note, together with the proof of BF-accessibility, should
%also yield an acylindrical version of accessibility for finitely presented groups where the condition is
%that stabilizers of long segments are small. Indeed, Dunwoody’s proof of Linnell accessibility and the
%proof of BF-accessibility are very similar in that they start a folding sequence with a well-understood
%splitting and gradually make it more complicated, on the way the complexity of the edge groups is
%being measured. In the case of Linnell accessibility, the original splitting is a wedge of circles and the
%measure of complexity for the edge groups is their order, while in the proof of BF-accessibility, the
%original splitting is a Dunwoody resolution and the measure of complexity is the complexity of the
%action on the Bass–Serre tree of the Dunwoody resolution.

\begin{abstract}
In \cite{Weidmann_kC} Weidmann shows that there a bound on the number of orbits of edges in a tree on which a finitely generated group acts $(k,C)$--acylindrically. In this paper we extend this result to actions which are $k$--acylindrical except on a family of groups with ``finite height''. We also give an example which gives a negative result to a conjecture of Weidmann from the same paper and produce a sharp bound for groups acting $k$--acylindrically. 
\end{abstract}

Given a group $G$ one question we would like to answer is if there is some bound on the size of a graph of groups decomposition for it. For example Grushko's theorem \cite{Grushko} implies that a free decomposition has at most $\rank(G)$ vertices with non-trivial label. Further examples include a result by Dunwoody \cite{Dunwoody_fp} which gives a bound for the number of edges of a (reduced) splitting over finite edge groups for a given finitely presented group, and by Bestvina and Feighn \cite{BestvinaFeighn} who extended this to \emph{small} edge groups. (A group is \emph{small} if it doesn't act hyperbolically on any tree.) These results do not extend to finitely generated groups; for instance Dunwoody \cite{Dunwoody_inaccess} gives an example of a finitely generated group which has splittings over finite edge groups with an arbitrary number of edges. 

It's also possible to obtain bounds by imposing restrictions other than restricting the class of edge groups. For example Sela \cite{Sela} showed that there is a bound for the size of a (minimal) splitting where the action on the corresponding Bass-Serre tree is $k$--acylindrical, assuming that $G$ is freely indecomposable and finitely generated. Weidmann \cite{WeidmannInitial} later reproved this and gave a nice bound of at most $2k(\rank G - 1)$ edges for the splitting. Later Delzant \cite{DelzantAcylindrical} showed that a bound for $(k,C)$--acylindrical actions exists provided that the acting group is finitely presented. Weidmann \cite{Weidmann_kC} then extended this to finitely generated groups and conjectured that there should be a common generalisation between this and Bestvina and Feighn's aforementioned result on actions with small edge stabilisers. More precisely they suggest that the currently known techniques should be enough to show that we can obtain a bound for a finitely presented group acting on a tree where \emph{large} subgroups fix at most region of bounded diameter \cite[pg.213]{Weidmann_kC}. (A group is \emph{large} if it's not small.) 

In this paper we extend Weidmann's result on $(k,C)$--acylindrical actions to actions which are $k$--acylindrical except on a family of groups with ``finite height''. Roughly speaking this means that if there is a family of subgroups $\mathcal{P}$ with a bound on the length of chains of subgroups and the diameter of the fixator of any group not in $\mathcal{P}$ is bounded by $k$, then we obtain a bound on the number of edges. (A more precise statement can be found in Section~\ref{sec:statements}.) In Section~\ref{sec:folds} we recall the definition and basic properties of Stallings folds \cite{Stallings}. In Section~\ref{sec:example} we construct examples which show that Weidmann's aforementioned conjecture from \cite[pg.213]{Weidmann_kC} is false. In Section~\ref{sec:influence} we introduce the notion of a \emph{forest of influence} and show that it interacts nicely with Stallings folds. We then apply this machinery in Sections~\ref{sec:reduced} and \ref{sec:extension} to prove the main positive results. Finally in Section~\ref{sec:sharp} we improve Weidmann's bound from \cite{WeidmannInitial} for actions which are $k$--acylindrical to have at most $2k \left( \rank G - \frac{5}{4} \right)$ edges (for non-cyclic groups) and show that this bound is sharp.

\section{Statement of positive results}\label{sec:statements}

We begin by recalling the natural correspondence between graph of groups decompositions of $G$ and trees on which $G$ acts \cite{Serre_trees}. If we take the ``universal cover'' of graph of groups we obtain a tree on which the fundamental group of the graph of groups acts. We call this tree the \emph{Bass-Serre tree}. Likewise we can quotient an (orientation preserving) action of $G$ on a tree to obtain a graph of groups whose fundamental group is $G$. Throughout we will implicitly use this correspondence. 

We now give a series of basic definitions.

\begin{definition}
Let $G$ be a group and $k$ be a non-negative integer. Let $\mathcal{Q}$ be a class of subgroups of $G$ which is closed under conjugation. An action of $G$ on a tree $T$ is \emph{(partially) $k$--acylindrical on $\mathcal{Q}$} if whenever some $H \in \mathcal{Q}$ fixes every edge in a reduced edge path $p$ then $p$ contains at most $k$ edges. 

If $\mathcal{Q}$ isn't specified then it's assumed to contain all the non-trivial subgroups of $G$ and we say the action is \emph{$k$--acylindrical}. An action is \emph{$(k,C)$--acylindrical} if it's $k$--acylindrical on the subgroups of $G$ with size strictly greater than $C$.
\end{definition}

\begin{definition}
The \emph{rank} of a finitely generated group $G$ is the minimal size of a generating set for $G$. Denote this quantity by $\rank G$.
\end{definition}

There a couple of trivial ways in which we can add arbitrarily many orbits of edges to a tree. The first is to add additional ``hanging'' edges which tell us nothing about the structure of the group; for example the splitting $G \cong G *_H H$ where $H \leqslant G$. The following definition prevents this.

\begin{definition}
The action of a group on a tree is said to be \emph{minimal} if there are no proper subtrees which are invariant under the action. 
\end{definition}
 
Another thing that we need to prevent is the possibility of arbitrary subdivision of edges. The notion of a tree being \emph{reduced} stops this behaviour; however a tree being $k$--acylindrical also stops this happening on edges whose stabiliser contains a member of $\mathcal{Q}$. As such we can introduce the following less restrictive notion.
 
\begin{definition}
Suppose $\mathcal{P}$ is a class of subgroups of $G$ which is closed under conjugation. A minimal action is said to be \emph{partially-reduced} over $\mathcal{P}$ if either

\begin{itemize}
\item $T/G$ is a circle consisting of a single vertex and edge; or
\item whenever a vertex $v$ of $T$ has stabiliser equal to that of an edge of $T$ and which is contained in a subgroup of a member of $\mathcal{P}$ then $v/G$ has valence at least $3$ in $T/G$.
\end{itemize} 

We say an action is \emph{reduced} if it's partially-reduced over the class of all subgroups of $G$. We also say an action is \emph{$C$--partially-reduced} if it's partially-reduced over the class of all subgroups of size at most $C$. 
\end{definition}

A critical element of the proof will be the idea of measuring how ``large'' a given group is relative to $\mathcal{P}$. 

\begin{definition}\thlabel{def:p_weight}
Let $\mathcal{P}$ be a conjugation invariant set of subgroups of $G$. For a subgroup $K \leqslant G$ suppose there is some maximal integer $n$ such that there are $H_1, \cdots, H_n \in \mathcal{P}$ with 
\begin{equation*}
K \leqslant H_1 < H_2 < \cdots < H_n
\end{equation*}
We define the \emph{$\mathcal{P}$--weight} of $K$ to be $2^n$ and we denote this quantity by $W_{\mathcal{P},K}$. (We say this is equal to $\infty$ if chains of arbitrary length exist.) We say that $K \leqslant G$ is \emph{larger than} $\mathcal{P}$ if it's not a subgroup of a member of $\mathcal{P}$; equivalently if $W_{\mathcal{P},K} = 1$. 

If $G$ acts on $T$ then we define the $\mathcal{P}$--weight of each edge to be the $\mathcal{P}$--weight of its stabiliser. We say that $\mathcal{P}$ has \emph{height} $M$ if the maximal weight of any $K \leqslant G$ is $2^M$; equivalently if $W_{\mathcal{P},1} = 2^M$.
\end{definition}

\begin{remark}
Since we insist that $\mathcal{P}$ is conjugation invariant we see that the $\mathcal{P}$-weight of a subgroup is conjugation invariant. As such we define the $\mathcal{P}$-weight of a conjugacy class of subgroups to be the $\mathcal{P}$-weight of any representative of that class.
\end{remark}

We now state an easier version of our main results. We will prove this before moving on to the full theorems as it will demonstrate the important ideas of the argument without being obscured by as many technical details.

\begin{theorem}\thlabel{res:main_simple}
Let $G$ be a finitely presented group and let $\mathcal{P}$ be a set of subgroups for $G$ with height $M$ and which is closed under conjugation and taking subgroups. Suppose $G$ acts on a tree $T$ and that this action is both $k$--acylindrical on groups larger than $\mathcal{P}$ and partially reduced on groups in $\mathcal{P}$. Then there is some $C(G)$ (which depends only on $G$) such the that number of edges of $T/G$ is bounded above by $(2k+1)2^M C(G)$.
\end{theorem}

Our main results are two different generalisations of the above. In the first we extend the result to certain cases where $\mathcal{P}$ isn't closed under taking subgroups, which is necessary for including infinite subgroups in $\mathcal{P}$. In the second we extend the result to groups which are merely finitely generated instead of just those which are finitely presented. In order to state the former we first need to make the following definitions.

\begin{definition}
Let $\mathcal{P}$ be a set of subgroups for $G$. Let $K$ be a subgroup of $G$. We say $H \in \mathcal{P}$ is a \emph{minimal extension of $K$ (to $\mathcal{P}$)} if $K \leqslant H$ and whenever $\tilde{H} \in \mathcal{P}$ with $K \leqslant \tilde{H} \leqslant H$ then $\tilde{H} = H$. We say that $K$ is \emph{$\mathcal{P}$--closed} if every minimal extension of any subgroup of $K$ is contained in $K$. We say an action of $G$ on a tree $T$ is \emph{$\mathcal{P}$--closed} if all its edge stabilisers are $\mathcal{P}$--closed. 
\end{definition}

\begin{remark}
If $\mathcal{P}$ has finite height then minimal extensions always exist for any group which isn't larger than $\mathcal{P}$.
\end{remark}

% The \emph{$\mathcal{P}$--closure} of $K$ (denoted $K_{\mathcal{P}}$) is defined to be the unique minimal subgroup of $G$ which contains $K$ and is $\mathcal{P}$--closed. (This is equal to the intersection of the subgroups of $G$ which are $\mathcal{P}$--closed and contain $K$. This intersection is non-empty as $G$ satisfies these conditions.) 

\begin{definition}\thlabel{def:condition_dagger}
We say that $\mathcal{P}$ satisfies condition $(\dagger)$ if the following conditions hold.
\begin{itemize}
\item $\mathcal{P}$ has finite height.
\item Suppose $G$ acts on a tree $T$ and let $e$ be an edge of $T$. Then for any subgroup $K \leqslant \stab e$ which is not larger than $\mathcal{P}$ there is a vertex $v$ of $T$ which is fixed by some minimal extension of $K$ to $\mathcal{P}$. In particular this always holds if each minimal extension is of finite index as a finite index extension of an elliptically acting group also fixes a point \cite{Serre_trees}.
\item Every member of $\mathcal{P}$ is $\mathcal{P}$-closed. Equivalently if $H_1$ and $H_2$ are in $\mathcal{P}$ then so is $H_1 \cap H_2$. A third equivalent condition is that minimal extensions to $\mathcal{P}$ are unique.
\end{itemize} 

% In particular $(\dagger)$ always holds if each minimal extension to $\mathcal{P}$ is of finite index as a finite index extension of an elliptically acting group also fixes a point \cite{Serre_trees}.

% Moreover we say $T$ satisfies condition $(\dagger\dagger)$ (on $\mathcal{P}$) if $K_{\mathcal{P}}$ also fixes the edge $e$. In other words $T$ satisfies $(\dagger\dagger)$ iff all its edge groups are $\mathcal{P}$--closed.
\end{definition}

\begin{example}\thlabel{ex:hyperbolic_dagger}
Suppose that $G$ is a torsion-free hyperbolic group. Take $\mathcal{P}$ to be the set of cyclic subgroups of $G$ which are root-closed. Equivalently $\mathcal{P}$ is the set of maximal cyclic subgroups of $G$. This $\mathcal{P}$ satisfies $(\dagger)$ since every cyclic subgroup is contained (with finite index) in a unique maximal cyclic subgroup. Moreover a tree which $G$ acts on is $\mathcal{P}$--closed iff all its edge stabilisers are root-closed.
\end{example}

We are now ready to state our main result, which as mentioned before is a pair of extensions of \thref{res:main_simple}.

\begin{theorem}\thlabel{res:main_result}
Let $G$ be a finitely generated group and let $\mathcal{P}$ be a set of subgroups for $G$ which is closed under conjugation and has height $M$. Suppose $G$ acts on a tree $T$ and that this action is both $k$--acylindrical on groups larger than $\mathcal{P}$ and partially reduced on subgroups of members of $\mathcal{P}$. Then the following statements hold.
\begin{enumerate}[label=(\alph*)]
\item \label{pt:main_fp} Suppose $G$ is finitely presented, $\mathcal{P}$ satisfies $(\dagger)$ and $T$ is $\mathcal{P}$--closed. Then there is some integer $C(G)$ such that the number of edges of $T/G$ is bounded above by $(2k+1)2^M C(G)$.
\item \label{pt:main_P_closed} If $\mathcal{P}$ is closed under taking subgroups then the number of edges of $T/G$ is bounded above by $\left(\frac{2k+1}{2}\right) 2^{M} (\rank G - 1)$. Moreover suppose either of the following conditions hold 
\begin{itemize}
\item $T$ is reduced and $k > 1$; or
\item every edge stabiliser of $T$ is not in $\mathcal{P}$;
\end{itemize}
then if $G$ isn't cyclic the number of edges of $T/G$ is bounded above by ${2^{M} k (\rank G - 1)}$.
\end{enumerate}
\end{theorem}

\begin{remark}
In \thref{res:main_simple} and \thref{res:main_result} \ref{pt:main_fp} the bound $C(G)$ is given by Dunwoody's resolution lemma (\thref{res:dunwoody_resolution}). Dunwoody's resolution lemma holds for so called \emph{almost finitely presented groups} and this extends to both \thref{res:main_simple} and \thref{res:main_result} \ref{pt:main_fp}. (A group is \emph{almost finitely presented} if it's both finitely generated and acts freely, simplicially and cocompactly on a simplicial complex $X$ with $H^1(X,\mathbb{Z}_2) = 0$.)
\end{remark}

The following is an immediate consequence of \thref{res:main_result} \ref{pt:main_P_closed} and is an extension of Weidmann's result on $(k,C)$--acylindrical actions \cite{Weidmann_kC}. In particular this shows that the number of prime factors is the limiting factor, not the absolute size of the group.

\begin{corollary}\thlabel{res:main_finite_edges}
Let $G$ be a finitely generated group and $M \in \mathbb{N}$. Suppose $G$ acts on a tree $T$ and that this action is both $k$--acylindrical on groups which are infinite or have at least $M$ prime factors and partially reduced on subgroups with at most $M-1$ prime factors. (Where the number of prime factors is counted with multiplicity.) Then the number of edges of $T/G$ is bounded above by $\left(\frac{2k+1}{2}\right) 2^{M} (\rank G - 1)$. Moreover if either $T$ is reduced and $k > 1$ or every edge stabiliser of $T$ is either infinite or has at least $M$ prime factors then the number of edges of $T/G$ is bounded above by $2^{M} k (\rank G - 1)$.
\end{corollary}

\begin{proof}
Let $\mathcal{P}$ be the set of finite subgroups of $G$ whose order has at most $M-1$ prime factors. Observe that $\mathcal{P}$ has height of at most $M$ and is closed under taking subgroups. The result now immediately follows from \thref{res:main_result} \ref{pt:main_P_closed}. $\square$
\end{proof}

We also apply \thref{res:main_result} \ref{pt:main_fp} to a couple of specific cases to get some interesting results. The first case is a generalisation of \thref{ex:hyperbolic_dagger}; where a torsion-free hyperbolic group acts on a tree with root-closed edge stabilisers. We now allow the group is to have finite order elements and the root closed condition is replaced by one which says that the maximal virtually $\mathbb{Z}$ subgroups of edge stabilisers should be ``almost'' maximal in $G$. In the second we consider splittings of RAAGs which are $k$--acylindrical on its non-abelian subgroups and the maximal abelian subgroups of edge stabilisers should be maximal for their rank.

\begin{corollary}\thlabel{res:main_hyperbolic}
Let $G$ be a hyperbolic group. We say a virtually $\mathbb{Z}$ subgroup $H \leqslant G$ is \emph{$m$-almost maximal} if whenever we have a virtually cyclic $K \leqslant G$ with $H \leqslant K$ then $\left[ H:K \right] \leq m$. Let $\mathcal{P}_m$ be the collection of subgroups of $G$ which are either finite or $m$-almost maximal. 

Suppose $G$ acts on a tree $T$ which is partially reduced on $\mathcal{P}_m$ and $k$--acylindrical on groups larger than $\mathcal{P}_m$. Suppose also that $T$ is $\mathcal{P}_m$-closed. Then the number of edges of $T/G$ is bounded above by $(2k+1) C'(G)$. (Where $C'(G) = 2^{n}C(G)$ for some $n \in \mathbb{N}$.)
\end{corollary}

\begin{proof}
Observe that $\mathcal{P}_m$ is not closed under intersections and so doesn't satisfy condition $(\dagger)$. Instead define $\mathcal{P}'_m$ to be the set of subgroups which are a (finite) intersection of groups in $\mathcal{P}_m$. It's clear that if $\mathcal{P}'_m$ has finite height then it satisfies $(\dagger)$. As a hyperbolic group has only finitely many conjugacy classes of finite subgroups \cite{FiniteHyperbolic} we just need to show that chains of infinite subgroups in $\mathcal{P}'_m$ have bounded length. It therefore suffices to show that given any index in $\mathbb{N}$ there is a uniform bound on the number of subgroups of that index for any virtually cyclic subgroup of $G$. 

We have that every virtually cyclic $H \leqslant G$ is either of the form $H \cong K*_{K}$ where $K$ is finite or $H \cong A*_{C}B$ where $A,B,C$ are finite and $C$ is an index $2$ subgroup of both $A$ and $B$ \cite{VirCycStructure}. The general description of a subgroup of the fundamental group of a graph of groups \cite{Serre_trees} now tells us there's a bound on the number of subgroups of $H$ of a given index which depends only on the index and the size of either $K$ or $C$ respectively. Again a hyperbolic group has finitely many conjugacy classes of finite subgroups, which uniformly bounds the order of $K$ and $C$.  This implies the result. $\square$
\end{proof}

%
%**********************************************************************************************************
%
% WE CAN MAKE THIS SATISFY THE NEW DAGGER IF WE REPLACE WITH AT MOST M PRIME FACTORS WITH DIVIDES SOME M.
%
% MIGHT NEED TO EXCLUDE TORSION AGAIN...
%
%**********************************************************************************************************

Before stating the other application we briefly recall the definition of a RAAG.

\begin{definition}
Let $\Gamma$ be a finite graph. Let $v_1, \cdots , v_n$ be the vertices of $\Gamma$. The \emph{right-angled Artin group} (RAAG) associated to $\Gamma$ is the group $A(\Gamma)$ where
\begin{equation*}
A(\Gamma) := \left\langle v_1, \cdots , v_n \:\: | \:\: [v_i , v_j] \:\: \text{wherever there's an edge between } v_i \text{ and } v_j \text{ in } \Gamma. \right\rangle
\end{equation*}
\end{definition}

\begin{corollary}\thlabel{res:main_raag}
Let $G = A(\Gamma)$ be a RAAG. An abelian subgroup $H \leqslant G$ is said to be \emph{rank maximal} if whenever we have an abelian subgroup $K \leqslant G$ which contains $H$ with finite index we have $K = H$. Let $\mathcal{P}$ be the collection of rank maximal abelian subgroups of $G$. Suppose $G$ acts on a tree $T$ which is partially reduced on abelian subgroups and $k$--acylindrical on non-abelian subgroups. Suppose also that $T$ is $\mathcal{P}$-closed. Then the number of edges of $T/G$ is bounded above by $(2k+1) 2^{n}C(G)$ (where $n$ is the size of the largest complete subgraph of $\Gamma$).
\end{corollary}

\begin{proof}
Recall that a RAAG acts freely and cocompactly on a simply connected $\CZ$-cube complex $X_{\Gamma}$ whose dimension is equal to the size of the largest complete subgraph of $\Gamma$.  The Flat Torus Theorem \cite[Theorem II.7.1]{BridsonHaefliger} says that an abelian subgroup $H \leqslant A(\Gamma)$ must act properly and cocompactly by isometries on a Euclidean hyperplane of $X_{\Gamma}$. In particular $H$ must have rank at most equal to the size of the largest complete subgraph of $\Gamma$. Hence $\mathcal{P}$ has finite height.

Now every rank $2$ subgroup $\left\langle u,v \right\rangle \leqslant A(\Gamma)$ is either free abelian or free \cite[Theorem~1.2]{Baudisch}. Suppose $H \leqslant A(\Gamma)$ is an abelian subgroup, $u$ is a root of an element of $H$ and pick any $v \in H$. Since a power of $u$ is in $H$ and $H$ is abelian we see that $\left\langle u,v \right\rangle$ cannot be non-abelian free and so $u$ and $v$ must commute. Hence every member of $\mathcal{P}$ is root closed.

%Consider $G$ as a series of HNN extensions starting from the trivial group. Observe that Britton's lemma implies (by induction on the number of vertices of $\Gamma$) that every non-cyclic abelian subgroup of $G$ is conjugate into some $A(\Gamma') \leqslant A(\Gamma)$ where $\Gamma'$ is a complete subgraph of $\Gamma$. In particular this implies that $\mathcal{P}$ has finite height and that each abelian subgroup of $G$ has a finite index extension to one which is rank maximal. In addition we see that every member of $\mathcal{P}$ is root closed as $A(\Gamma')$ is a root closed subgroup of $A(\Gamma)$ where $\Gamma'$ is a subgraph of $\Gamma$. 

It remains to check that $\mathcal{P}$ satisfies condition $(\dagger)$, then we can apply \thref{res:main_result} \ref{pt:main_fp} to get the result. Pick any $H \in \mathcal{P}$ and $K \leqslant H$ and let $M \in \mathcal{P}$ be a minimal extension of $K$. We must have $M \leqslant H$ as $H$ is root closed and $K$ is a finite index subgroup of $M$. Hence $\mathcal{P}$ is $\mathcal{P}$-closed and hence satisfies $(\dagger)$. $\square$
\end{proof}

\begin{remark}
For a general group $G$ it need not be the case that $\mathcal{P}$ as defined in \thref{res:main_raag} satisfies $(\dagger)$. For example if $G \cong \mathbb{Z} *_{2\mathbb{Z}} \mathbb{Z}$ then $G$ acts freely and cocompactly on a $\CZ$ space; but contains two rank maximal copies of $\mathbb{Z}$ whose intersection is another copy of $\mathbb{Z}$ which is not rank maximal.
\end{remark}

\thref{res:main_result} \ref{pt:main_P_closed} also immediately implies Weidmann's earlier result on $k$--acylindrical actions \cite{WeidmannInitial}; which says that a finitely generated group acting $k$-acylindrically on a tree without edges with trivial stabiliser has at most $2(\rank G - 1)k$ orbits of edges. Indeed with slightly more work we'll show it's possible to improve their bound sightly further to one which we'll show is the best possible.

\begin{theorem}\thlabel{res:sharp_bound}
Let $G$ be a (non-cyclic) finitely generated group acting $k$--acylindrically on a minimal tree $T$ (where $k \geq 1$.) Suppose that each edge of $T$ has non-trivial stabiliser. Then $T/G$ has at most $\left\lfloor \left(2\rank G - \frac{5}{2}\right)k \right\rfloor$ edges. If $G$ is torsion-free then this bound can be improved to $\left( 2 \rank G - 3 \right) k$.
\end{theorem}

\begin{theorem}\thlabel{res:sharp_example}
For any $k > 0$ and $r \geq 2$ there is a finitely presented group $G$ with $\rank G = r$ which acts $k$--acylindrically on a minimal tree $T$ where each edge of $T$ has non-trivial stabiliser and $T/G$ has exactly $\left\lfloor \left(2\rank G - \frac{5}{2}\right)k \right\rfloor$ edges. 

Similarly $F_r$ admits a $k$--acylindrical action on a minimal tree $T$ where each edge of $T$ has non-trivial stabiliser and $T/F_r$ has exactly $\left(2r - 3 \right)k$ edges.
\end{theorem}

\begin{remark}
Unlike in the previous results there is no requirement that $T$ needs to be reduced. Instead the conditions that $T$ is $k$--acylindrical and has no edges with trivial stabiliser are enough to completely prevent the unrestricted edge subdivision which motivated the definition of a reduced action.
\end{remark}

\section{Stallings folds}\label{sec:folds}

The idea of a fold will be of vital importance. Recall the following.

\begin{definition}
Let $G$ act on a tree $T$. Let $e_1$, $e_2$ be distinct edges with a common endpoint $x$ and let $\phi:e_1 \rightarrow e_2$ be the linear map which leaves $x$ fixed. Let $\sim$ be the minimal equivalence relation on $T$ such that $x \sim \phi(x)$ for each $x \in e_1$ and such that $T/\sim$ is a naturally a tree on which $G$ acts. A \emph{fold} is the map $T \rightarrow T/\sim$.
\end{definition}

Folds were introduced by Stallings in \cite{Stallings}. In particular they showed that maps between trees with finitely generated edge stabilisers can be decomposed into a finite sequence of folds. We need something similar for trees whose edge groups need not be finitely generated. Fortunately we can ``add in'' generators of each edge group one at a time, then take a limit to see that our maps are a composition of (potentially infinitely many) folds. This is formalised into the following theorem, the proof of which is largely the same as the one given in \cite[p.455]{BestvinaFeighn} with changes to deal with the fact that the edge stabilisers aren't necessarily finitely generated.

\begin{theorem}\thlabel{res:stallings_folding_thm}
Let $G$ be a countable group. Suppose $\Psi: S \rightarrow T$ is a surjective simplical equivariant map between trees which $G$ acts on with $S/G$ finite and where no edge of $S$ gets mapped to a point by $\Psi$. Then $\Psi$ can be viewed as a (possibly infinite) composition of folds. i.e. $\Psi = \cdots\alpha_2\alpha_1$ where each $\alpha_i$ is an (orientation preserving) fold.
\end{theorem}

\begin{remark}
The codomain of $\cdots\alpha_2\alpha_1$ is $S/\sim$ where $\sim$ is the equivalence relation generated by all the $\alpha_i$. This is a tree which $G$ acts on in the obvious way with vertex and edge stabilisers equal to the natural direct limit of their preimages.
\end{remark}

\begin{remark}
The condition that no edge of $T$ gets collapsed is not a restrictive one in practice. In particular if $\Psi$ maps an edge of $S$ to a point then let $\pi$ be the map which collapses each edge of $S$ which is sent to a point by $\Psi$. Then there is a natural composition $\Psi = \Psi' \circ \pi$ where no edge in the domain of $\Psi'$ is sent to a point in $T$.
\end{remark} 

\begin{proof}
Throughout we'll let $\alpha_i$ fold the tree $S_{i-1}$ into the tree $S_i$. Also we let $\beta_i := \alpha_i \circ \cdots \circ \alpha_1$ and $\gamma_i$ be the map such that $\Psi = \gamma_i \circ \beta_i$. 

First suppose that we have a surjective simplical map $m: A \rightarrow B$ between finite trees where no edge gets mapped to a point. Claim that $m$ can be considered to be a finite series of folds; a fact we shall refer to as $(\star)$. If $m$ is injective then this is trivial. Otherwise we have distinct vertices $x,y \in A$ with $m(x) = m(y)$. Let $p$ be the reduced edge path from $x$ to $y$. Since every edge of $A$ is mapped to an edge in $B$ and $B$ is a tree we see that there must be a vertex $z \in p$ such that $m |_p$ is not locally injective at $z$. Let $e_1$ and $e_2$ be the edges in $p$ which contain $z$ as an endpoint and observe that $m(e_1) = m(e_2)$. Thus $m$ factors though the fold with edges $e_1$ and $e_2$. Repeat this process until the map is injective, which must happen as the number of edges is finite and decreasing at each stage. This completes the decomposition of $m$ into folds. 

Now suppose we have an equivarient simplical map $\delta: R \rightarrow R'$. Let $A$ be a finite subtree of $R$. We can apply $(\star)$ to $\delta |_{A}$ to obtain a finite series of folds $\delta': R \rightarrow R''$ which factors through $\delta$ and where the corresponding $\delta'': R'' \rightarrow R'$ is injective on $\delta'(A)$. This is how we will apply $(\star)$ in practice. 

Our initial folds of $\Psi$ will be to set $S_i/G$ isomorphic as a graph to $T/G$ for all sufficiently large $i$. Let $F$ be the closure of a fundamental domain of $S = S_0$. We now use $(\star)$ to find a series of folds $\alpha_1, \cdots, \alpha_N$ such that $\gamma|_{\beta(F)}$ is a homeomorphism onto its image. Thus $\gamma_N / G$ is a homeomorphism of graphs. 

Let $K = \Psi(F)$. Let $v$ be a vertex in $K$ and $g \in \stab(v)$. Let $v_i$ be a preimage of $v$ in $F_i$ and let $p_i$ be the reduced edge path from $v_i$ to $gv_i$. Now we apply $(\star)$ to $p_i$ in order to get folds $\alpha_{i}, \cdots \alpha_{j}$ which get $g$ in the relevant vertex group in $S_j/G$. Now repeat this process for each vertex $v \in K$ and $g \in \stab(v)$. This potentially gives an infinite sequence of folds as $K$ has finitely many vertices and each $\stab(v)$ is countable. 

Now let $\tilde{\gamma}$ be the map such that $\Psi = \tilde{\gamma} \circ (\cdots \circ \alpha_2 \circ \alpha_1)$. Claim that $\tilde{\gamma}$ is a homeomporphism. Indeed by construction we see that $\tilde{\gamma}$ induces a bijection between the orbits of vertices; moreover the stabiliser of each vertex is the same as that of its image. Hence $\tilde{\gamma}$ induces a bijection between the vertices and hence is a homeomorphism between trees. Thus $\cdots \circ \alpha_2 \circ \alpha_1$ is a decomposition of $\Psi$ into folds. $\square$
\end{proof}

\begin{remark}
It should be straightforward to extend this result to the case where $G$ is uncountable and where $S/G$ is not necessarily finite using the well ordering principle. We do not do this here because it is unnecessary to prove our main results.
\end{remark}

%\begin{remark}
%It is possible to insist that the decomposition given by Stallings' folding theorem (\thref{res:stallings_folding_thm}) is finite as long we replace folds with \emph{generalised folds} CITE. We don't do so here because this infinite chain causes no real extra difficulty to work with and because the present author believes there is some value in having both approaches in the literature. We note that all arguments in this paper involving folds work just as well for generalised folds. 
%\end{remark}

In general a simplicial map is too restrictive of a notion. As such we now introduce the idea of a combinatorial map.

\begin{definition}
A \emph{combinatorial map} $\Psi: S \rightarrow T$ is a $G$--equivariant map where each vertex gets sent to a vertex and each edge $e = [u,v]$ gets sent to the reduced edge path from $\Psi(u)$ to $\Psi(v)$. 
\end{definition}

Observe that a combinatorial map can be viewed as a simplical map after subdividing edges in the domain. As such \thref{res:stallings_folding_thm} applies to combinatorial maps as long as we subdivide first. 

The following will be a useful shorthand.

\begin{definition}
Suppose that $\Psi: T \rightarrow T'$ is a combinatorial map. We say that another combinatorial map $\alpha: R \rightarrow R'$ \emph{factors through $\Psi$} if there is some $\beta: T \rightarrow R$ and $\gamma: R' \rightarrow T'$ where $\Psi = \gamma \circ \alpha \circ \beta$.
\end{definition}

Depending which of the vertices and edges are in common $G$--orbits there are a few different cases that can arise from a fold. The following classification of folds is the same as the one found in \cite{BestvinaFeighn}. 

First we make the distinction between whether $x$ is in the same $G$ orbit as one of the $y_i$. If $x$ is not in the same $G$--orbit as either $y_i$ we say that the fold is of type A. Otherwise WLOG we have $gx = y_1$ for some $g \in G$ and we say that the fold is of type B. Note that such a $g$ must act hyperbolicily on $T$, (with translation length $1$,) as it moves a vertex an odd distance. 

Additionally we split each of these cases into three additional categories. We say the fold is of type I if $y_1$ and $y_2$ are in distinct orbits of $G$. We say the fold is of type II if $e_1$ and $e_2$ are in a common orbit of $G$. Finally we say the fold is of type III if $y_1$ and $y_2$ are in a common $G$--orbit, but $e_1$ and $e_2$ are not. We will now go into the specifics of each type of fold. Throughout we let $e_i$ be the vertex between the vertices $x$ and $y_i$ and use capital letters to denote the group associated to the corresponding vertex or edge. 

\begin{remark}
The following diagrams represent what happens to the relevant subgraph of a particular graph of groups decomposition. Crucially the pictures for type I and III folds only give the correct groups if both $e_1$ and $e_2$ are in the fundamental domain for this decomposition. In general we need to conjugate certain groups in the decomposition before these pictures become accurate. 
\end{remark}

\begin{proof}[Type I]
We have $y_1$ and $y_2$ in distinct orbits of $G$. In this case the number of vertices and the number of edges of the graph of groups decomposition both decrease by one so the Euler characteristic of the underlying graph stays the same.  

\begin{figure}[h!]
\centering
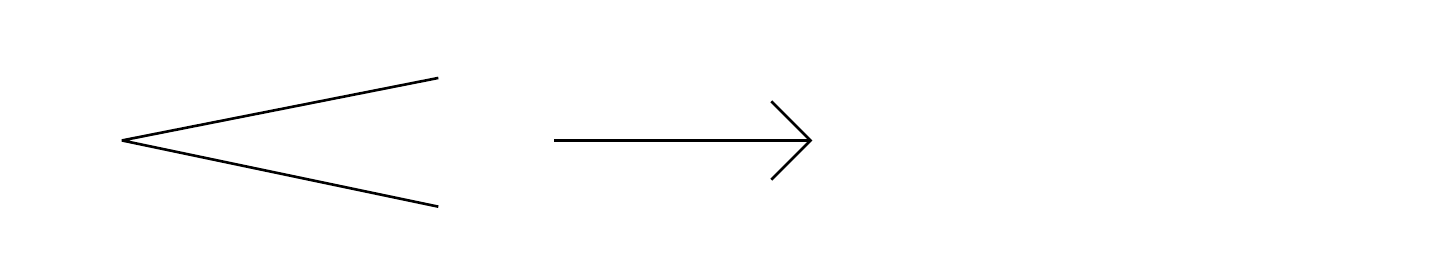 \\
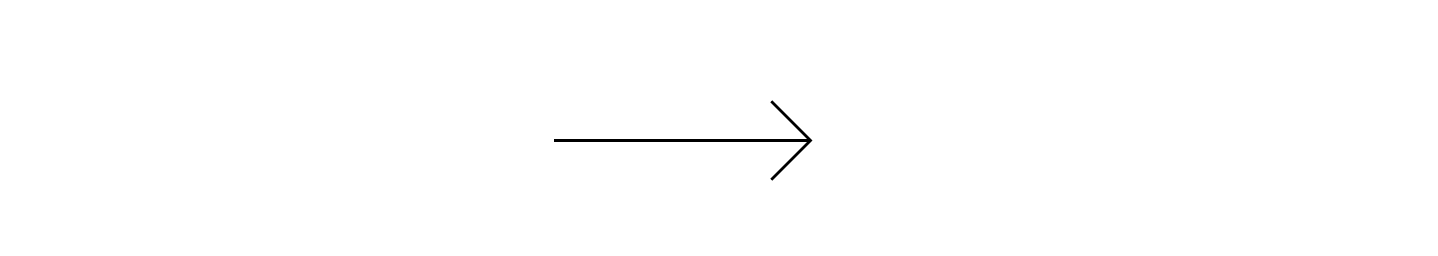 
\caption{A typical example of the effects of a type I fold on a graph of groups. The vertices $y_1$ and $y_2$ are inequivalent so the fold reduces the number of vertices by $1$. Likewise for the edges $e_1$ and $e_2$.}
\end{figure}
\end{proof}

\begin{proof}[Type II]
We have $e_1$ and $e_2$ in a common orbit of $G$, suppose that $he_1 = e_2$. Observe that if $h$ acts hyperbolicily on $T$ then the action of $G$ after the fold is not orientation preserving and so we will ignore this case. Thus we can assume that $h \in X$. In this case the underlying graph of the graph of groups decomposition doesn't change. Instead the element $h$ gets ``pulled'' along the edge in the graph of groups decomposition. 

\begin{figure}[h!]
\centering
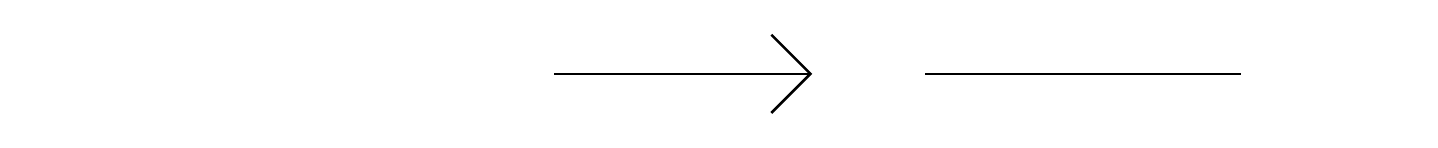 \\
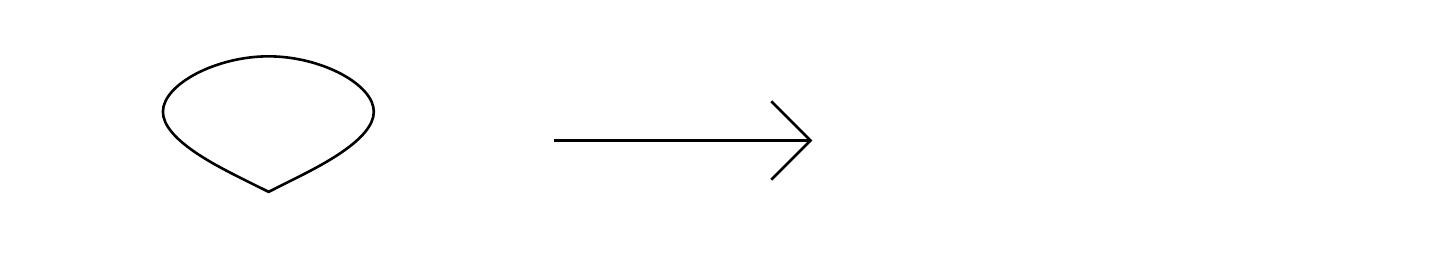  
\caption{A typical example of the effects of a type II fold on a graph of groups. The vertices $y_1$ and $y_2$ are equivalent so the fold keeps the number of vertices the same. Likewise for the edges $e_1$ and $e_2$.}
\end{figure}
\end{proof}

\begin{proof}[Type III]
We have $y_1$ and $y_2$ in a common orbit, but $e_1$ and $e_2$ are not. Suppose that $hy_1 = y_2$. Observe that $h$ has to act hyperbolicily on $T$ with translation length $2$. After the fold this $h$ now fixes the image of $y_1$ and $y_2$ thus no longer acts hyperbolicily. This type of fold reduces the number of edges of the graph of groups by one while keeping the number of vertices fixed. Thus the Euler characteristic  
of the underlying graph increases by one.

\begin{figure}[h!]
\centering
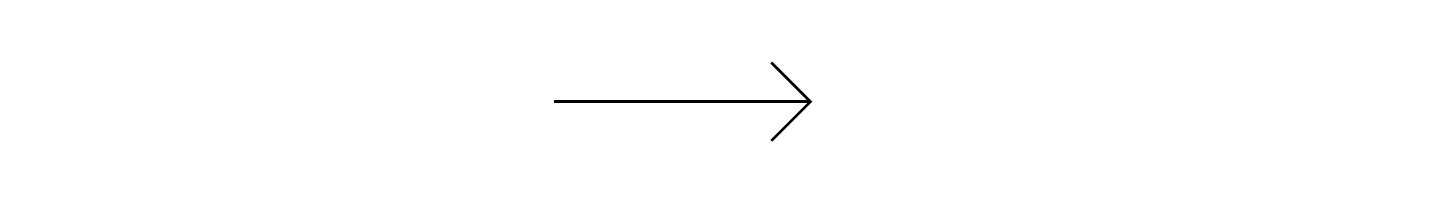 \\
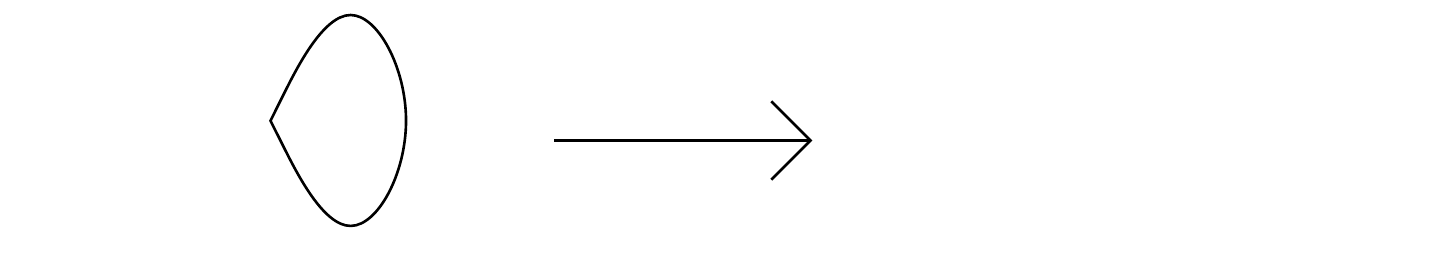 
\caption{A typical example of the effects of a type III fold on a graph of groups. The vertices $y_1$ and $y_2$ are equivalent so the fold keeps the number of vertices the same. However the edges $e_1$ and $e_2$ are inequivalent so the fold reduces the number of edges by $1$.}
\end{figure}
\end{proof}

\section{An inaccessible example}\label{sec:example}

In \cite{BestvinaFeighn} Bestvina and Feighn showed that a reduced tree with small edge stabilisers has a bound on the number of edges as long as the underlying group is (almost) finitely presented. Recall the following.

\begin{definition}
A group $G$ is \emph{large} if it acts on a tree $T$ hyperbolicly. That is to say that there are two elements $g_1, g_2 \in G$ which don't fix a point of $T$ and their axes (lines of minimal displacement) have bounded intersection \cite{Serre_trees}. A group is \emph{small} if it's not large.
\end{definition}

Note that the ping-pong lemma implies that any large group must contain $F_2$ as a subgroup. The converse is not true; for example $SL_3(\mathbb{Z})$ contains many subgroups isomorphic to $F_2$ and has Serre's property (FA) \cite{Serre_trees}, so any tree it acts on has fixed point. 

Later Weidmann \cite{Weidmann_kC} showed that the action of a finitely generated group acting $(k,C)$--acylindrically on a $C$--partially reduced tree also has a bound on the number of edges depending only on the rank of the group and $k$. In the same paper Weidmann then goes on to conjecture that some sort of common generalisation between their result and the aforementioned result of Bestvina and Feighn might exist. More precisely they suggest it should be possible to give a positive answer to the following using known techniques.

\noindent \textbf{Question \cite[pg.213]{Weidmann_kC}} Given a finitely presented group $G$ and $k>0$ is there some $C(G,k)$ such that any reduced action of $G$ which is $k$--acylindrical on large subgroups has at most $C(G,k)$ orbits of edges?

The purpose of this section is to construct an example which shows that the answer to the above question is no. In fact we will construct a counterexample with even stronger properties.

\begin{theorem}\thlabel{res:strongCE}
There is a finitely presented group $G$ which for any $N > 0$ acts on a reduced tree which is $1$--acylindrical on infinite subgroups and has $N$ orbits of edges. 
\end{theorem}

\begin{proof}
Let $D := \left\langle a_1, a_2, \cdots \: | \: a_1^2 = 1, \: a_{i+1}^2 = a_i \:\: \forall i \geq 1 \right\rangle \cong \mathbb{Z}\left[ \frac{1}{2} \right] / \mathbb{Z}$; the additive group of dyadic rationals modulo $\mathbb{Z}$. Let $A$ be any finitely presented group into which $D$ embeds; for example we can take $A$ to be Thompson's group $T$ \cite{ThompsonGroup}. Let $B := \left\langle b \right\rangle \cong \mathbb{Z}$. Take $G := A*B$ and pick any $N > 0$. Start by taking the one edge splitting corresponding to $G \cong A * B$ and subdividing this edge into $N$ subedges. (In the diagrams we take $N = 4$.) 

\noindent%% Creator: Inkscape inkscape 0.92.4, www.inkscape.org
%% PDF/EPS/PS + LaTeX output extension by Johan Engelen, 2010
%% Accompanies image file 'Counterexample1.pdf' (pdf, eps, ps)
%%
%% To include the image in your LaTeX document, write
%%   \input{<filename>.pdf_tex}
%%  instead of
%%   \includegraphics{<filename>.pdf}
%% To scale the image, write
%%   \def\svgwidth{<desired width>}
%%   \input{<filename>.pdf_tex}
%%  instead of
%%   \includegraphics[width=<desired width>]{<filename>.pdf}
%%
%% Images with a different path to the parent latex file can
%% be accessed with the `import' package (which may need to be
%% installed) using
%%   \usepackage{import}
%% in the preamble, and then including the image with
%%   \import{<path to file>}{<filename>.pdf_tex}
%% Alternatively, one can specify
%%   \graphicspath{{<path to file>/}}
%% 
%% For more information, please see info/svg-inkscape on CTAN:
%%   http://tug.ctan.org/tex-archive/info/svg-inkscape
%%
\begingroup%
  \makeatletter%
  \providecommand\color[2][]{%
    \errmessage{(Inkscape) Color is used for the text in Inkscape, but the package 'color.sty' is not loaded}%
    \renewcommand\color[2][]{}%
  }%
  \providecommand\transparent[1]{%
    \errmessage{(Inkscape) Transparency is used (non-zero) for the text in Inkscape, but the package 'transparent.sty' is not loaded}%
    \renewcommand\transparent[1]{}%
  }%
  \providecommand\rotatebox[2]{#2}%
  \newcommand*\fsize{\dimexpr\f@size pt\relax}%
  \newcommand*\lineheight[1]{\fontsize{\fsize}{#1\fsize}\selectfont}%
  \ifx\svgwidth\undefined%
    \setlength{\unitlength}{413.85826772bp}%
    \ifx\svgscale\undefined%
      \relax%
    \else%
      \setlength{\unitlength}{\unitlength * \real{\svgscale}}%
    \fi%
  \else%
    \setlength{\unitlength}{\svgwidth}%
  \fi%
  \global\let\svgwidth\undefined%
  \global\let\svgscale\undefined%
  \makeatother%
  \begin{picture}(1,0.10273973)%
    \lineheight{1}%
    \setlength\tabcolsep{0pt}%
    \put(0.03662987,0.3576232){\color[rgb]{0,0,0}\makebox(0,0)[lt]{\begin{minipage}{0.93116981\unitlength}\centering \end{minipage}}}%
    \put(0,0){\includegraphics[width=\unitlength,page=1]{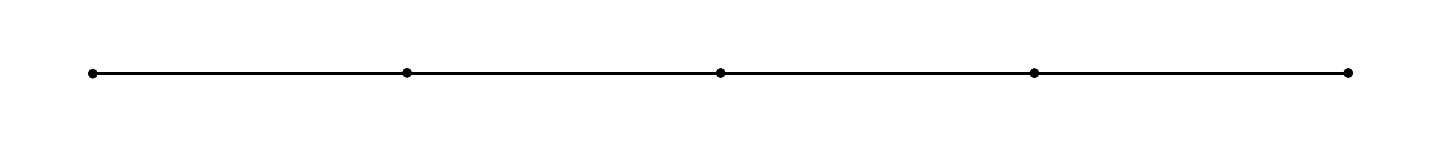}}%
    \put(0.0637779,0.0638342){\color[rgb]{0,0,0.50196078}\makebox(0,0)[t]{\lineheight{1.25}\smash{\begin{tabular}[t]{c}$A$\end{tabular}}}}%
    \put(0.28230374,0.0638342){\color[rgb]{0,0,0.50196078}\makebox(0,0)[t]{\lineheight{1.25}\smash{\begin{tabular}[t]{c}$1$\end{tabular}}}}%
    \put(0.50033517,0.0638342){\color[rgb]{0,0,0.50196078}\makebox(0,0)[t]{\lineheight{1.25}\smash{\begin{tabular}[t]{c}$1$\end{tabular}}}}%
    \put(0.71935539,0.0638342){\color[rgb]{0,0,0.50196078}\makebox(0,0)[t]{\lineheight{1.25}\smash{\begin{tabular}[t]{c}$1$\end{tabular}}}}%
    \put(0.93689247,0.0638342){\color[rgb]{0,0,0.50196078}\makebox(0,0)[t]{\lineheight{1.25}\smash{\begin{tabular}[t]{c}$B$\end{tabular}}}}%
    \put(0.83026505,0.02342433){\color[rgb]{0,0.50196078,0}\makebox(0,0)[t]{\lineheight{1.25}\smash{\begin{tabular}[t]{c}$1$\end{tabular}}}}%
    \put(0.61404769,0.02342433){\color[rgb]{0,0.50196078,0}\makebox(0,0)[t]{\lineheight{1.25}\smash{\begin{tabular}[t]{c}$1$\end{tabular}}}}%
    \put(0.39783026,0.02342433){\color[rgb]{0,0.50196078,0}\makebox(0,0)[t]{\lineheight{1.25}\smash{\begin{tabular}[t]{c}$1$\end{tabular}}}}%
    \put(0.17436041,0.02342433){\color[rgb]{0,0.50196078,0}\makebox(0,0)[t]{\lineheight{1.25}\smash{\begin{tabular}[t]{c}$1$\end{tabular}}}}%
  \end{picture}%
\endgroup%

Now we apply $N$ folds of type II. The first ``pulls'' $a_N$ across the first edge, the second ``pulls'' $a_{N-1}$ across the second edge and so on, so that the $i \textsuperscript{th}$ fold ``pulls'' $a_{N+1-i}$ across the $i \textsuperscript{th}$ edge. 

\noindent%% Creator: Inkscape inkscape 0.92.4, www.inkscape.org
%% PDF/EPS/PS + LaTeX output extension by Johan Engelen, 2010
%% Accompanies image file 'Counterexample2.pdf' (pdf, eps, ps)
%%
%% To include the image in your LaTeX document, write
%%   \input{<filename>.pdf_tex}
%%  instead of
%%   \includegraphics{<filename>.pdf}
%% To scale the image, write
%%   \def\svgwidth{<desired width>}
%%   \input{<filename>.pdf_tex}
%%  instead of
%%   \includegraphics[width=<desired width>]{<filename>.pdf}
%%
%% Images with a different path to the parent latex file can
%% be accessed with the `import' package (which may need to be
%% installed) using
%%   \usepackage{import}
%% in the preamble, and then including the image with
%%   \import{<path to file>}{<filename>.pdf_tex}
%% Alternatively, one can specify
%%   \graphicspath{{<path to file>/}}
%% 
%% For more information, please see info/svg-inkscape on CTAN:
%%   http://tug.ctan.org/tex-archive/info/svg-inkscape
%%
\begingroup%
  \makeatletter%
  \providecommand\color[2][]{%
    \errmessage{(Inkscape) Color is used for the text in Inkscape, but the package 'color.sty' is not loaded}%
    \renewcommand\color[2][]{}%
  }%
  \providecommand\transparent[1]{%
    \errmessage{(Inkscape) Transparency is used (non-zero) for the text in Inkscape, but the package 'transparent.sty' is not loaded}%
    \renewcommand\transparent[1]{}%
  }%
  \providecommand\rotatebox[2]{#2}%
  \newcommand*\fsize{\dimexpr\f@size pt\relax}%
  \newcommand*\lineheight[1]{\fontsize{\fsize}{#1\fsize}\selectfont}%
  \ifx\svgwidth\undefined%
    \setlength{\unitlength}{413.85826772bp}%
    \ifx\svgscale\undefined%
      \relax%
    \else%
      \setlength{\unitlength}{\unitlength * \real{\svgscale}}%
    \fi%
  \else%
    \setlength{\unitlength}{\svgwidth}%
  \fi%
  \global\let\svgwidth\undefined%
  \global\let\svgscale\undefined%
  \makeatother%
  \begin{picture}(1,0.10273973)%
    \lineheight{1}%
    \setlength\tabcolsep{0pt}%
    \put(0.03662987,0.3576232){\color[rgb]{0,0,0}\makebox(0,0)[lt]{\begin{minipage}{0.93116981\unitlength}\centering \end{minipage}}}%
    \put(0,0){\includegraphics[width=\unitlength,page=1]{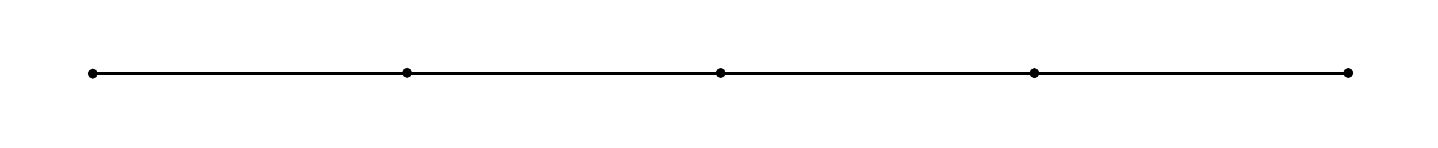}}%
    \put(0.0637779,0.0638342){\color[rgb]{0,0,0.50196078}\makebox(0,0)[t]{\lineheight{1.25}\smash{\begin{tabular}[t]{c}$A$\end{tabular}}}}%
    \put(0.28230374,0.0638342){\color[rgb]{0,0,0.50196078}\makebox(0,0)[t]{\lineheight{1.25}\smash{\begin{tabular}[t]{c}$\langle a_4 \rangle$\end{tabular}}}}%
    \put(0.50033517,0.0638342){\color[rgb]{0,0,0.50196078}\makebox(0,0)[t]{\lineheight{1.25}\smash{\begin{tabular}[t]{c}$\langle a_3 \rangle$\end{tabular}}}}%
    \put(0.71935539,0.0638342){\color[rgb]{0,0,0.50196078}\makebox(0,0)[t]{\lineheight{1.25}\smash{\begin{tabular}[t]{c}$\langle a_2 \rangle$\end{tabular}}}}%
    \put(0.93689247,0.0638342){\color[rgb]{0,0,0.50196078}\makebox(0,0)[t]{\lineheight{1.25}\smash{\begin{tabular}[t]{c}$\langle a_1,b \rangle$\end{tabular}}}}%
    \put(0.83026505,0.02342433){\color[rgb]{0,0.50196078,0}\makebox(0,0)[t]{\lineheight{1.25}\smash{\begin{tabular}[t]{c}$\langle a_1 \rangle$\end{tabular}}}}%
    \put(0.61404769,0.02342433){\color[rgb]{0,0.50196078,0}\makebox(0,0)[t]{\lineheight{1.25}\smash{\begin{tabular}[t]{c}$\langle a_2 \rangle$\end{tabular}}}}%
    \put(0.39783026,0.02342433){\color[rgb]{0,0.50196078,0}\makebox(0,0)[t]{\lineheight{1.25}\smash{\begin{tabular}[t]{c}$\langle a_3 \rangle$\end{tabular}}}}%
    \put(0.17436041,0.02342433){\color[rgb]{0,0.50196078,0}\makebox(0,0)[t]{\lineheight{1.25}\smash{\begin{tabular}[t]{c}$\langle a_4 \rangle$\end{tabular}}}}%
  \end{picture}%
\endgroup%

Let $b_0 := b$ and for $1 \leq i \leq N$ we define $b_i := b_{i-1} a_{i} b^2_{i-1}$. We now apply $N-1$ folds in the opposite direction. The first ``pulls'' $b_1$ across the first edge, the second ``pulls'' $b_2$ across the second edge and so on, so that the $i \textsuperscript{th}$ fold ``pulls'' $b_i$ across the $i \textsuperscript{th}$ edge. 

\noindent%% Creator: Inkscape inkscape 0.92.4, www.inkscape.org
%% PDF/EPS/PS + LaTeX output extension by Johan Engelen, 2010
%% Accompanies image file 'Counterexample3.pdf' (pdf, eps, ps)
%%
%% To include the image in your LaTeX document, write
%%   \input{<filename>.pdf_tex}
%%  instead of
%%   \includegraphics{<filename>.pdf}
%% To scale the image, write
%%   \def\svgwidth{<desired width>}
%%   \input{<filename>.pdf_tex}
%%  instead of
%%   \includegraphics[width=<desired width>]{<filename>.pdf}
%%
%% Images with a different path to the parent latex file can
%% be accessed with the `import' package (which may need to be
%% installed) using
%%   \usepackage{import}
%% in the preamble, and then including the image with
%%   \import{<path to file>}{<filename>.pdf_tex}
%% Alternatively, one can specify
%%   \graphicspath{{<path to file>/}}
%% 
%% For more information, please see info/svg-inkscape on CTAN:
%%   http://tug.ctan.org/tex-archive/info/svg-inkscape
%%
\begingroup%
  \makeatletter%
  \providecommand\color[2][]{%
    \errmessage{(Inkscape) Color is used for the text in Inkscape, but the package 'color.sty' is not loaded}%
    \renewcommand\color[2][]{}%
  }%
  \providecommand\transparent[1]{%
    \errmessage{(Inkscape) Transparency is used (non-zero) for the text in Inkscape, but the package 'transparent.sty' is not loaded}%
    \renewcommand\transparent[1]{}%
  }%
  \providecommand\rotatebox[2]{#2}%
  \newcommand*\fsize{\dimexpr\f@size pt\relax}%
  \newcommand*\lineheight[1]{\fontsize{\fsize}{#1\fsize}\selectfont}%
  \ifx\svgwidth\undefined%
    \setlength{\unitlength}{413.85826772bp}%
    \ifx\svgscale\undefined%
      \relax%
    \else%
      \setlength{\unitlength}{\unitlength * \real{\svgscale}}%
    \fi%
  \else%
    \setlength{\unitlength}{\svgwidth}%
  \fi%
  \global\let\svgwidth\undefined%
  \global\let\svgscale\undefined%
  \makeatother%
  \begin{picture}(1,0.10273973)%
    \lineheight{1}%
    \setlength\tabcolsep{0pt}%
    \put(0.03662987,0.3576232){\color[rgb]{0,0,0}\makebox(0,0)[lt]{\begin{minipage}{0.93116981\unitlength}\centering \end{minipage}}}%
    \put(0,0){\includegraphics[width=\unitlength,page=1]{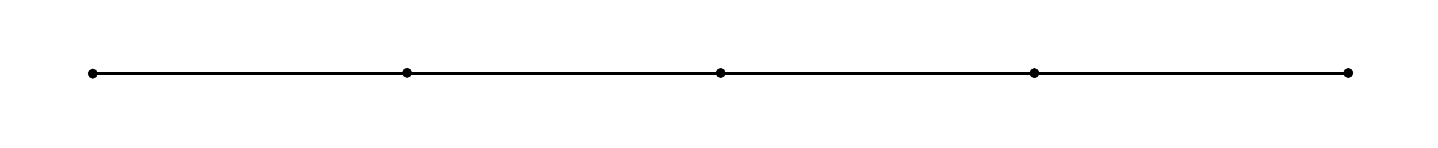}}%
    \put(0.0637779,0.0638342){\color[rgb]{0,0,0.50196078}\makebox(0,0)[t]{\lineheight{1.25}\smash{\begin{tabular}[t]{c}$A$\end{tabular}}}}%
    \put(0.28230374,0.0638342){\color[rgb]{0,0,0.50196078}\makebox(0,0)[t]{\lineheight{1.25}\smash{\begin{tabular}[t]{c}$\langle a_4,b_3 \rangle$\end{tabular}}}}%
    \put(0.50033517,0.0638342){\color[rgb]{0,0,0.50196078}\makebox(0,0)[t]{\lineheight{1.25}\smash{\begin{tabular}[t]{c}$\langle a_3,b_2 \rangle$\end{tabular}}}}%
    \put(0.71935539,0.0638342){\color[rgb]{0,0,0.50196078}\makebox(0,0)[t]{\lineheight{1.25}\smash{\begin{tabular}[t]{c}$\langle a_2,b_1 \rangle$\end{tabular}}}}%
    \put(0.93689247,0.0638342){\color[rgb]{0,0,0.50196078}\makebox(0,0)[t]{\lineheight{1.25}\smash{\begin{tabular}[t]{c}$\langle a_1,b_0 \rangle$\end{tabular}}}}%
    \put(0.83026505,0.02342433){\color[rgb]{0,0.50196078,0}\makebox(0,0)[t]{\lineheight{1.25}\smash{\begin{tabular}[t]{c}$\langle a_1,b_1 \rangle$\end{tabular}}}}%
    \put(0.61404769,0.02342433){\color[rgb]{0,0.50196078,0}\makebox(0,0)[t]{\lineheight{1.25}\smash{\begin{tabular}[t]{c}$\langle a_2,b_2\rangle$\end{tabular}}}}%
    \put(0.39783026,0.02342433){\color[rgb]{0,0.50196078,0}\makebox(0,0)[t]{\lineheight{1.25}\smash{\begin{tabular}[t]{c}$\langle a_3,b_3\rangle$\end{tabular}}}}%
    \put(0.17436041,0.02342433){\color[rgb]{0,0.50196078,0}\makebox(0,0)[t]{\lineheight{1.25}\smash{\begin{tabular}[t]{c}$\langle a_4 \rangle$\end{tabular}}}}%
  \end{picture}%
\endgroup%

It's clear that this is a reduced decomposition. It remains to show that the action on the corresponding Bass-Serre tree is $1$--acylindrical on infinite subgroups. In other words it suffices to show that the stabilisers of any two distinct edges with a common end vertex are finite. Observe that a generic vertex of this decomposition has label $\left\langle a', b' \: | \: a'^{2^r} \right\rangle \cong \mathbb{Z} * \left( \mathbb{Z} / 2^{r}\mathbb{Z} \right)$ with two edges with labels $\left\langle a'^2, b' \right\rangle$ and $\left\langle a', b'a'b'^2 \right\rangle$ respectively.

\noindent%% Creator: Inkscape inkscape 0.92.4, www.inkscape.org
%% PDF/EPS/PS + LaTeX output extension by Johan Engelen, 2010
%% Accompanies image file 'Counterexample4.pdf' (pdf, eps, ps)
%%
%% To include the image in your LaTeX document, write
%%   \input{<filename>.pdf_tex}
%%  instead of
%%   \includegraphics{<filename>.pdf}
%% To scale the image, write
%%   \def\svgwidth{<desired width>}
%%   \input{<filename>.pdf_tex}
%%  instead of
%%   \includegraphics[width=<desired width>]{<filename>.pdf}
%%
%% Images with a different path to the parent latex file can
%% be accessed with the `import' package (which may need to be
%% installed) using
%%   \usepackage{import}
%% in the preamble, and then including the image with
%%   \import{<path to file>}{<filename>.pdf_tex}
%% Alternatively, one can specify
%%   \graphicspath{{<path to file>/}}
%% 
%% For more information, please see info/svg-inkscape on CTAN:
%%   http://tug.ctan.org/tex-archive/info/svg-inkscape
%%
\begingroup%
  \makeatletter%
  \providecommand\color[2][]{%
    \errmessage{(Inkscape) Color is used for the text in Inkscape, but the package 'color.sty' is not loaded}%
    \renewcommand\color[2][]{}%
  }%
  \providecommand\transparent[1]{%
    \errmessage{(Inkscape) Transparency is used (non-zero) for the text in Inkscape, but the package 'transparent.sty' is not loaded}%
    \renewcommand\transparent[1]{}%
  }%
  \providecommand\rotatebox[2]{#2}%
  \newcommand*\fsize{\dimexpr\f@size pt\relax}%
  \newcommand*\lineheight[1]{\fontsize{\fsize}{#1\fsize}\selectfont}%
  \ifx\svgwidth\undefined%
    \setlength{\unitlength}{413.85826772bp}%
    \ifx\svgscale\undefined%
      \relax%
    \else%
      \setlength{\unitlength}{\unitlength * \real{\svgscale}}%
    \fi%
  \else%
    \setlength{\unitlength}{\svgwidth}%
  \fi%
  \global\let\svgwidth\undefined%
  \global\let\svgscale\undefined%
  \makeatother%
  \begin{picture}(1,0.10273973)%
    \lineheight{1}%
    \setlength\tabcolsep{0pt}%
    \put(0.03662987,0.3576232){\color[rgb]{0,0,0}\makebox(0,0)[lt]{\begin{minipage}{0.93116981\unitlength}\centering \end{minipage}}}%
    \put(0,0){\includegraphics[width=\unitlength,page=1]{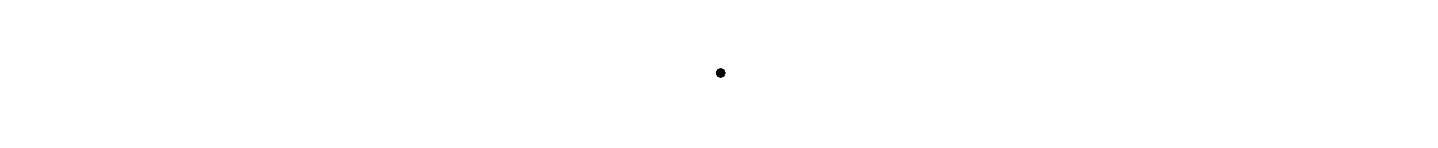}}%
    \put(0.50033517,0.0638342){\color[rgb]{0,0,0.50196078}\makebox(0,0)[t]{\lineheight{1.25}\smash{\begin{tabular}[t]{c}$\langle a',b' \: | \: a'^{2^r} \rangle$\end{tabular}}}}%
    \put(0.63216985,0.02342433){\color[rgb]{0,0.50196078,0}\makebox(0,0)[t]{\lineheight{1.25}\smash{\begin{tabular}[t]{c}$\langle a'^2,b' \rangle$\end{tabular}}}}%
    \put(0.37608366,0.02342433){\color[rgb]{0,0.50196078,0}\makebox(0,0)[t]{\lineheight{1.25}\smash{\begin{tabular}[t]{c}$\langle a',b'a'b'^2 \rangle$\end{tabular}}}}%
    \put(0,0){\includegraphics[width=\unitlength,page=2]{Counterexample4.pdf}}%
  \end{picture}%
\endgroup%

So there are three different pairs of edges we need to consider. 

\begin{proof}[Case 1]
The intersection of $\left\langle a'^2, b' \right\rangle$ and $\left\langle a'^2, b' \right\rangle^g$ for $g \in \left\langle a', b' \right\rangle \setminus \left\langle a'^2, b' \right\rangle$. 

Let $\tilde{w} = \alpha_1 \alpha_2 \cdots \alpha_n$ and $\tilde{w}^h$ be cyclicly reduced words in $\left\lbrace a'^{\pm 2},b'^{\pm 1} \right\rbrace$ for some $h \in \left\langle a', b' \right\rangle$. Observe by cycling letters in $\left\lbrace a'^{\pm 1},b'^{\pm 1} \right\rbrace$ that either $h \in \left\langle a'^2, b' \right\rangle$ or $\alpha_i = a'^{\pm 2}$ for all $i$ and $h$ is an (odd) power of $a'$. It follows that every element in the intersection is conjugate to a power of $a'$. Moreover we see that $w \in \left\langle a'^2, b' \right\rangle \cap \left\langle a'^2, b' \right\rangle^g$ if and only if there are $g_1$, $g_2 \in \left\langle a'^2, b' \right\rangle$ and $n \in \mathbb{Z}$ such that $g = g_1a'g_2$ and $w = (a'^{2n})^{g_2}$. 

If we can show that $g$ can only be expressed in the form $g_1a'g_2$ in an ``essentially unique'' way then it follows that the intersection is cyclic and hence finite as $a'^2$ has finite order. More precisely it suffices to show that whenever $g_1a'g_2 = g'_1a'g'_2$ (where $g_1, g'_1, g_2, g'_2 \in \left\langle a'^2, b' \right\rangle$) then $a'^{g_2} = a'^{g'_2}$. By the rigidity of reduced words in $\left\langle a', b' \: | \: a'^{2^r} \right\rangle$ observe that this equality only happens if $g'_1 = g_1 a'^{-2r}$ and $g'_2 = a'^{2r} g_2$ for some $r \in \mathbb{Z}$. Thus 
\begin{equation*}
a'^{g'_2} = a'^{\left(a'^{2r}g_2\right)} = \left(a'^{a'^{2r}}\right)^{g_2} = a'^{g_2}
\end{equation*}
As required. 
\end{proof}

\begin{proof}[Case 2]
The intersection of $\left\langle a', b'a'b'^2 \right\rangle$ and $\left\langle a', b'a'b'^2 \right\rangle^g$ for $g \in \left\langle a', b' \right\rangle \setminus \left\langle a', b'a'b'^2 \right\rangle$. \\

Let $\tilde{w} = \alpha_1 \alpha_2 \cdots \alpha_n$ and $\tilde{w}^h$ be cyclicly reduced words in $\left\lbrace a'^{\pm 1}, (b'a'b'^2)^{\pm 1} \right\rbrace$ for some $h \in \left\langle a', b' \right\rangle$. Observe by cycling letters in $\left\lbrace a'^{\pm 1},b'^{\pm 1} \right\rbrace$ that we must have $h \in \left\langle a', b'a'b'^2 \right\rangle$. It follows that every element in the intersection trivial unless $g \in  \left\langle a', b'a'b'^2 \right\rangle$.
\end{proof}

\begin{proof}[Case 3]
The intersection of $\left\langle a'^{2}, b' \right\rangle$ and $\left\langle a', b'a'b'^2 \right\rangle^g$ for $g \in \left\langle a', b' \right\rangle$. \\

Let $\tilde{w}_1 = \alpha_1 \alpha_2 \cdots \alpha_n$ be a cyclicly reduced word in $\left\lbrace a'^{\pm 2}, b'^{\pm 1} \right\rbrace$and let $\tilde{w}_2 = \beta_1 \beta_2 \cdots \beta_m$ be a cyclicly reduced word in $\left\lbrace a'^{\pm 1}, (b'a'b'^2)^{\pm 1} \right\rbrace$. Suppose that $\tilde{w}_1$ and $\tilde{w}_2$ conjugate to each other. Observe that $(b'a'b'^2)^{\pm}$ can't be a subword of any cyclic permutation (in $\left\lbrace a'^{\pm 1}, b'^{\pm 1} \right\rbrace$) of $\tilde{w}_1$ and so $\beta_j \neq (b'a'b'^2)^{\pm}$ for any $j$. Hence $\tilde{w}_1$ and $\tilde{w}_2$ are (even) powers of $a'$. Essentially the same argument also shows that $\left\langle a'^{2}, b' \right\rangle \cap \left\langle a', b'a'b'^2 \right\rangle = \left\langle a'^2 \right\rangle$. 

Now the above says that if $w \in \left\langle a'^{2}, b' \right\rangle \cap \left\langle a', b'a'b'^2 \right\rangle^g$ then we must have $w = (a'^{2n})^{h}$ for some $h \in \left\langle a', b' \right\rangle$. Now arguments from case 1 imply that either $h$ or $a'h$ must be in $\left\langle a'^2, b' \right\rangle$ as $w \in \left\langle a'^{2}, b' \right\rangle$. Likewise arguments from case 2 implies that $k^{-1} := hg^{-1} \in \left\langle a', b'a'b'^2 \right\rangle$ as $w^{g^{-1}} \in \left\langle a'^{2}, b' \right\rangle$. So $g = kh$ and WLOG we have $h \in \left\langle a'^2, b' \right\rangle$; as if $a'h \in \left\langle a'^2, b' \right\rangle$ we can just replace $h$ with $a'^{-1}h$ and $k$ with $ka'$. 

If we can show that $g$ can only be expressed in the form $kh$ in an ``essentially unique'' way then it follows that the intersection is cyclic and hence finite as $a'^2$ has finite order. More precisely we wish to show that if $kh = k'h'$ (where $k, k' \in \left\langle a', b'a'b'^2 \right\rangle$ and $h, h' \in \left\langle a'^2, b' \right\rangle$) then $a'^{h} = a'^{h'}$. Since $\left\langle a'^{2}, b' \right\rangle \cap \left\langle a', b'a'b'^2 \right\rangle = \left\langle a'^2 \right\rangle$ and by the rigidity of reduced words in $\left\langle a', b' \: | \: a'^{2^r} \right\rangle$ we see that this only happens if $k' = ka'^{-2r}$ and $h' = a'^{2r}h$ for some $r \in \mathbb{Z}$. Once again we get $a'^{h} = a'^{h'}$ in the same way as in case 1. $\square$
\end{proof}

\end{proof}

Note that a hyperbolic group $G$ cannot satisfy \thref{res:strongCE}. This is because there are only finitely many conjugacy classes of finite subgroups of a hyperbolic group \cite{Brady}; thus there is some bound on the order of finite subgroups. We can then apply the bound for $(k,C)$--acylindrical actions to get a bound here. One may then wonder if Weidmann's conjecture holds for hyperbolic groups; however a slight tweak to our example shows that this isn't true either, even for free groups.

\begin{theorem}\thlabel{res:hyperbolicCE}
For any $N > 0$ there is an action of $F_2$ on a reduced tree which is $1$--acylindrical on non-cyclic subgroups and has $N$ orbits of edges. 
\end{theorem}

\begin{proof}
The construction is mostly the same as \thref{res:strongCE} and so we will only detail the changes. This time we define $A := \left\langle a \right\rangle \cong \mathbb{Z}$ so that $G = \left\langle a,b \right\rangle \cong F_2$. Pick any $N>0$ and define $a_i = a^{2^{i}}$. We now define the tree and see that it satisfies the necessary conditions in the same way as before. $\square$
\end{proof}

In both of these constructions we exploit chains of subgroups with arbitrary length. More precisely we have the chain of subgroups $\left\langle a_0 \right\rangle > \left\langle a_1 \right\rangle > \cdots > \left\langle a_N \right\rangle$ and build the tree in such a way that each group in this chain fixes a vertex which isn't fixed by any of the larger ones. Forcing $\mathcal{P}$ to have finite height and insisting the tree is $\mathcal{P}$-closed ensures that we cannot use these long chains to make arbitrarily complicated decompositions in the same way.

\section{Forests of Influence}\label{sec:influence}

We begin by stating Dunwoody's resolution lemma.

\begin{theorem}[\cite{Dunwoody_fp}]\thlabel{res:dunwoody_resolution}
Suppose that $G$ is an (almost) finitely presented group. Then there is some $C(G) \in \mathbb{N}$ with the following properties. Whenever $G$ acts on a minimal tree $T$ there is some minimal tree $T'$ with at most $C(G)$ orbits of edges and a combinatorial map $\Psi: T' \rightarrow T$ where no edge gets mapped to a point. 
\end{theorem}

We'll now give an \emph{extremely} rough outline of the core ideas of the argument. Suppose that $G$ acts on a minimal tree $T$ which is $k$--acylindrical on groups larger than $\mathcal{P}$. Use Dunwoody's resolution lemma to obtain a tree $T'$ which has a bound on the number of edges and a map $\Psi: T' \rightarrow T$. If some edge of $T'$ (before subdividing) has a stabiliser larger than $\mathcal{P}$ then its image in $T$ cannot have more than $k$ edges because of the acylindrical condition. Thus we can collapse this edge in $T'$ and only collapse at most $k$ edges of $T$. 

Now subdivide $T'$ to make $\Psi$ simplicial, but note that the initial vertices are `more important' in the sense that every vertex stabiliser is contained in one of these. So we can build a collection of disjoint subtrees for $T'$ by starting with this set of initial vertices and then iteratively expanding to include vertices whose stabiliser is contained in the stabiliser of the corresponding initial vertex. 

Now we subdivide $\Psi$ into folds using Stallings' folding theorem (\thref{res:stallings_folding_thm}) and apply the first fold. If every vertex stabilizer is still contained in a stabilizer for one of the initial vertices then we have still have a collection of subtrees with the same properties as before. Otherwise some vertex stabiliser isn't contained in one of the initial ones. This only happens if two of our subtrees gets folded together in some way which is unavoidable. We then add this vertex to our set of ``initial'' ones and then rebuild our collection of subtrees with the same properties as before. However we will see that the intersections of the stabilisers between one of the original initial vertices and this ``new initial vertex'' is strictly larger than the intersection of the original initial vertices. (See Figure~\ref{fig:FoIExample} for an example or \thref{res:forests_enjoy_folds} for a more precise statement.) If $\mathcal{P}$ has finite height this means that this can only happen boundedly often before one of these intersections is larger than $\mathcal{P}$ and so can collapse down a path of length at most $k$. So either we can keep doing this until we are left with a single point or we get a set of ``initial'' vertices for $T$. In the latter case if $T$ is $\mathcal{P}$--partially reduced we can find a bound for the number of edges using our set ``initial'' vertices. (See \thref{res:reduced_key_lemma}.)  

% Now subdivide $T'$ to make $\Psi$ simplicial, but note that the initial vertices are `more important' in the sense that every vertex stabiliser is contained in one of these. Now begin folding as in Stallings folding theorem. Whenever some vertex stabiliser isn't contained in one of of the `important' ones we add the corresponding vertex to this list of `important' vertices. When this happens we will see that the intersections of the stabilisers between two `consecutive important vertices' must contain larger subgroups than before. (See diagram below for an example or \thref{res:forests_enjoy_folds} for a more precise statement.) If $\mathcal{P}$ has finite height this means that this can only happen boundedly often before one of these intersections is larger than $\mathcal{P}$ and so can collapse down a path . So either we can keep doing this until we are left with a single point or we get a set of `important' vertices for $T$. In the latter case if $T$ is $\mathcal{P}$--partially reduced we can find a bound for the number of edges using our set `important' vertices. (See \thref{res:reduced_key_lemma}.) 

\begin{figure}[h!]
\centering
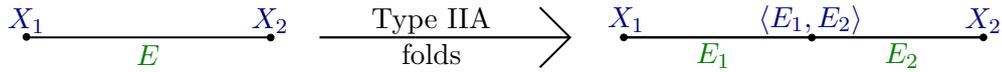
\caption{After applying a series of type II folds to an edge (with subdivisions) there may a vertex whose stabiliser isn't contained in the stabiliser of either of the initial vertices. If this happens we see that the intersection of stabilisers between this vertex and either of the initial vertices must contain the original edge group as a \emph{proper} subgroup.}
\label{fig:FoIExample}
\end{figure}

In order to make the above precise we introduce the following notions.

\begin{definition}
Suppose $G$ acts on a tree $T$. We call a subset of vertices $S$ a set of \emph{seed vertices} if it's $G$--invariant and for every vertex $v$ (with non-trivial stabiliser) there is some $u \in S$ with $\stab v \leqslant \stab u$. In particular if the action on $T$ is free we also allow the empty set to be a set of seed vertices, otherwise $S$ is necessarily non-empty.
\end{definition}

\begin{definition}
Suppose $G$ acts on a tree $T$. A $G$--invariant subgraph $\Gamma \subseteq T$ is a \emph{forest of influence} if the following conditions hold.
\begin{itemize}
\item $\Gamma$ deformation retracts to a non-empty set of seed vertices $S$, equivalently every component of $\Gamma$ contains exactly one member of $S$. We say that $\Gamma$ is \emph{grown} from $S$.
\item If vertices $u$ and $v$ are in the same connected component of $\Gamma$ with $u \in S$ then $\stab v \leqslant \stab u$. We call such a component the \emph{tree of influence} of $u$, say that $v$ is \emph{influenced} by $u$ and call the reduced edge path from $v$ to $u$ the \emph{branch} of $v$.
\item Every vertex of $T$ is contained in $\Gamma$.
\end{itemize}
\end{definition}

\begin{remark}\thlabel{res:branch_edge}
The branch of any vertex $v$ is stabilised by $\stab v$. As such the first edge on the branch of $v$ must have the same stabiliser as $v$ as any edge cannot be fixed by more than either of its endpoints. 
\end{remark}

\begin{definition}
Suppose $G$ acts on a tree $T$ and that $\Gamma \subseteq T$ is a forest of influence. We call the edges of $T \setminus \Gamma$ the \emph{connecting edges} of $\Gamma$. The \emph{connecting groups} are the conjugacy classes of (a set of representatives for) the connecting edges, counted with multiplicity.
\end{definition}

In general there is not a distinguished choice for a forest of influence. However the following proposition says there is something canonical lurking  underneath. This will allow us to move between different choices with minimal difficulties. 

\begin{proposition}\thlabel{res:modify_trees}
Suppose that $T$ has finitely many orbits of vertices. Suppose also that $\Gamma_1$ and $\Gamma_2$ are forests of influence which are both grown from the same set of seed vertices $S$. Then $\Gamma_1$ and $\Gamma_2$ have the same connecting groups. In other words the connecting groups are determined by $S$.
\end{proposition}

Before proving this we'll first we'll define an \emph{elementary transformation} of a forest of influence. Take a forest of influence $\Gamma$ and pick a vertex $v \in \Gamma \setminus S$. Suppose that $v$ is contained in the tree of influence of $u$ and let $e_1$ be the first edge on the branch of $v$. Observe that $\stab e_1 = \stab v$ and pick some connecting edge $e_2$ with endpoint $v$ and with $\stab e_2 = \stab v$. We now define $\Gamma' := (\Gamma \setminus G \left\lbrace e_1 \right\rbrace) \cup G \left\lbrace e_2 \right\rbrace$. In other words we replace the orbit of $e_1$ in $\Gamma$ with the orbit of $e_2$ in $\Gamma'$. (See Figure~\ref{fig:EleTrnsfm}.) Since $\stab e_1 = \stab e_2 = \stab v$ we see that $\Gamma'$ is also a forest of influence grown from $S$ and that both $\Gamma$ and $\Gamma'$ have the same connecting groups. 

\begin{figure}[h!]
\centering
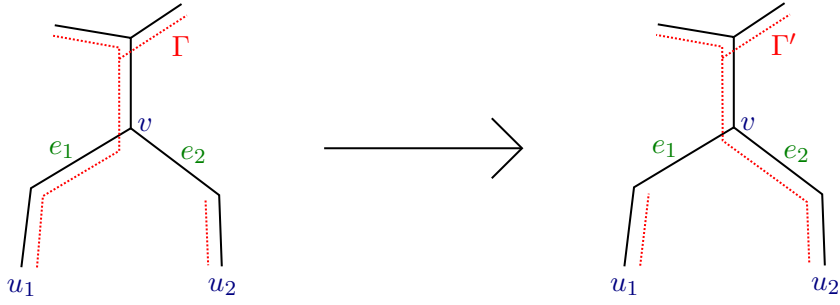
\caption{An example of an elementary transformation. The edge $e_1$ is removed and replaced with $e_2$. For $\Gamma'$ to be a forest of influence we must have $\stab e_1 = \stab e_2 = \stab v$.}
\label{fig:EleTrnsfm}
\end{figure}

\thref{res:modify_trees} is now an immediate consequence of the following.

\begin{lemma}\thlabel{res:ele_trnsfm_chain}
Suppose that $T$ has finitely many orbits of vertices and that $\Gamma_1$ and $\Gamma_2$ are forests of influence which are both grown from the same set of seed vertices $S$. Then we can apply a finite series of elementary transformations to $\Gamma_1$ to obtain $\Gamma_2$.
\end{lemma}

\begin{proof}
Let $d(\Gamma_1, \Gamma_2)$ be the number of (orbits of) vertices which are in trees of influence of different seed vertices in $\Gamma_1$ and $\Gamma_2$. If $d(\Gamma_1, \Gamma_2) = 0$ then $\Gamma_1 = \Gamma_2$ and there is nothing to show. 

If $d(\Gamma_1, \Gamma_2) > 0$ pick a vertex $v$ which is in the tree of influence of $u_1$ in $\Gamma_1$ and of $u_2 \neq u_1$ in $\Gamma_2$. Let $e'_2$ be the final edge in the branch of $v$ (in $\Gamma_2$) which is not contained in $\Gamma_1$ and so is a connecting edge of $\Gamma_1$. Let $v'$ be the endpoint of $e'_2$ which is not in the tree of influence of $u_2$ in $\Gamma_1$. Observe that $v'$ is in the tree of influence of $u_2$ in $\Gamma_2$ as $v$ is. Suppose that $v'$ is in the tree of influence of $u'_1$ in $\Gamma_1$ and let $e'_1$ be the first edge on the branch of $v'$ (in $\Gamma_1$). Since $\stab v' = \stab e'_1 = \stab e'_2$ we can apply an elementary transformation to $\Gamma_1$ by removing the orbit of $e'_1$ and adding the orbit of $e'_2$ to get $\Gamma'_1$. (See Figure~\ref{fig:ModifyForest}.) 

\begin{figure}[h!]
\centering
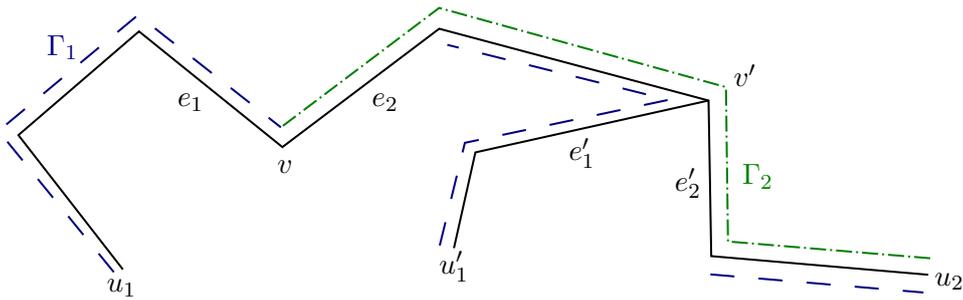
\caption{An example of a situation where $d(\Gamma_1, \Gamma_2) > 0$. By replacing $e'_1$ with $e'_2$ in $\Gamma_1$ we can make it ``more similar'' to $\Gamma_2$. This idea of how similar two forests of influence are is formalised by the metric $d$.}
\label{fig:ModifyForest}
\end{figure}

If we can show that $d(\Gamma'_1, \Gamma_2) < d(\Gamma_1, \Gamma_2)$ then we are done by induction. Observe that $v'$ is in the tree of influence of $u_2$ in $\Gamma'_1$ and $\Gamma_2$ but not in $\Gamma_1$. Thus we just need to show that any vertex which is influenced by the same seed vertex in $\Gamma_1$ and $\Gamma_2$ is also influenced by the same one in $\Gamma'_1$. This holds because the only vertices whose influencing vertex changed under the elementary transformation were those in (the orbit of) the tree of influence of $u'_1$ in $\Gamma_1$ at and beyond $v'$. These can't be influenced by $u'_1$ in $\Gamma_2$ as $v'$ is influenced by $u_2$ in $\Gamma_2$ and so the tree of influence of $u'_1$ in $\Gamma_2$ cannot contain them. $\square$
\end{proof}

Recall the definition of the $\mathcal{P}$--weight of a subgroup $K \leqslant G$ from \thref{def:p_weight} as $W_{\mathcal{P},K} \leq 2^M$ where $M$ is the length of the longest chain of groups in $\mathcal{P}$ which contain $K$. From this definition we note that the following properties are all obvious.

\begin{proposition}\thlabel{res:weight_basics}
Let $\mathcal{P}$ be a conjugation invariant set of subgroups of $G$.
\begin{enumerate}[label=(\alph*)]
\item \label{pt:max_weight} If $\mathcal{P}$ has height $M$ then $W_{\mathcal{P},K} \leq 2^M$ for any $K \leqslant G$.
\item \label{pt:minimal_weight} $K \leqslant G$ has $\mathcal{P}$--weight $1$ if and only if it's larger than $\mathcal{P}$.
\item \label{pt:extention_weight} If $H \in \mathcal{P}$ and $H < K \leq G$ then $W_{\mathcal{P},H} \leq \frac{1}{2} W_{\mathcal{P},K}$.
\end{enumerate}
\end{proposition}

We will now extend our definition of $\mathcal{P}$--weight to sets of seed vertices. \thref{res:modify_trees} ensures this is well defined.

\begin{definition}
If $G$ acts on a tree $T$ and $S$ is a non-empty set of seed vertices for $T$ then we define its \emph{$\mathcal{P}$--weight} $W_{\mathcal{P},S}$ to be the sum of the $\mathcal{P}$--weights of the corresponding connecting groups (and $\infty$ if any of the connecting groups have infinite $\mathcal{P}$--weight). If $S$ is empty then we instead define $W_{\mathcal{P},S} := (\beta_1(T/G) - 1)W_{\mathcal{P},1}$.
\end{definition}

\begin{remark}
The case of a free action is special because the stabiliser of each vertex is trivial. As there are no ``interesting'' stabilisers we aren't really missing anything by just forgoing seed vertices entirely. If the action is free and $S$ is non empty then we see that
\begin{equation*}
W_{\mathcal{P},S} := (\beta_1(T/G) + \left| S/G \right| - 1)W_{\mathcal{P},1}.
\end{equation*}
This justifies the definition of $W_{\mathcal{P},S}$ for empty $S$ by setting $\left| S/G \right| = 0$ in the above equation. On a more practical level we allow the empty set to be a set of seed vertices for a free action to prevent an otherwise guaranteed drop in $\mathcal{P}$--weight if a fold causes a free action to become non-free. (See \thref{res:free_action_folds}.) 
\end{remark}

With this in hand we are ready to state the key lemma. From this \thref{res:main_simple} will follow quickly.

\begin{lemma}\thlabel{res:key_lemma_simple}
Suppose $G$ is a non-cyclic countable group. Let $\mathcal{P}$ be a conjugation invariant set of subgroups of $G$ which is closed under taking subgroups. Let $G$ act on a tree $T$ where this action is both $\mathcal{P}$--partially-reduced and $k$--acylindrical on a subgroups larger than $\mathcal{P}$. Let $G$ act on another tree $T'$ and suppose that there is a $G$-equivarient combinatorial map $\Psi: T' \rightarrow T$. Suppose also that $T'$ has a set of seed vertices $S$ with finite $\mathcal{P}$--weight $W_{\mathcal{P},S}$. Then $T/G$ has at most $\left( \frac{2k+1}{2} \right) W_{\mathcal{P},S}$ edges.
\end{lemma}

The remainder of this section as well as the entirety of Section~\ref{sec:reduced} will be dedicated to providing the necessary tools to prove this. 

Recall that our plan involves decomposing $\Psi$ into folds. The following says that we can recursively find a nice set of seed vertices for each intermediate step.

\begin{lemma}\thlabel{res:forests_enjoy_folds}
Suppose that $\alpha: R \rightarrow \tilde{R}$. Suppose that $S$ is a non-empty set of seed vertices for $R$ where all of the connecting groups are in $\mathcal{P}$. Then there is a set of seed vertices $\tilde{S}$ for $\tilde{R}$ with $\alpha(S) \subseteq \tilde{S}$ and $W_{\mathcal{P},\tilde{S}} \leq W_{\mathcal{P},S}$. Moreover if $W_{\mathcal{P},\tilde{S}} = W_{\mathcal{P},S}$ then $\alpha|_{S}$ is injective. 
\end{lemma}

\begin{proof}
Suppose that $\alpha$ folds together the edges $e_1 = [x,y_1]$ and $e_2 = [x,y_2]$. Suppose $\alpha(e_1) = \alpha(e_2) = e'$ and $\alpha(y_1) = \alpha(y_2) = y'$. Let $y_i$ be in the tree of influence of $u_i$ and if $y_i \neq u_i$ we also let $f_i$ be the first edge in the branch of $y_i$. Throughout we will assume that $y_i \neq u_i$ and so $f_i$ exists. The cases where $y_i = u_i$ turn out to be essentially the same except the lack of $f_i$ sometimes causes $W_{\mathcal{P}, \tilde{S}}$ to be smaller. We will split into cases depending on if there is a forest of influence containing $e_1$ and/or $e_2$. 

\begin{proof}[Case 1]
There is a forest of influence $\Gamma$ containing both $e_1$ and $e_2$. \\

\begin{figure}[h!]
\centering
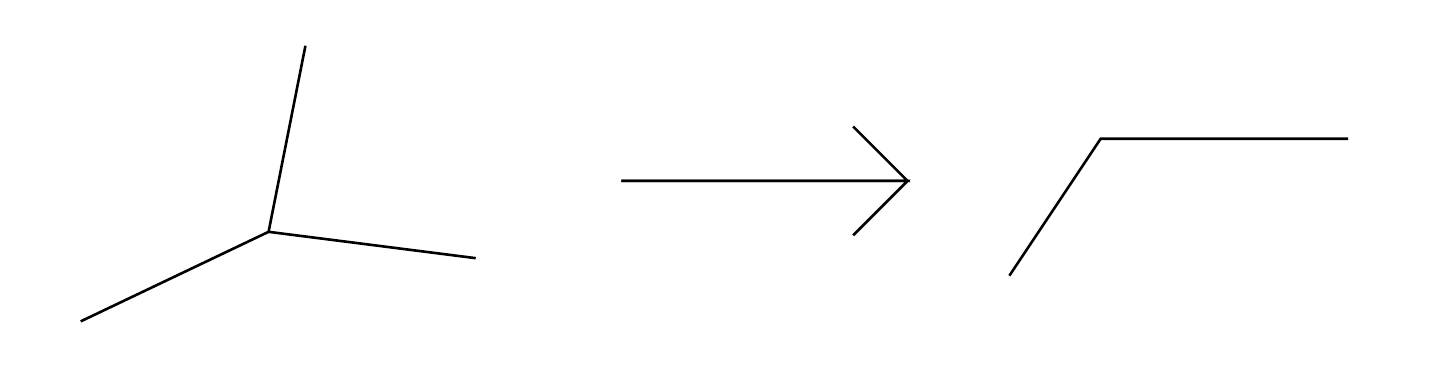 \\
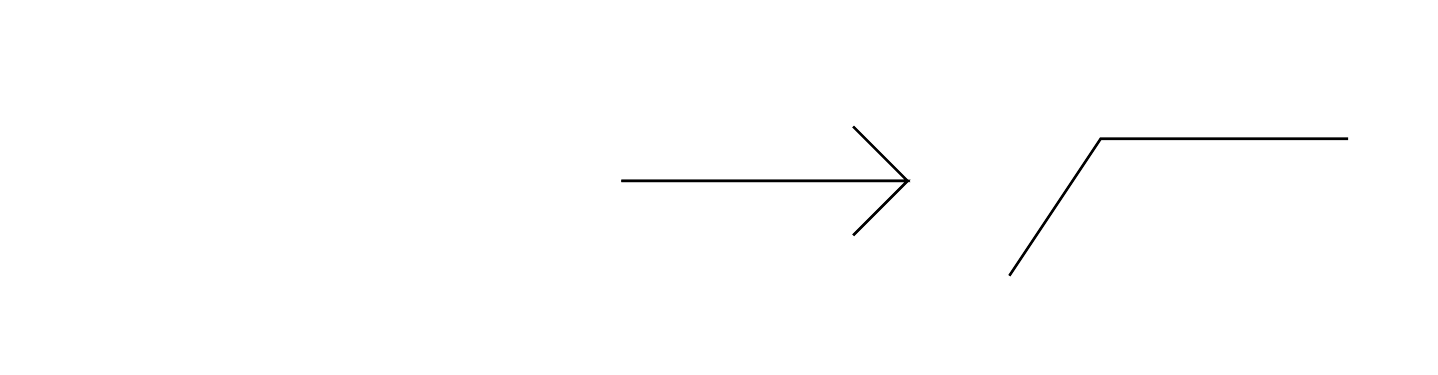
\caption{The two pictures that can arise when $\Gamma$ contains both $e_1$ and $e_2$. In either case $\stab u_1$ contains $\stab x$, $\stab y_1$ and $\stab y_2$ and hence also contains $\stab y'$.}
\end{figure}

The fold cannot be of type III as otherwise $y_1$ and $y_2 = hy_1$ would be in the same tree of influence. So $\alpha(\Gamma)$ is a forest of influence for $\tilde{R}$ which is grown from $\alpha(S)$. The connecting edges of $\Gamma$ are untouched by $\alpha$ and so $S$ and $\alpha(S)$ have the same connecting groups. 
\end{proof} 

\begin{proof}[Case 2]
There is no forest of influence containing either $e_1$ or $e_2$. \\

\begin{figure}[h!]
\centering
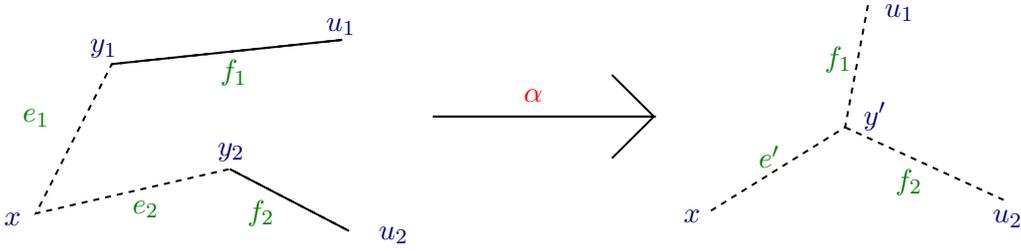
\caption{The picture that arises when $\Gamma$ cannot possibly contain either $e_1$ or $e_2$. Observe that $\stab e_i < \stab u_i$. After applying the fold we remove the edges $f_i$ (if they exist) from the forest of influence. Separate arguments depending on the type of fold are now required to show that this doesn't increase the $\mathcal{P}$-weight.}
\end{figure}

Pick any forest of influence $\Gamma$. Observe that $\tilde{S} = \alpha(S) \cup G \left\lbrace e' \right\rbrace$ is a set of seed vertices which grows into $\alpha(\Gamma) \setminus G \left\lbrace f_1, f_2 \right\rbrace$. Since $e_i$ is not contained in a forest of influence we must have $e_i < f_i$ as otherwise we could apply an elementary transformation to get it into one. So $W_{f_i,\mathcal{P}} \leq \frac{1}{2}W_{e_i,\mathcal{P}}$ by \thref{res:weight_basics}~\ref{pt:extention_weight}. 

If $\alpha$ is a fold of type I then $e_1$, $e_2$, $f_1$ and $f_2$ are pairwise inequivalent. Moreover $e'$ contains the image of both $e_1$ and $e_2$, so $W_{e',\mathcal{P}} \leq \max(W_{e_1,\mathcal{P}}, W_{e_2,\mathcal{P}})$ by \thref{res:weight_basics}~\ref{pt:extention_weight} and hence
\begin{eqnarray*}
W_{S,\mathcal{P}} - W_{\tilde{S},\mathcal{P}}   &  =   &   W_{e_1,\mathcal{P}} + W_{e_2,\mathcal{P}} - W_{f_1,\mathcal{P}} - W_{f_2,\mathcal{P}} - W_{e',\mathcal{P}} \\
                                                & \geq &   \frac{1}{2}W_{e_1,\mathcal{P}} + \frac{1}{2}W_{e_2,\mathcal{P}} - \max(W_{e_1,\mathcal{P}}, W_{e_2,\mathcal{P}}) \\
                                                & \geq &   0
\end{eqnarray*}

If $\alpha$ is a fold of type II then $e_1$ is equivalent to $e_2$ and $f_1$ is equivalent to $f_2$. Additionally $\stab e' > \stab e_1$. So by \thref{res:weight_basics}~\ref{pt:extention_weight} we have 
\begin{eqnarray*}
W_{S,\mathcal{P}} - W_{\tilde{S},\mathcal{P}}   &  =   &   W_{e_1,\mathcal{P}} - W_{f_1,\mathcal{P}} - W_{e',\mathcal{P}} \\
                                                & \geq &   W_{e_1,\mathcal{P}} - \frac{1}{2}W_{e_1,\mathcal{P}} - \frac{1}{2}W_{e_1,\mathcal{P}} \\
                                                &  =   &   0
\end{eqnarray*}

Now assume $\alpha$ is a fold of type III. We see that $f_1$ and $f_2$ are equivalent, while $e_1$ and $e_2$ are inequivalent. Thus 
\begin{eqnarray*}
W_{S,\mathcal{P}} - W_{\tilde{S},\mathcal{P}}   &  =   &   W_{e_1,\mathcal{P}} + W_{e_2,\mathcal{P}} - W_{f_1,\mathcal{P}} - W_{e',\mathcal{P}} \\
                                                & \geq &   W_{e_1,\mathcal{P}} - W_{f_1,\mathcal{P}} \\
                                                & \geq &   W_{e_1,\mathcal{P}} - \frac{1}{2}W_{e_1,\mathcal{P}} \\
                                                &  >   &   0
\end{eqnarray*}
\end{proof}

\begin{proof}[Case 3]
There is a forest of influence $\Gamma$ containing $e_1$ but not $e_2$; also there isn't one which contains both of them. \\

Note that $\alpha$ cannot be a fold of type II (as then the $e_i$ are equivalent) or type IIIB (so both $e_i$ are always connecting edges). We will split into four subcases; corresponding to combinations whether or not $e_1$ is equal to $f_1$ and whether or not $\stab y_1$ is a subgroup of $\stab y_2$. 

\begin{proof}[Case 3ai]
We have $e_1 = f_1$ and $\stab y_1 \leqslant \stab y_2$. \\

\begin{figure}[h!]
\centering
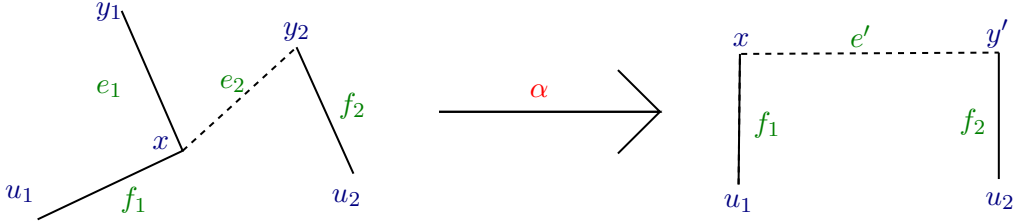
\caption{The picture that arises when $\Gamma$ contains $e_1$ but not $e_2$, the branch of $y_1$ contains $x$ and $\stab y_2$ contains $\stab y_1$. Here we take the image of each connecting edge to still be a connecting edge.}
\end{figure}

If the fold is of type I then observe that $\tilde{S} := \alpha(S)$ is a set of seed vertices for $\tilde{R}$. Observe that the image of the the connecting edges of $\Gamma$ are the connecting edges of a forest of influence grown from $\tilde{S}$. If instead the fold is of type IIIA then $\tilde{S} := \alpha(S) \cup G \left\lbrace y' \right\rbrace$ is a set of seed vertices. In this case we have a forest of influence $\tilde{\Gamma} := \alpha(\Gamma) \setminus G \left\lbrace f_1 \right\rbrace$. The connecting edges of $\tilde{\Gamma}$ are the image of the connecting edges of $\Gamma$ with the orbit of $e_2$ removed and the orbit of $\alpha(f_1)$ added. Observe that $\stab e_2 \leqslant \stab f_1$. Hence the $\mathcal{P}$--weight can't increase in either case.
\end{proof}

\begin{proof}[Case 3aii]
We have $e_1 = f_1$ and $\stab y_1$ is not contained in $\stab y_2$. \\

\begin{figure}[h!]
\centering
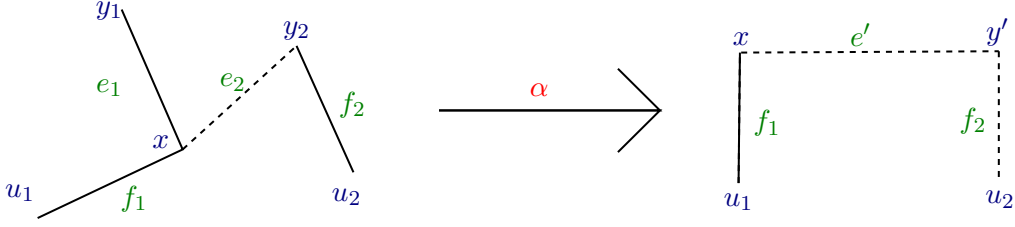
\caption{The picture that arises when $\Gamma$ contains $e_1$ but not $e_2$, the branch of $y_1$ contains $x$ and $\stab y_2$ doesn't contain $\stab y_1$. Here we take the new connecting edges to be the image of the old connecting edges together with $f_2$ (if it exists).}
\end{figure}

If the fold if type IIIA then proceed as in case 3ai. Otherwise observe $\tilde{S} = \alpha(S) \cup G \left\lbrace y' \right\rbrace$ is a set of seed vertices for $\tilde{R}$ and that $\tilde{\Gamma} = \alpha(\Gamma) \setminus G \left\lbrace e', f_2 \right\rbrace$ is a forest of influence grown from $\tilde{S}$. Since $\stab y_1$ is not contained in $\stab y_2$ and $\stab y_1 = \stab e_1$ we have $\stab e_2 < \stab e'$. We also have $\stab e_2 < \stab f_2$ because otherwise we could apply an elementary transformation to $\Gamma$ to get a new forest of influence which is in case 1. Hence by \thref{res:weight_basics}~\ref{pt:extention_weight}
\begin{eqnarray*}
W_{S,\mathcal{P}} - W_{\tilde{S},\mathcal{P}}   &  =   &   W_{e_2,\mathcal{P}} - W_{f_2,\mathcal{P}} - W_{e',\mathcal{P}} \\
                                                & \geq &   W_{e_2,\mathcal{P}} - \frac{1}{2}W_{e_2,\mathcal{P}} - \frac{1}{2}W_{e_2,\mathcal{P}} \\
                                                &  =   &   0 
\end{eqnarray*}
\end{proof}

\begin{proof}[Case 3bi]
We have $e_1 \neq f_1$ and $\stab y_1 \leqslant \stab y_2$. \\

\begin{figure}[h!]
\centering
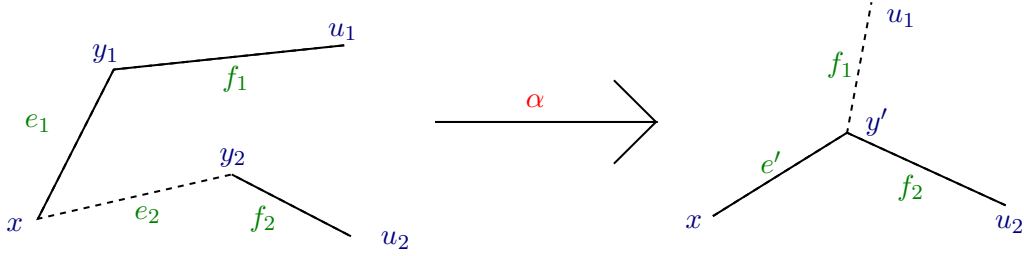
\caption{The picture that arises when $\Gamma$ contains $e_1$ but not $e_2$, the branch of $y_1$ doesn't contain $x$ and $\stab y_2$ contains $\stab y_1$.  Here we take the new forest of influence to be the image of the old one with $f_1$ removed (if it exists).}
\end{figure}

If the fold is of type I then $\tilde{S} := \alpha(S)$ is a set of seed vertices for $\tilde{R}$. If instead the fold is of type IIIA then $\tilde{S} := \alpha(S) \cup G \left\lbrace y' \right\rbrace$ is a set of seed vertices instead. In either case observe that $\alpha(\Gamma) \setminus G \left\lbrace f_1 \right\rbrace$ is a forest of influence grown from $\tilde{S}$. (Note that if the fold is of type IIIA then $f_1$ and $f_2$ are in a common orbit, so $f_2$ also becomes a connecting edge.) Since $\stab e_2 \leqslant \stab x \leqslant \stab f_1$ we have $W_{\tilde{S},C} \leq W_{S,C}$.
\end{proof}

\begin{proof}[Case 3bii]
We have $e_1 \neq f_1$ and $\stab y_1$ is not contained in $\stab y_2$. \\

\begin{figure}[h!]
\centering
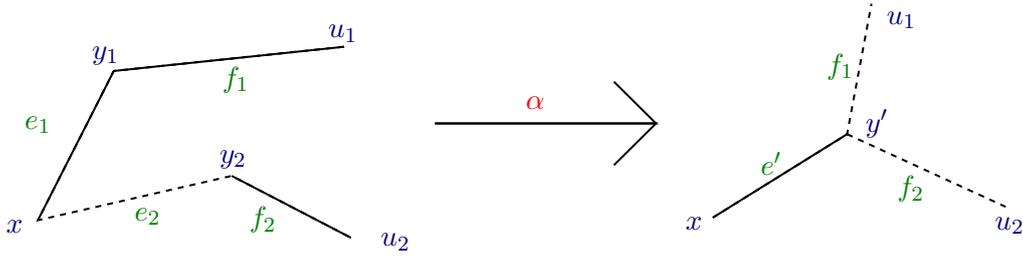
\caption{The picture that arises when $\Gamma$ contains $e_1$ but not $e_2$, the branch of $y_1$ doesn't contain $x$ and $\stab y_1$ doesn't contain $\stab y_2$. Here we take the new forest of influence to be the image of the old one with both $f_1$ and $f_2$ removed (if they exist).}
\end{figure}

If the fold is of type IIIA then proceed as in case 3bi. Otherwise observe that $\tilde{S} = \alpha(S) \cup G \left\lbrace y' \right\rbrace$ is a set of seed vertices which grows into a forest of influence $\tilde{\Gamma} = \alpha(\Gamma) \setminus G \left\lbrace f_1, f_2 \right\rbrace$. If $\stab e_2 = \stab f_2$ then we could apply an elementary transformation to get both $e_1$ and $e_2$ in the same forest of influence and so we are in case 1; hence $\stab f_2 < \stab e_2$. Also since $\stab y_1$ is not contained in $\stab y_2$  and $\stab x \leqslant \stab y_1$ we have 
\begin{eqnarray*}
\stab e_2 & = & \stab y_1 \cap \stab y_2 \\
          & < & \stab y_1 \\
          & = & \stab f_1.
\end{eqnarray*}
Hence by \thref{res:weight_basics}~\ref{pt:extention_weight} 
\begin{eqnarray*}
W_{S,\mathcal{P}} - W_{\tilde{S},\mathcal{P}}   &  =   &   W_{e_2,\mathcal{P}} - W_{f_1,\mathcal{P}} - W_{f_2,\mathcal{P}} \\
                                                & \geq &   W_{e_2,\mathcal{P}} - \frac{1}{2}W_{e_2,\mathcal{P}} - \frac{1}{2}W_{e_2,\mathcal{P}} \\
                                                &  =   &   0            \:\:\:\:\:\:\:\:\:\:\:\:\:\:\:\:\:\:\:\:\:\:\:\:\:\:\:\:\:\:\:\:\:\:\:\:\:\:\:\:\:\:\:\:\:\:\:\:\:\:\:\:\:\:\:\:\:\:\:\:\:\:\:\:\:\:              \square
\end{eqnarray*}
\end{proof}
\end{proof}
\end{proof}

Recall that a free action is a special case as we allow the set of seed vertices to be empty. Thus we must deal with this case separately.

\begin{lemma}\thlabel{res:free_action_folds}
Suppose that $\alpha: R \rightarrow \tilde{R}$ is a fold and $S$ is a set of seed vertices for $R$. Suppose also that $\mathcal{P}$ has finite height. If $G$ acts freely on $R$ but not on $\tilde{R}$ then there is a set of seed vertices $\tilde{S}$ for $\tilde{R}$ such that $W_{\mathcal{P},\tilde{S}} \leq W_{\mathcal{P},S}$. 
\end{lemma}

\begin{proof}
Recall from \thref{def:p_weight} that we must have $W_{\mathcal{P},S} \geq (\beta_1(R/G) - 1)W_{\mathcal{P},1}$. As the action on $\tilde{R}$ is non-free and $\alpha$ is a fold it follows that $\alpha$ is a fold of type III. Moreover all the edge stabilisers of $\tilde{R}$ are trivial and there is a single vertex $u$ (up to equivalence) with a non-trivial stabiliser. We define $\tilde{S} := G \left\lbrace u \right\rbrace$. Since all the connecting groups of $\tilde{S}$ are trivial we get that $W_{\mathcal{P},\tilde{S}} = \beta_1(\tilde{R}/G)W_{\mathcal{P},1}$. Since $\alpha$ is a fold of type III we have $\beta_1(\tilde{R}/G) = \beta_1(R/G) - 1$ and so the result follows. $\square$
\end{proof}

Now suppose that we have a map $\Psi: T \rightarrow T'$ where $T$ has a set of seed vertices $S$ and the action on $T'$ is $k$-acylindrical on groups larger than $\mathcal{P}$. If a connecting group in $S$ is larger than $\mathcal{P}$ then there are seed vertices, say $u_1$ and $u_2$, whose images in $T'$ are separated by distance at most $k$. Since we have control of the length of the path between these images we wish to collapse it to avoid unnecessary extra counting. However in general the image of the path between $u_1$ and $u_2$ need not lie in the path between $\Psi(u_1)$ and $\Psi(u_2)$. For our core argument to work we require this containment and the following says we can do this with some extra folds.

\begin{lemma}\thlabel{res:collapse_connecting_edge}
Let $\Psi: T \rightarrow T'$ be a simplical map where the action on $T'$ is $k$-acylindrical on groups larger than $\mathcal{P}$. Let $S$ be a set of seed vertices for $T$ where at least one of the connecting groups are larger than $\mathcal{P}$. Then there are $\overline{T}$, $\overline{T'}$ and a simplical $\overline{\Psi}: \overline{T} \rightarrow \overline{T'}$ such that the action on $\overline{T'}$ is $k$-acylindrical on groups larger than $\mathcal{P}$, there's a set of seed vertices $\overline{S}$ for $\overline{T}$ with $W_{\mathcal{P},\overline{S}} \leq W_{\mathcal{P},S'} - 1$ and $T'$ has at most $k$ more edges than $\overline{T'}$.
\end{lemma}

\begin{proof}
Now suppose that $S$ has a connecting edge $e$ of $\mathcal{P}$--weight $1$ (in some forest of influence $\Gamma$). Suppose this edge connects the trees of influence of $u_1$ and $u_2$. Let $\gamma$ be the reduced edge path between $u_1$ and $u_2$. Recall statement $(\star)$ from the proof of Stallings folding theorem (\thref{res:stallings_folding_thm}) and apply it to $\gamma$. We get a composition of folds $\rho: T \rightarrow \tilde{T}$ so that the induced map $\tilde{\Psi}: \tilde{T} \rightarrow T'$ is locally injective on $\rho(\gamma)$. The intersection of the stabilisers for $u_1$ and $u_2$ is larger than $\mathcal{P}$ since $e$ has $\mathcal{P}$--weight $1$. So since the action on $T$ is $k$--acylindrical on groups larger than $\mathcal{P}$ the distance between the $\rho(u_i)$ is at most $k$. Hence we can collapse at most $k$ edges of $T'$ to get a new tree $\overline{T'}$ and an induced simplicial map $\overline{\Psi}: \overline{T} \rightarrow \overline{T'}$. Observe that $\overline{T}$ has a set of seed vertices $\overline{S}$, the image of $S$, with $W_{\mathcal{P},\overline{S}} \leq W_{\mathcal{P},S} - 1$ as the image of the connecting edges of $\Gamma$ except the orbit of $e$ is a set of connecting edges for $\overline{S}$.
\end{proof}

\section{Building partially reduced trees}\label{sec:reduced}

It remains to bound the number of edges of a $k$--acylindrical action on a tree given a set of seed vertices.

\begin{lemma}\thlabel{res:reduced_key_lemma}
Let $\mathcal{P}$ be a class of subgroups for a group $G$ which is closed under conjugation. Suppose $G$ acts on a tree $T$ and that this action is both partially-reduced on $\mathcal{P}$ and $k$--acylindrical on groups larger than $\mathcal{P}$. Suppose that $S$ is a non-empty set of seed vertices for $T$ with $n$ orbits of connecting edges. Then $T/G$ has at most $(2k+1)n$ edges. Furthermore if $T$ is reduced and $k>1$ then $T/G$ has at most $2kn$ edges.
\end{lemma}

\begin{proof}
First observe that we can assume that each connecting group is a subgroup of a group in $\mathcal{P}$. Indeed suppose that there are $r>0$ connecting groups which are larger than $\mathcal{P}$. For each of the corresponding connecting edges we see that path consisting of it together with the branches of both its endpoints must be fixed by the connecting group, which is larger than $\mathcal{P}$. Since the action is $k$--acylindrical on groups larger than $\mathcal{P}$ each of these paths have length at most $k$. Thus we can collapse these paths to get a new tree with at most $kr$ fewer edges and a set of seed vertices with $r$ fewer connecting groups. 

Let $F \subset T$ be the forest consisting of $S$ together with every edge and vertex whose stabiliser is larger than $\mathcal{P}$. Since all of the connecting groups are contained in a member of $\mathcal{P}$ we see that $F$ must deformation retract to $S$. Let $R \subseteq T$ be a maximal subtree where every edge stabiliser is contained in a member of $\mathcal{P}$. Let $A = R \cap F$, the vertices with stabiliser larger than $\mathcal{P}$ and seed vertices which are in $R$. We define $\tilde{R}$ as the union of $R$ and the branches of each $v \in A$. 

\begin{figure}[h!]
\centering
%% Creator: Inkscape inkscape 0.92.4, www.inkscape.org
%% PDF/EPS/PS + LaTeX output extension by Johan Engelen, 2010
%% Accompanies image file 'BoundReducedTrees.pdf' (pdf, eps, ps)
%%
%% To include the image in your LaTeX document, write
%%   \input{<filename>.pdf_tex}
%%  instead of
%%   \includegraphics{<filename>.pdf}
%% To scale the image, write
%%   \def\svgwidth{<desired width>}
%%   \input{<filename>.pdf_tex}
%%  instead of
%%   \includegraphics[width=<desired width>]{<filename>.pdf}
%%
%% Images with a different path to the parent latex file can
%% be accessed with the `import' package (which may need to be
%% installed) using
%%   \usepackage{import}
%% in the preamble, and then including the image with
%%   \import{<path to file>}{<filename>.pdf_tex}
%% Alternatively, one can specify
%%   \graphicspath{{<path to file>/}}
%% 
%% For more information, please see info/svg-inkscape on CTAN:
%%   http://tug.ctan.org/tex-archive/info/svg-inkscape
%%
\begingroup%
  \makeatletter%
  \providecommand\color[2][]{%
    \errmessage{(Inkscape) Color is used for the text in Inkscape, but the package 'color.sty' is not loaded}%
    \renewcommand\color[2][]{}%
  }%
  \providecommand\transparent[1]{%
    \errmessage{(Inkscape) Transparency is used (non-zero) for the text in Inkscape, but the package 'transparent.sty' is not loaded}%
    \renewcommand\transparent[1]{}%
  }%
  \providecommand\rotatebox[2]{#2}%
  \newcommand*\fsize{\dimexpr\f@size pt\relax}%
  \newcommand*\lineheight[1]{\fontsize{\fsize}{#1\fsize}\selectfont}%
  \ifx\svgwidth\undefined%
    \setlength{\unitlength}{413.85826772bp}%
    \ifx\svgscale\undefined%
      \relax%
    \else%
      \setlength{\unitlength}{\unitlength * \real{\svgscale}}%
    \fi%
  \else%
    \setlength{\unitlength}{\svgwidth}%
  \fi%
  \global\let\svgwidth\undefined%
  \global\let\svgscale\undefined%
  \makeatother%
  \begin{picture}(1,0.26027397)%
    \lineheight{1}%
    \setlength\tabcolsep{0pt}%
    \put(0.03662987,0.3576232){\color[rgb]{0,0,0}\makebox(0,0)[lt]{\begin{minipage}{0.93116981\unitlength}\centering \end{minipage}}}%
    \put(0,0){\includegraphics[width=\unitlength,page=1]{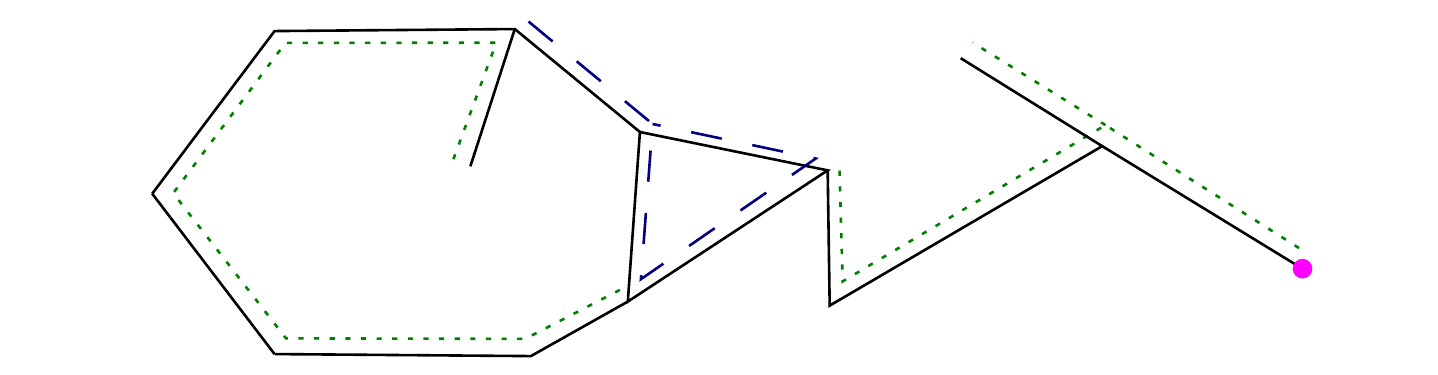}}%
    \put(0.31706873,0.18101292){\color[rgb]{0,0.50196078,0}\makebox(0,0)[t]{\lineheight{1.25}\smash{\begin{tabular}[t]{c}$F$\end{tabular}}}}%
    \put(0.47781858,0.17331873){\color[rgb]{0,0,0.50196078}\makebox(0,0)[t]{\lineheight{1.25}\smash{\begin{tabular}[t]{c}$R$\end{tabular}}}}%
    \put(0,0){\includegraphics[width=\unitlength,page=2]{BoundReducedTrees.pdf}}%
    \put(0.60657486,0.13437752){\color[rgb]{1,0,0}\makebox(0,0)[t]{\lineheight{1.25}\smash{\begin{tabular}[t]{c}$A$\end{tabular}}}}%
    \put(0.88118817,0.05819856){\color[rgb]{1,0,1}\makebox(0,0)[t]{\lineheight{1.25}\smash{\begin{tabular}[t]{c}$S$\end{tabular}}}}%
  \end{picture}%
\endgroup%

\caption{An example of what this construction may look like in the graph of groups. Here we see that $p := |A| = 3$, $\chi(R/G) = 0$ and that $R$ contains three orbits of connecting edges; the bottom right one and any two of the other three in $R/G$. Note the leftmost member of $A$ has valence $1$ in $R/G$ and isn't in $S$, so if the tree is reduced $F$ must extend beyond it. Because of this the length of the branch of this vertex is actually bounded above by $k-1$ instead of just the usual $k$.}
\end{figure}

Let $\left| A/G \right| = p$ and suppose $R$ contains $q$ connecting edges (up to equivalence) of some forest of influence $\Gamma$ grown from $S$. Now the branch of each vertex of $R$ must contain a vertex in $A$. Moreover this is unique as two distinct members of $A$ are influenced by different seed vertices and so must lie in different components of $\Gamma$. Therefore $R \cap \Gamma$ deformation retracts to $A$ and so $\chi(R/G) = p - q$. 

%Observe that each member of $A$ is in a different tree of influence of $\Gamma$. So we can find the Euler characteristic of $R/G$ by building it in the following way.

%\begin{itemize}
%\item Begin with the $p$ vertices of $A/G$.
%\item Add a non-connecting edge which starts at a vertex we have already included. The terminal vertex of this edge must be new as otherwise this we could find a path in $T$ between seed vertices which contains no connecting edge. Hence the Euler characteristic remains unchanged. Repeat this until all non-connecting edges are included.
%\item At this point we observe that all the vertices of $R/G$ must be included because the branch of any vertex in $R$ must contain a vertex in $A$.
%\item Add the connecting edges. Since all vertices are accounted for each connecting edge reduces the Euler characteristic by $1$.  
%\end{itemize}
%Thus $\chi(R/G) = p - q$. In particular we see that $q = p - \chi(R/G) > 0$. 

Recall from \thref{res:branch_edge} that every non-seed vertex has an edge with equal stabiliser to it and hence every vertex of $R/G$ with valence $1$ or $2$ (in $R/G$) must be in $A$ as $T$ is partially reduced on $\mathcal{P}$. Suppose that $R/G$ has $n_i$ vertices of valence $i$ and observe that $p \geq n_1 + n_2$. Observe that $\chi(R/G) = \frac{1}{2}(n_1 - n_3 - 2n_4 - \cdots)$ and hence ${\sum\limits_{i \geq 3} n_i \leq n_1 - 2\chi(R/G)}$. So 
\begin{eqnarray*}
\#(\text{edges of } R/G)  &   =  &  \#(\text{vertices of } R/G) - \chi(R/G) \\
                          & \leq &  2n_1 + n_2 - 3\chi(R/G) 
\end{eqnarray*}
Now the length of the branch of each $v \in A$ is at most $k$ as the action is $k$--acylindrical on $\mathcal{P}$. (If $v \in A \setminus S$ then $\stab v$ is lager than $\mathcal{P}$ and fixes its branch.) Hence we see that
\begin{eqnarray*}
\#(\text{edges of } \tilde{R}/G)  & \leq &  \#(\text{edges of } R/G) + kp \\
                                  & \leq &  2n_1 + n_2 + kp - 3\chi(R/G) \\
                                  & \leq &  (k+2)p - 3\chi(R/G)
\end{eqnarray*}
We now split into cases depending on the value of $\chi(R/G)$. First suppose that $\chi(R/G) \leq 0$. Then $q \geq p$ and so 
\begin{eqnarray*}
\#(\text{edges of } \tilde{R}/G)  & \leq &  (k+2)p - 3\chi(R/G) \\
                                  &  =   &  (k-1)p + 3(p - \chi(R/G)) \\
                                  & \leq &  (k-1)q + 3q \\
                                  &  =   &  (k+2)q
\end{eqnarray*}
Otherwise $\chi(R/G) = 1$ and we have $q = p - 1$ and so
\begin{eqnarray*}
\#(\text{edges of } \tilde{R}/G)  & \leq &  (k+2)p - 3\chi(R/G) \\
                                  & \leq &  (k+2)(q+1) - 3 \\
                                  &  =   &  (k+2)q + (k-1) \\
                                  & \leq &  (k+2)q + (k-1)q \\
                                  &  =   &  (2k+1)q
\end{eqnarray*}

Let $\tilde{T}$ be the tree obtained by collapsing each edge of $G \tilde{R}$ and let $\pi: T \rightarrow \tilde{T}$. Observe that $\tilde{S} := \pi(S)$ is a set of seed vertices for $\tilde{T}$ and that the number of connecting edges of $\tilde{S}$ is $n - q$. Hence by induction $\tilde{T}/G$ has at most $(2k+1)(n-q)$ edges. Combining this with the above we see that $T/G$ has at most $(2k+1)n$ edges as required. 

It remains to show the improved bound if $T$ is reduced and $k>1$. In this case any $v \in A \setminus S$ where $v/G$ has valence $1$ in $R/G$ must be the endpoint of at least $2$ edges not contained in $R/G$. This means that the path from $v$ to the corresponding $u \in S$ actually has length at most $k-1$. Hence in this case
\begin{eqnarray*}
\#(\text{edges of } \tilde{R}/G)  & \leq &  \#(\text{edges of } R/G) + k(p - n_1) + (k-1)n_1 \\
                                  & \leq &  (k+1)n_1 + n_2 + k(p - n_1) - 3\chi(R/G) \\
                                  & \leq &  (k+1)p - 3\chi(R/G) 
\end{eqnarray*}

The rest of the calculations are essentially the same as before and so are omitted for the sake of brevity. We will note however that we only actually obtain the improved bound if $k>1$. (In the case where $\chi(\tilde{R}/G) = 1$ we need $(k-2) \leq (k-2)q$. Since $q \geq 1$ we need $k \geq 2$.) $\square$
\end{proof}

We now have all the pieces we need to prove \thref{res:key_lemma_simple} and hence \thref{res:main_simple}.

\begin{proof}[Proof of \thref{res:key_lemma_simple}]
We will proceed by induction on $W_{\mathcal{P},S'}$. If $W_{\mathcal{P},S'} = 0$ then since $G$ isn't isomorphic to $\mathbb{Z}$ there is a single seed vertex in $T'$, which must be fixed by $G$. So the image of this vertex in $T$ is fixed by $G$ and so as $T$ is minimal it must just consist of a single vertex. 

So WLOG $W_{\mathcal{P},S'} > 0$. Start by using Stallings folding theorem (\thref{res:stallings_folding_thm}) to decompose $\alpha$ into folds $\alpha_i: T_{i-1} \rightarrow T_{i}$ and let $S_0 := S'$. Recursively for each $i>0$ if $S_{i-1}$ is defined and has connecting edges contained in $\mathcal{P}$ we obtain a set of seed vertices $S_i$ for $T_i$ at each step using either \thref{res:forests_enjoy_folds} or \thref{res:free_action_folds} (depending on if the action on $T_{i-1}$ is free). If the $\mathcal{P}$--weight at any step decreases then we are done by induction on $W_{\mathcal{P},S'}$. 

If instead $S_{i-1}$ has a connecting edge which is larger than $\mathcal{P}$ we apply \thref{res:collapse_connecting_edge}. We obtain $\overline{\Psi}: \overline{T}_{i-1} \rightarrow \overline{T}$ where the action on $\overline{T}$ is $k$-acylindrical on groups larger than $\mathcal{P}$. There's a set of seed vertices $\overline{S}_{i-1}$ for $\overline{T}_{i-1}$ with $W_{\mathcal{P},\overline{S}_{i-1}} \leq W_{\mathcal{P},S'} - 1$ and $T$ has at most $k$ more edges than $\overline{T}$. Hence by induction on $W_{\mathcal{P},S'}$ we see that $\overline{T}/G$ has at most $(W_{\mathcal{P},S'} - 1)k$ edges, hence $T/G$ has at most $W_{\mathcal{P},S'}k$ edges as desired. 
 
%Observe that if all the edges of $T$ are larger than $\mathcal{P}$ then we always arrive back in one of the above cases and so the number of edges of $T/G$ is bounded above by $kW_{\mathcal{P},S'}$. 

So WLOG we can assume that $S_i$ is always defined and that both the $\mathcal{P}$--weight is constant and that we never have a connecting edge of $\mathcal{P}$--weight $1$. Recall that \thref{res:forests_enjoy_folds} says that $\alpha_{i}(S_{i-1}) \subseteq S_i$ and since the number of connecting edges is bounded above by $\frac{1}{2}W_{\mathcal{P},S'}$ we must have $S_i = \alpha_{i}(S_{i-1})$ for all sufficiently large $i$. Thus by taking limits we see that there is a set of seed vertices $S$ for $T$ with $W_{\mathcal{P},S} = W_{\mathcal{P},S'}$. Now \thref{res:reduced_key_lemma} implies that the number of edges of $T/G$ is bounded above by $\left( \frac{2k+1}{2} \right) W_{\mathcal{P},S'}$. (Since each connecting edge has weight of at least $2$.) $\square$

%If $k>1$ and $T$ is reduced then we can improve the bound on the number of edges of $T/G$ to $2kW_{\mathcal{P},S'}$.
\end{proof}

\begin{proof}[Proof of \thref{res:main_simple}]
First we use Dunwoody's resolution lemma (\thref{res:dunwoody_resolution}) to get $G$ acting on a tree $T'$ which has at most $\alpha(G)$ orbits of edges together with a combinatorial map $\Psi: T' \rightarrow T$. Let $S$ be the set of vertices of $T'$ before subdividing. Observe that $S$ is a set of seed vertices for $T'$ and that it has $\mathcal{P}$--weight of at most $2^M C(G)$ since $\mathcal{P}$ has height $M$. Hence by \thref{res:key_lemma_simple} we see that $T$ has as most $(2k+1) 2^{M-1} C(G)$ edges. $\square$
\end{proof}

\section{Extending to the main results}\label{sec:extension}

Now that we have finished proving our simplified result it's time to extend it to get our main theorems. The first way we're going to do this is to show that we don't require $\mathcal{P}$ to be closed under taking subgroups; although it still must satisfy condition $(\dagger)$. (See Section~\ref{sec:statements} for the statement of $(\dagger)$.) The following is the analogue to \thref{res:key_lemma_simple} in this context.

\begin{lemma}\thlabel{res:key_lemma_fp}
Let $G$ be a non-cyclic group and let $\mathcal{P}$ be a conjugation invariant set of subgroups of $G$ which satisfies $(\dagger)$. Let $G$ act on a tree $T$ and suppose this action is $\mathcal{P}$--partially-reduced and $k$--acylindrical on a subgroups larger than $\mathcal{P}$. Let $G$ act on another tree $T'$ and there is a $G$--equivarient combinatorial map $\Psi: T' \rightarrow T$. Suppose that $T'$ has a set of seed vertices $S$ with finite $\mathcal{P}$--weight $W_{\mathcal{P},S}$ and $T$ is $\mathcal{P}$--closed. Then $T/G$ has at most $\left( \frac{2k+1}{2} \right) W_{\mathcal{P},S}$ edges.
\end{lemma}

The added difficulty is that \thref{res:forests_enjoy_folds} requires every connecting group to be in $\mathcal{P}$. Previously this was not an issue as every subgroup of $G$ was either in $\mathcal{P}$ or larger than it. We will solve this problem by adding extra folds at each step which forces the connecting groups to be in $\mathcal{P}$.

\begin{lemma}\thlabel{res:extra_mod_fp}
Suppose $\Psi: T' \rightarrow T$ and $\beta: T' \rightarrow R$ are $G$-equivarient combinatorial maps where $\beta$ factors through $\Psi$. Let $S$ be a set of seed vertices for $R$ with finite $\mathcal{P}$--weight $W_{\mathcal{P},S}$ and where none of the connecting groups are larger than $\mathcal{P}$. Suppose that $\mathcal{P}$ satisfies $(\dagger)$ and $T$ is $\mathcal{P}$--closed. Then there is a combinatorial map $\rho: R \rightarrow R'$ which factors through $\Psi$ such that $R'$ has a set of seed vertices $S' := \rho(S)$ such that $W_{\mathcal{P},S'} \leq W_{\mathcal{P},S}$ and all its connecting groups are in $\mathcal{P}$.
\end{lemma}

\begin{proof}
Let $\Gamma$ be any forest of influence which is grown from $S$. If each connecting group of $S$ is in $\mathcal{P}$ then we are done; so WLOG there is some connecting edge $e$ of $\Gamma$ which is not in $\mathcal{P}$. Since $\mathcal{P}$ satisfies $(\dagger)$ we have $H \in \mathcal{P}$ which is a minimal extension of $\stab e$ to $\mathcal{P}$ and acts elliptically on $R$. Suppose $H$ fixes the vertex $v$ in $R$. Let $p$ be the reduced edge path which starts at $v$ and has final edge $e$. Let $\tilde{p}$ be the union of $p$ together with the branch of each vertex on $p$. (See Figure~\ref{fig:FPModificationDomain}.) Since $T$ is $\mathcal{P}$--closed and the stabiliser of each edge $f$ in $\tilde{p}$ contains $\stab e$ we see that image of $f$ in $T$ must be stabilised by $H$. Let $\rho$ be the (possibly infinite) composition of type II folds which ``pulls'' $H$ onto each edge of $\tilde{p}$ and observe that this factors through $\Psi$ since $T$ is $\mathcal{P}$--closed. Hence if $e'$ is a connecting edge of $\Gamma$ with stabiliser in $\mathcal{P}$ then either $\stab e' = \stab \rho(e')$ or $W_{\mathcal{P},e'} < W_{\mathcal{P},e}$. Moreover $\rho(\Gamma)$ is a forest of influence grown from the seed vertices $\rho(S)$ with $W_{\mathcal{P},S'} \leq W_{\mathcal{P},S}$. Hence we can apply this process finitely many times until we get the result. $\square$

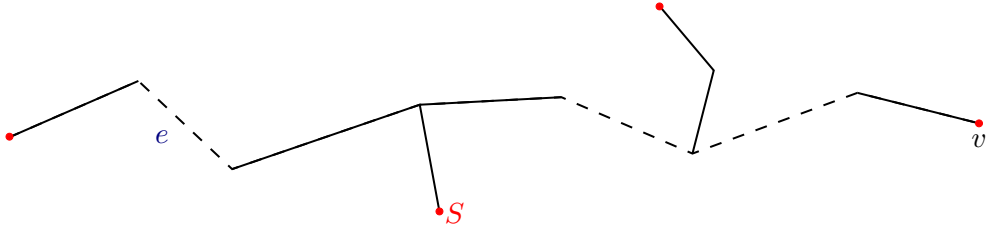
\begin{figure}[h!]
\centering
%% Creator: Inkscape inkscape 0.92.4, www.inkscape.org
%% PDF/EPS/PS + LaTeX output extension by Johan Engelen, 2010
%% Accompanies image file 'FPModificationDomain.pdf' (pdf, eps, ps)
%%
%% To include the image in your LaTeX document, write
%%   \input{<filename>.pdf_tex}
%%  instead of
%%   \includegraphics{<filename>.pdf}
%% To scale the image, write
%%   \def\svgwidth{<desired width>}
%%   \input{<filename>.pdf_tex}
%%  instead of
%%   \includegraphics[width=<desired width>]{<filename>.pdf}
%%
%% Images with a different path to the parent latex file can
%% be accessed with the `import' package (which may need to be
%% installed) using
%%   \usepackage{import}
%% in the preamble, and then including the image with
%%   \import{<path to file>}{<filename>.pdf_tex}
%% Alternatively, one can specify
%%   \graphicspath{{<path to file>/}}
%% 
%% For more information, please see info/svg-inkscape on CTAN:
%%   http://tug.ctan.org/tex-archive/info/svg-inkscape
%%
\begingroup%
  \makeatletter%
  \providecommand\color[2][]{%
    \errmessage{(Inkscape) Color is used for the text in Inkscape, but the package 'color.sty' is not loaded}%
    \renewcommand\color[2][]{}%
  }%
  \providecommand\transparent[1]{%
    \errmessage{(Inkscape) Transparency is used (non-zero) for the text in Inkscape, but the package 'transparent.sty' is not loaded}%
    \renewcommand\transparent[1]{}%
  }%
  \providecommand\rotatebox[2]{#2}%
  \newcommand*\fsize{\dimexpr\f@size pt\relax}%
  \newcommand*\lineheight[1]{\fontsize{\fsize}{#1\fsize}\selectfont}%
  \ifx\svgwidth\undefined%
    \setlength{\unitlength}{413.85826772bp}%
    \ifx\svgscale\undefined%
      \relax%
    \else%
      \setlength{\unitlength}{\unitlength * \real{\svgscale}}%
    \fi%
  \else%
    \setlength{\unitlength}{\svgwidth}%
  \fi%
  \global\let\svgwidth\undefined%
  \global\let\svgscale\undefined%
  \makeatother%
  \begin{picture}(1,0.21232877)%
    \lineheight{1}%
    \setlength\tabcolsep{0pt}%
    \put(0.03662987,0.3576232){\color[rgb]{0,0,0}\makebox(0,0)[lt]{\begin{minipage}{0.93116981\unitlength}\centering \end{minipage}}}%
    \put(-0.40547742,2.13428267){\color[rgb]{0,0,0}\makebox(0,0)[lt]{\begin{minipage}{2.54478703\unitlength}\centering \end{minipage}}}%
    \put(0,0){\includegraphics[width=\unitlength,page=1]{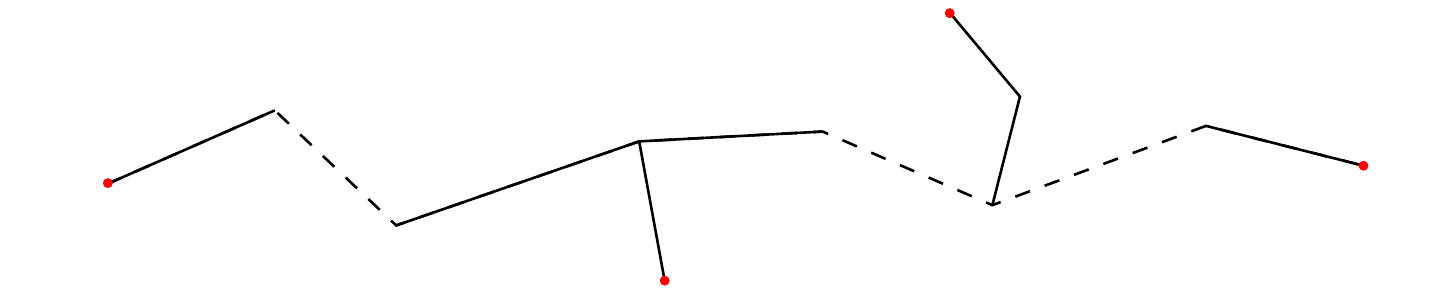}}%
    \put(0.21234135,0.07915424){\color[rgb]{0,0,0.50196078}\makebox(0,0)[t]{\lineheight{1.25}\smash{\begin{tabular}[t]{c}$e$\end{tabular}}}}%
    \put(0.4763333,0.0064275){\color[rgb]{1,0,0}\makebox(0,0)[t]{\lineheight{1.25}\smash{\begin{tabular}[t]{c}$S$\end{tabular}}}}%
    \put(0.94884062,0.07532804){\color[rgb]{0,0,0}\makebox(0,0)[t]{\lineheight{1.25}\smash{\begin{tabular}[t]{c}$v$\end{tabular}}}}%
  \end{picture}%
\endgroup%

\caption{An example of the domain we need to apply extra folds to. The reduced path from $v$ to $e$ is $p$ while the entire diagram is $\tilde{p}$. We need to apply folds to get $H$ fixing this whole region and not just $p$ as otherwise we would need additional seed vertices.}
\label{fig:FPModificationDomain}
\end{figure}
\end{proof}

\begin{proof}[Proof of \thref{res:key_lemma_fp}]
As before we will proceed by induction on $W_{\mathcal{P},S}$. If $W_{\mathcal{P},S} = 0$ then there is some vertex of $T'$ which is fixed by $G$. So the image of this vertex in $T$ is fixed by $G$ and so as $T$ is minimal it must just consist of just this single vertex. 

Use Stallings folding theorem (\thref{res:stallings_folding_thm}) to decompose $\alpha$ into folds ${\alpha^{(0)}_i: T^{(0)}_{i-1} \rightarrow T^{(0)}_{i}}$. We will iteratively define trees $T^{(j)}_{i}$ (for $i \geq j - 1$) and sets of seed vertices $S^{(j)}_{i}$ for $T^{(j)}_{i}$ (for $i = j$ and $i = j + 1$) together with maps $\alpha^{(j)}_i: T^{(j)}_{i-1} \rightarrow T^{(j)}_{i}$ and $\rho^{(j)}_{i}: T^{(j)}_{i} \rightarrow T^{(j+1)}_{i}$ as follows for each $j > 0$. 
\begin{equation*}
\begin{tikzcd}
T^{(0)}_{0} \arrow[r, "\alpha^{(0)}_1"] \arrow[d, "\rho^{(0)}_{0}", red] & T^{(0)}_{1} \arrow[r, "\alpha^{(0)}_2"] \arrow[d, "\rho^{(0)}_{1}"] & T^{(0)}_{2} \arrow[r, "\alpha^{(0)}_3"] \arrow[d, "\rho^{(0)}_{2}"] & \cdots \\
T^{(1)}_{0} \arrow[r, "\alpha^{(1)}_1", red] & T^{(1)}_{1} \arrow[r, "\alpha^{(1)}_2"] \arrow[d, "\rho^{(1)}_{1}", red] & T^{(1)}_{2} \arrow[r, "\alpha^{(1)}_3"] \arrow[d, "\rho^{(1)}_{2}"] & \cdots \\
& T^{(2)}_{1} \arrow[r, "\alpha^{(2)}_2", red] & T^{(2)}_{2} \arrow[r, "\alpha^{(2)}_3"] \arrow[d, "\rho^{(2)}_{2}", red] & \cdots \\
& & \vdots & \ddots
\end{tikzcd}
\end{equation*}
First use  define $\beta_j := \alpha^{(j-1)}_{j-1} \circ \rho^{(j-2)}_{j-2} \circ \cdots \circ \rho^{(1)}_{1} \circ \alpha^{(1)}_{1} \circ \rho^{(0)}_{0}$. (Part of the red path along the bottom of the diagram which ends with a right facing arrow.) Applying \thref{res:extra_mod_fp} to $\beta_j$ we obtain a map ${\rho^{(j)}_{j}: T^{(j-1)}_{j-1} \rightarrow T^{(j-1)}_{j-1}}$ where $S^{(j)}_{j-1} := \rho^{(j)}(S^{(j-1)}_{j-1})$ is a set of seed vertices where the connecting groups are in $\mathcal{P}$ and $W_{\mathcal{P},S^{(j)}_{j-1}} \leq W_{\mathcal{P},S^{(j-1)}_{j-1}}$. If $\alpha^{(j)}_i$ folds together the edges $e_1$ and $e_2$ we define $\alpha^{(j+1)}_i$ (for $i > j$) to be the fold (or identity) obtained by identifying the $\rho^{(j)}_{i+1}$ images of the $e_i$. Finally define $\rho^{(j+1)}_i$ (for $i > j + 1$) to make the above diagram commute. Finally \thref{res:forests_enjoy_folds} says that the fold $\alpha^{(j)}_j$ induces a set of seed vertices $S^{(j)}_{j}$ on $T^{(j)}_{j}$ with $W_{\mathcal{P},S^{(j)}_{j}} \leq W_{\mathcal{P},S^{(j)}_{j-1}}$. (If a separating edge is larger than $\mathcal{P}$ then we can reduce the $\mathcal{P}$--weight by collapsing at most $k$ edges using \thref{res:collapse_connecting_edge} and then proceeding by induction on $W_{\mathcal{P},S}$.) 

As $\alpha = \cdots \alpha^{(0)}_2 \circ \alpha^{(0)}_1$ we see that $\alpha = \cdots \alpha^{(2)}_2 \circ \rho{(1)} \circ \alpha^{(1)}_1 \circ \rho^{(0)}$ by the definitions of $\alpha^{(j)}_i$ and $\rho^{(j)}$. Moreover at each step we have a set of seed vertices with non-increasing $\mathcal{P}$--weight, the number of orbits of connecting edges are non-decreasing and hence that all but finitely many of the sets of seed vertices are the image of the seed vertices at the previous level. At this point the proof is exactly the same as the proof of \thref{res:key_lemma_simple}. $\square$
\end{proof}

\begin{proof}[Proof of \thref{res:main_result} \ref{pt:main_fp}]
First we use Dunwoody's resolution lemma (\thref{res:dunwoody_resolution}) to get $G$ acting on a tree $T'$ which has at most $\alpha(G)$ orbits of edges together with a combinatorial map $\Psi: T' \rightarrow T$. Let $S$ be the set of vertices of $T'$ before subdividing. Observe that $S$ is a set of seed vertices for $T'$ and that it has $\mathcal{P}$--weight of at most $2^M C(G)$ since $\mathcal{P}$ has height $M$. Hence by \thref{res:key_lemma_simple} we see that $T$ has as most $(2k-1) 2^{M-1} C(G)$ edges. $\square$
\end{proof}

It remains to extend \thref{res:main_simple} to finitely generated groups. An immediate hurdle for this is the lack of Dunwoody's resolution lemma (\thref{res:dunwoody_resolution}), as this only holds for (almost) finitely presented groups. Instead take a finite generating set $X$ for $G$ and consider the free group $F(X)$ acting freely on a tree $T'$. Whenever $G$ acts on a tree $T$ we see that there is an $F(X)$--equivarient combinatorial map $T' \rightarrow T$. It's this map which we intend to decompose into folds and apply our prior methods to. 

Before stating the analogue to \thref{res:key_lemma_simple} we first need to extend the definition of $\mathcal{P}$--weights.

\begin{definition}
Let $\phi: H \rightarrow G$ be a surjective homomorphism of groups. Let $\mathcal{P}$ be a set of subgroups for $G$ which is closed under conjugation. We define the \emph{$\mathcal{P}$--weight} of a $K \leq H$, (denoted $W_{\mathcal{P},\phi,K}$ or $W_{\mathcal{P},K}$ if $\phi$ is understood,) to be equal to $W_{\mathcal{P},\phi(K)}$. If $H$ acts on a tree with a set of seed vertices $S$ then we define its \emph{$\mathcal{P}$--weight} $W_{\mathcal{P},\phi,K}$ (or $W_{\mathcal{P},K}$ if $\phi$ is understood) to be equal to the sum of the $\mathcal{P}$--weights of the connecting groups.
\end{definition}

\begin{lemma}\thlabel{res:key_lemma_P_closed}
Let $\mathcal{P}$ be a conjugation invariant set of subgroups of $G$ which is closed under taking subgroups. Let $G$ act on a tree $T$ and suppose this action is $\mathcal{P}$--partially-reduced and $k$--acylindrical on a subgroups larger than $\mathcal{P}$. Suppose also that $G'$ is a countable group acting on a tree $T'$ and that the following conditions hold.
\begin{itemize}
\item There is a surjective homomorphism $\phi: G' \rightarrow G$.
\item The kernel of $\phi$ has trivial intersection with every edge stabiliser of $T'$.
\item There is a $G'$-equivarient combinatorial map $\Psi: T' \rightarrow T$. (Where the action of $G'$ on $T$ is the natural one given by $\phi$.)
\item $T'$ has a set of seed vertices $S'$ with $\mathcal{P}$--weight $W_{\mathcal{P},S'}$.
\end{itemize}
Then $T/G$ has at most $\left(\frac{2k+1}{2}\right) W_{\mathcal{P},S'}$ edges. Furthermore if either $T$ is reduced and $k>1$ or all of the edges of $T$ have stabiliser of size greater than $\mathcal{P}$ then $T/G$ has at most $kW_{\mathcal{P},S'}$ edges. 
\end{lemma}

First observe that the following variation of \thref{res:forests_enjoy_folds} and \thref{res:free_action_folds} holds with the exact same proof as before.

\begin{lemma}\thlabel{res:forests_folds_P_closed}
Let $\phi: H \rightarrow G$ be a surjective homomorphism and $\mathcal{P}$ is a conjugation invariant set of subgroups for $G$.  Suppose that there is a $H$--equivariant map $\alpha: R \rightarrow \tilde{R}$ which is a fold. Suppose that the kernel of $\phi$ has trivial intersection with each vertex stabiliser of $R$. Suppose that $S$ is a set of seed vertices for $R$ where the image of each connecting group is in $\mathcal{P}$. Then there is a set of seed vertices $\tilde{S}$ for $\tilde{R}$ with ${W_{\mathcal{P},\phi,\tilde{S}} \leq W_{\mathcal{P},\phi,S}}$. Moreover if ${W_{\mathcal{P},\phi,\tilde{S}} = W_{\mathcal{P},\phi,S}}$ then $\alpha|_{S}$ is injective.
\end{lemma}

After each step elements in the kernel of $\phi$ may end up acting elliptically on the intermediate tree. As such we need a way of modifying a group $G_i$ and tree $T_i$ which essentially keeps the action and map $\Psi_i: T_i \rightarrow T$ but removes problematic group elements found in the kernel of $\phi: G_i \rightarrow G$.

\begin{lemma}\thlabel{res:extra_mod_P_closed}
Let $\phi: G' \rightarrow G$ be a surjective group homomorphism and suppose that $G'$ acts on a tree $T'$. Then there's a group $G''$ acting on a tree $T''$ together with surjective homomorphisms $\phi': G'' \rightarrow G$ and $\sigma: G' \rightarrow G''$ and a $G'$--equivarient simplical map $\rho: T' \rightarrow T''$. (The action of $G'$ on $T''$ is given by $\sigma$.) Additionally $\rho/G: T'/G \rightarrow T''/G$ is a homeomorphism of graphs with $\phi$ having trivial intersection with each edge stabiliser of $T'$ and $\sigma(\stab e) = \sigma(\stab \rho(e))$. Moreover the kernel of $\phi'$ has trivial intersection with every vertex stabiliser of $T''$. 

Hence if $T'$ has a set of seed vertices $S'$ then $S'' := \rho(S')$ is a set of seed vertices for $T''$ with $W_{\mathcal{P},S''} = W_{\mathcal{P},S'}$.
\end{lemma}

\begin{proof}
We define $G''$ as the fundamental group of a graph of groups decomposition corresponding to $T'$ but with each vertex label replaced with its image under $\phi$ and let $T''$ be the corresponding Bass-Serre tree. This naturally induces maps $\phi': G'' \rightarrow G$ and $\tilde{\phi}: G' \rightarrow G''$ and $\tilde{\Psi}: T'' \rightarrow T$. Moreover we naturally get a set of seed vertices $S''$ for $T''$ with $W_{\mathcal{P}_{\phi'},S''} = W_{\mathcal{P}_{\phi},S'}$. $\square$
\end{proof}

\begin{proof}[Proof of \thref{res:key_lemma_P_closed}]
As before we will proceed by induction on $W_{\mathcal{P},S}$. If $W_{\mathcal{P},S} = 0$ then there is some vertex of $T'$ which is fixed by $G$. So the image of this vertex in $T$ is fixed by $G$ and so as $T$ is minimal it must just consist of this single vertex. 

Use Stallings folding theorem (\thref{res:stallings_folding_thm}) to decompose $\alpha$ into folds ${\alpha^{(0)}_i: T^{(0)}_{i-1} \rightarrow T^{(0)}_{i}}$ and let $G^{(0)} := G'$. We will iteratively define groups $G^{(j)}$ which act on trees $T^{(j)}_{i}$ with sets of seed vertices $S^{(j)}_{i}$ for $T^{(j)}_{i}$. Also we will define maps $\sigma^{(j)}: G^{(j)} \rightarrow G^{(j+1)}$ and $\sigma^{(j)}: G^{(j)} \rightarrow G$ together with $\alpha^{(j)}_i: T^{(j)}_{i-1} \rightarrow T^{(j)}_{i}$ and $\rho^{(j)}_{i}: T^{(j)}_{i} \rightarrow T^{(j+1)}_{i}$ (for $i \geq j - 1$).
\begin{equation*}
\begin{tikzcd}
T^{(0)}_{0} \arrow[r, "\alpha^{(0)}_1"] \arrow[d, "\rho^{(0)}_{0}", red] & T^{(0)}_{1} \arrow[r, "\alpha^{(0)}_2"] \arrow[d, "\rho^{(0)}_{1}"] & T^{(0)}_{2} \arrow[r, "\alpha^{(0)}_3"] \arrow[d, "\rho^{(0)}_{2}"] & \cdots \\
T^{(1)}_{0} \arrow[r, "\alpha^{(1)}_1", red] & T^{(1)}_{1} \arrow[r, "\alpha^{(1)}_2"] \arrow[d, "\rho^{(1)}_{1}", red] & T^{(1)}_{2} \arrow[r, "\alpha^{(1)}_3"] \arrow[d, "\rho^{(1)}_{2}"] & \cdots \\
& T^{(2)}_{1} \arrow[r, "\alpha^{(2)}_2", red] & T^{(2)}_{2} \arrow[r, "\alpha^{(2)}_3"] \arrow[d, "\rho^{(2)}_{2}", red] & \cdots \\
& & \vdots & \ddots
\end{tikzcd}
\end{equation*}
First we define $\beta_j := \alpha^{(j)}_{j} \circ \rho^{(j-1)}_{j-1} \circ \cdots \circ \rho^{(1)}_{1} \circ \alpha^{(1)}_{1} \circ \rho^{(0)}_{0}$. (Part of the red path along the bottom of the diagram which ends with a right facing arrow.) First use \thref{res:extra_mod_P_closed} on $\beta_j$ we obtain a map $\rho^{(j)}: T^{(j)}_{j} \rightarrow T^{(j+1)}_{j}$ together with group homomorphisms $\sigma^{(j)}: G^{(j)} \rightarrow G^{(j+1)}$ and $\phi^{(j+1)}: G^{(j+1)} \rightarrow G$. Also $S^{(j+1)}_{j} := \rho^{(j)}(S^{(j)}_{j})$ is a set of seed vertices where the images of the connecting groups are in $\mathcal{P}$ and $W_{\mathcal{P},S^{(j)}_{j-1}} \leq W_{\mathcal{P},S^{(j-1)}_{j-1}}$. If $\alpha^{(j)}_i$ folds together the edges $e_1$ and $e_2$ we define $\alpha^{(j+1)}_i$ (for $i > j$) to be the fold (or identity) obtained by identifying the $\rho^{(j)}_{i+1}$ images of the $e_i$. Finally define $\rho^{(j+1)}_i$ (for $i > j + 1$) to make the above diagram commute. 

Finally \thref{res:forests_folds_P_closed} says that the fold $\alpha^{(j)}_j$ induces a set of seed vertices $S^{(j)}_{j}$ on $T^{(j)}_{j}$ with $W_{\mathcal{P},S^{(j)}_{j}} \leq W_{\mathcal{P},S^{(j)}_{j-1}}$. (If a separating edge is larger than $\mathcal{P}$ then we can reduce the $\mathcal{P}$--weight by collapsing at most $k$ edges using \thref{res:collapse_connecting_edge} and then proceeding by induction.) 

As $\alpha = \cdots \alpha^{(0)}_2 \circ \alpha^{(0)}_1$ we see that $\alpha = \cdots \alpha^{(2)}_2 \circ \rho^{(1)}_{1} \circ \alpha^{(1)}_1 \circ \rho^{(0)}_{0}$ by the definitions of $\alpha^{(j)}_i$ and $\rho^{(j)}_{i}$. Moreover at each step we have a set of seed vertices with non-increasing $\mathcal{P}$--weight, the number of orbits of connecting edges are non-decreasing and hence that all but finitely many of the sets of seed vertices are the image of the seed vertices at the previous level. At this point the proof is exactly the same as the proof of \thref{res:key_lemma_simple}. $\square$
\end{proof}

\begin{proof}[Proof of \thref{res:main_result} \ref{pt:main_P_closed}]
Pick a minimal generating set $X$ for $G$ and let $G' := F(X)$. Let $\phi: G' \rightarrow G$ be the natural projection and let $T'$ be the tree corresponding to the rose with $\rank G$ petals labelled by the elements of $X$. Let $\Psi: T' \rightarrow T$ be any $G'$--equivarient combinatorial map. If $G$ acts freely on $T$ then (make new lemma) implies that $T$ has at most $3(\rank G - 1)$ edges. Let $S$ be the set of vertices of $T'$ before subdividing. Observe that $S$ is a set of seed vertices for $T'$ and that it has $\mathcal{P}$--weight of at most $2^M \rank(G)$ since $\mathcal{P}$ has height $M$. Hence by \thref{res:key_lemma_P_closed} $T$ has as most $(2k+1) 2^{M-1} (\rank(G) - 1)$ edges. If all of the edges of $T$ have stabiliser larger than $\mathcal{P}$ or $T$ is reduced with $k>1$ then the number of edges is in fact bounded by $k 2^{M} (\rank(G) - 1)$. $\square$
\end{proof}

\section{Sharpness of bounds}\label{sec:sharp}

We will now restrict our attention to the case where $\mathcal{P} = \mathtt{1}$, the collection which only contains the trivial subgroup. In other words we are to consider actions which are $k$--acylindrical. In \cite{WeidmannInitial} Weidmann showed that a finitely generated group acting $k$--acylindrically on a tree where all the edges have non-trivial stabiliser has at most $2k(\rank G - 1)$ orbits of edges. The purpose of this section is to prove \thref{res:sharp_bound}, which improves this bound to $2k \left( \rank G - \frac{5}{4} \right)$ edges, and to construct an example which shows that this is the best possible bound. (\thref{res:sharp_example}) Additionally we'll refine this further to $2k \left( \rank G - \frac{3}{2} \right)$ the the case where the group is torsion-free. 

\begin{proof}[Proof of \thref{res:sharp_example}]
We need to show that for any integers $k \geq 1$ and $r \geq 2$ that there is a group of rank $r$ acting $k$-acylindrically on a tree with $\left\lfloor 2k \left( \rank G - \frac{5}{4} \right) \right\rfloor$ orbits of edges, none of which have trivial stabilisers. Pick distinct primes $p$ and $q$ such that $(p-1)(q-1) \geq 2(r-2)$. Let $G := \left\langle a, h_1, \cdots, h_{r-1} \:\: | \:\: a^{pq} = 1 \right\rangle \cong \left( \frac{\mathbb{Z}}{pq\mathbb{Z}} \right) * F_{r-1}$ and note that $\rank G = r$. We will now construct a tree $T$ for $G$ to act on. Start with the graph of groups decomposition consisting of the rose with $r-1$ petals representing the $h_i$ and with a single vertex on the loop representing $h_1$ with label $\left\langle a \right\rangle$. (In the diagrams we take $k = 4$ and $r = 3$.) 

\noindent%% Creator: Inkscape inkscape 0.92.4, www.inkscape.org
%% PDF/EPS/PS + LaTeX output extension by Johan Engelen, 2010
%% Accompanies image file 'SharpExampleWTorsionStep0.pdf' (pdf, eps, ps)
%%
%% To include the image in your LaTeX document, write
%%   \input{<filename>.pdf_tex}
%%  instead of
%%   \includegraphics{<filename>.pdf}
%% To scale the image, write
%%   \def\svgwidth{<desired width>}
%%   \input{<filename>.pdf_tex}
%%  instead of
%%   \includegraphics[width=<desired width>]{<filename>.pdf}
%%
%% Images with a different path to the parent latex file can
%% be accessed with the `import' package (which may need to be
%% installed) using
%%   \usepackage{import}
%% in the preamble, and then including the image with
%%   \import{<path to file>}{<filename>.pdf_tex}
%% Alternatively, one can specify
%%   \graphicspath{{<path to file>/}}
%% 
%% For more information, please see info/svg-inkscape on CTAN:
%%   http://tug.ctan.org/tex-archive/info/svg-inkscape
%%
\begingroup%
  \makeatletter%
  \providecommand\color[2][]{%
    \errmessage{(Inkscape) Color is used for the text in Inkscape, but the package 'color.sty' is not loaded}%
    \renewcommand\color[2][]{}%
  }%
  \providecommand\transparent[1]{%
    \errmessage{(Inkscape) Transparency is used (non-zero) for the text in Inkscape, but the package 'transparent.sty' is not loaded}%
    \renewcommand\transparent[1]{}%
  }%
  \providecommand\rotatebox[2]{#2}%
  \newcommand*\fsize{\dimexpr\f@size pt\relax}%
  \newcommand*\lineheight[1]{\fontsize{\fsize}{#1\fsize}\selectfont}%
  \ifx\svgwidth\undefined%
    \setlength{\unitlength}{413.85826772bp}%
    \ifx\svgscale\undefined%
      \relax%
    \else%
      \setlength{\unitlength}{\unitlength * \real{\svgscale}}%
    \fi%
  \else%
    \setlength{\unitlength}{\svgwidth}%
  \fi%
  \global\let\svgwidth\undefined%
  \global\let\svgscale\undefined%
  \makeatother%
  \begin{picture}(1,0.14383562)%
    \lineheight{1}%
    \setlength\tabcolsep{0pt}%
    \put(0.03662987,0.3576232){\color[rgb]{0,0,0}\makebox(0,0)[lt]{\begin{minipage}{0.93116981\unitlength}\centering \end{minipage}}}%
    \put(0,0){\includegraphics[width=\unitlength,page=1]{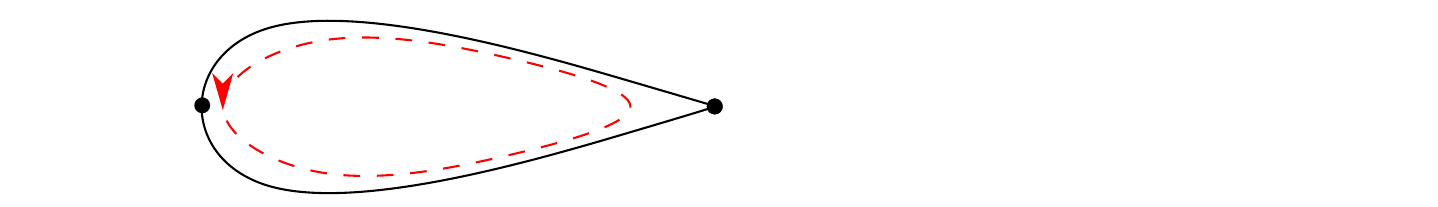}}%
    \put(0.10748477,0.06315341){\color[rgb]{0,0,0.50196078}\makebox(0,0)[t]{\lineheight{1.25}\smash{\begin{tabular}[t]{c}$\langle a \rangle$\end{tabular}}}}%
    \put(0.4965513,0.08614861){\color[rgb]{0,0,0.50196078}\makebox(0,0)[t]{\lineheight{1.25}\smash{\begin{tabular}[t]{c}$1$\end{tabular}}}}%
    \put(0.18507187,0.06507455){\color[rgb]{1,0,0}\makebox(0,0)[t]{\lineheight{1.25}\smash{\begin{tabular}[t]{c}$h_1$\end{tabular}}}}%
    \put(-0.40547742,2.13428267){\color[rgb]{0,0,0}\makebox(0,0)[lt]{\begin{minipage}{2.54478703\unitlength}\centering \end{minipage}}}%
    \put(0,0){\includegraphics[width=\unitlength,page=2]{SharpExampleWTorsionStep0.pdf}}%
    \put(0.80969813,0.06507665){\color[rgb]{1,0,0}\makebox(0,0)[t]{\lineheight{1.25}\smash{\begin{tabular}[t]{c}$h_2$\end{tabular}}}}%
    \put(0.87605793,0.06388959){\color[rgb]{0,0.50196078,0}\makebox(0,0)[t]{\lineheight{1.25}\smash{\begin{tabular}[t]{c}$1$\end{tabular}}}}%
  \end{picture}%
\endgroup%

Subdivide the loop representing $h_1$ so that it consists of $\left\lceil \frac{k}{2} \right\rceil + 2$ edges. Apply folds of type II to ``pull'' $a$ onto each vertex on the loop except the central one. 

\noindent%% Creator: Inkscape inkscape 0.92.4, www.inkscape.org
%% PDF/EPS/PS + LaTeX output extension by Johan Engelen, 2010
%% Accompanies image file 'SharpExampleWTorsionStep1.pdf' (pdf, eps, ps)
%%
%% To include the image in your LaTeX document, write
%%   \input{<filename>.pdf_tex}
%%  instead of
%%   \includegraphics{<filename>.pdf}
%% To scale the image, write
%%   \def\svgwidth{<desired width>}
%%   \input{<filename>.pdf_tex}
%%  instead of
%%   \includegraphics[width=<desired width>]{<filename>.pdf}
%%
%% Images with a different path to the parent latex file can
%% be accessed with the `import' package (which may need to be
%% installed) using
%%   \usepackage{import}
%% in the preamble, and then including the image with
%%   \import{<path to file>}{<filename>.pdf_tex}
%% Alternatively, one can specify
%%   \graphicspath{{<path to file>/}}
%% 
%% For more information, please see info/svg-inkscape on CTAN:
%%   http://tug.ctan.org/tex-archive/info/svg-inkscape
%%
\begingroup%
  \makeatletter%
  \providecommand\color[2][]{%
    \errmessage{(Inkscape) Color is used for the text in Inkscape, but the package 'color.sty' is not loaded}%
    \renewcommand\color[2][]{}%
  }%
  \providecommand\transparent[1]{%
    \errmessage{(Inkscape) Transparency is used (non-zero) for the text in Inkscape, but the package 'transparent.sty' is not loaded}%
    \renewcommand\transparent[1]{}%
  }%
  \providecommand\rotatebox[2]{#2}%
  \newcommand*\fsize{\dimexpr\f@size pt\relax}%
  \newcommand*\lineheight[1]{\fontsize{\fsize}{#1\fsize}\selectfont}%
  \ifx\svgwidth\undefined%
    \setlength{\unitlength}{413.85826772bp}%
    \ifx\svgscale\undefined%
      \relax%
    \else%
      \setlength{\unitlength}{\unitlength * \real{\svgscale}}%
    \fi%
  \else%
    \setlength{\unitlength}{\svgwidth}%
  \fi%
  \global\let\svgwidth\undefined%
  \global\let\svgscale\undefined%
  \makeatother%
  \begin{picture}(1,0.18493151)%
    \lineheight{1}%
    \setlength\tabcolsep{0pt}%
    \put(0.03662987,0.3576232){\color[rgb]{0,0,0}\makebox(0,0)[lt]{\begin{minipage}{0.93116981\unitlength}\centering \end{minipage}}}%
    \put(0,0){\includegraphics[width=\unitlength,page=1]{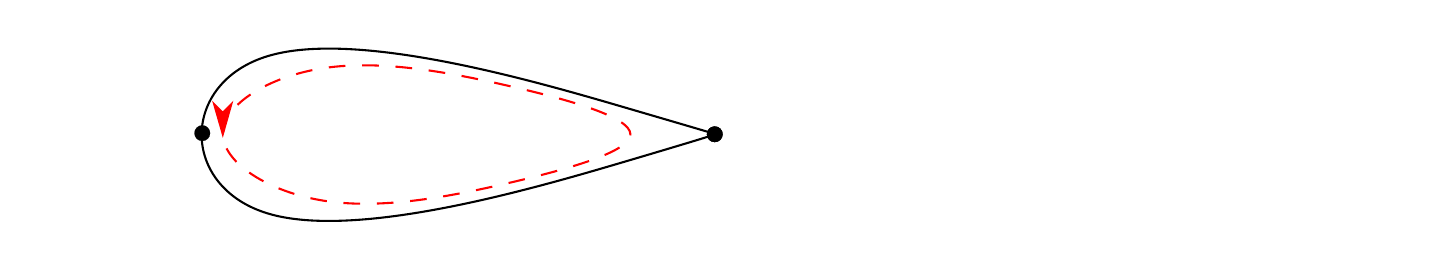}}%
    \put(0.10748477,0.08490032){\color[rgb]{0,0,0.50196078}\makebox(0,0)[t]{\lineheight{1.25}\smash{\begin{tabular}[t]{c}$\langle a \rangle$\end{tabular}}}}%
    \put(0.4965513,0.10789552){\color[rgb]{0,0,0.50196078}\makebox(0,0)[t]{\lineheight{1.25}\smash{\begin{tabular}[t]{c}$1$\end{tabular}}}}%
    \put(0.18507187,0.08682146){\color[rgb]{1,0,0}\makebox(0,0)[t]{\lineheight{1.25}\smash{\begin{tabular}[t]{c}$h_1$\end{tabular}}}}%
    \put(-0.40547742,2.13428267){\color[rgb]{0,0,0}\makebox(0,0)[lt]{\begin{minipage}{2.54478703\unitlength}\centering \end{minipage}}}%
    \put(0,0){\includegraphics[width=\unitlength,page=2]{SharpExampleWTorsionStep1.pdf}}%
    \put(0.21444995,0.16415196){\color[rgb]{0,0,0.50196078}\makebox(0,0)[t]{\lineheight{1.25}\smash{\begin{tabular}[t]{c}$\langle a \rangle$\end{tabular}}}}%
    \put(0.21565467,0.00390123){\color[rgb]{0,0,0.50196078}\makebox(0,0)[t]{\lineheight{1.25}\smash{\begin{tabular}[t]{c}$\langle a \rangle$\end{tabular}}}}%
    \put(0,0){\includegraphics[width=\unitlength,page=3]{SharpExampleWTorsionStep1.pdf}}%
    \put(0.80969813,0.08682324){\color[rgb]{1,0,0}\makebox(0,0)[t]{\lineheight{1.25}\smash{\begin{tabular}[t]{c}$h_2$\end{tabular}}}}%
    \put(0.87605793,0.08563619){\color[rgb]{0,0.50196078,0}\makebox(0,0)[t]{\lineheight{1.25}\smash{\begin{tabular}[t]{c}$1$\end{tabular}}}}%
  \end{picture}%
\endgroup%

Next subdivide the edges on the loop representing $h_1$ which are adjacent to the central vertex into $\left\lfloor \frac{k}{2} \right\rfloor$ sub-edges. Apply folds of type II which ``pull'' $a^p$ along one of these series of edges and ``pull'' $a^q$ along the other. We see that the central vertex has stabiliser $H := \left\langle b,c \:\: | \:\: b^{q} = c^{p} = 1 \right\rangle \cong \left( \frac{\mathbb{Z}}{p\mathbb{Z}} \right) * \left( \frac{\mathbb{Z}}{q\mathbb{Z}} \right)$ where $b = a^{p}$ and $c = \left( a^{q} \right)^{h_1}$. 

\noindent%% Creator: Inkscape inkscape 0.92.4, www.inkscape.org
%% PDF/EPS/PS + LaTeX output extension by Johan Engelen, 2010
%% Accompanies image file 'SharpExampleWTorsionStep2.pdf' (pdf, eps, ps)
%%
%% To include the image in your LaTeX document, write
%%   \input{<filename>.pdf_tex}
%%  instead of
%%   \includegraphics{<filename>.pdf}
%% To scale the image, write
%%   \def\svgwidth{<desired width>}
%%   \input{<filename>.pdf_tex}
%%  instead of
%%   \includegraphics[width=<desired width>]{<filename>.pdf}
%%
%% Images with a different path to the parent latex file can
%% be accessed with the `import' package (which may need to be
%% installed) using
%%   \usepackage{import}
%% in the preamble, and then including the image with
%%   \import{<path to file>}{<filename>.pdf_tex}
%% Alternatively, one can specify
%%   \graphicspath{{<path to file>/}}
%% 
%% For more information, please see info/svg-inkscape on CTAN:
%%   http://tug.ctan.org/tex-archive/info/svg-inkscape
%%
\begingroup%
  \makeatletter%
  \providecommand\color[2][]{%
    \errmessage{(Inkscape) Color is used for the text in Inkscape, but the package 'color.sty' is not loaded}%
    \renewcommand\color[2][]{}%
  }%
  \providecommand\transparent[1]{%
    \errmessage{(Inkscape) Transparency is used (non-zero) for the text in Inkscape, but the package 'transparent.sty' is not loaded}%
    \renewcommand\transparent[1]{}%
  }%
  \providecommand\rotatebox[2]{#2}%
  \newcommand*\fsize{\dimexpr\f@size pt\relax}%
  \newcommand*\lineheight[1]{\fontsize{\fsize}{#1\fsize}\selectfont}%
  \ifx\svgwidth\undefined%
    \setlength{\unitlength}{413.85826772bp}%
    \ifx\svgscale\undefined%
      \relax%
    \else%
      \setlength{\unitlength}{\unitlength * \real{\svgscale}}%
    \fi%
  \else%
    \setlength{\unitlength}{\svgwidth}%
  \fi%
  \global\let\svgwidth\undefined%
  \global\let\svgscale\undefined%
  \makeatother%
  \begin{picture}(1,0.18493151)%
    \lineheight{1}%
    \setlength\tabcolsep{0pt}%
    \put(0.03662987,0.3576232){\color[rgb]{0,0,0}\makebox(0,0)[lt]{\begin{minipage}{0.93116981\unitlength}\centering \end{minipage}}}%
    \put(0,0){\includegraphics[width=\unitlength,page=1]{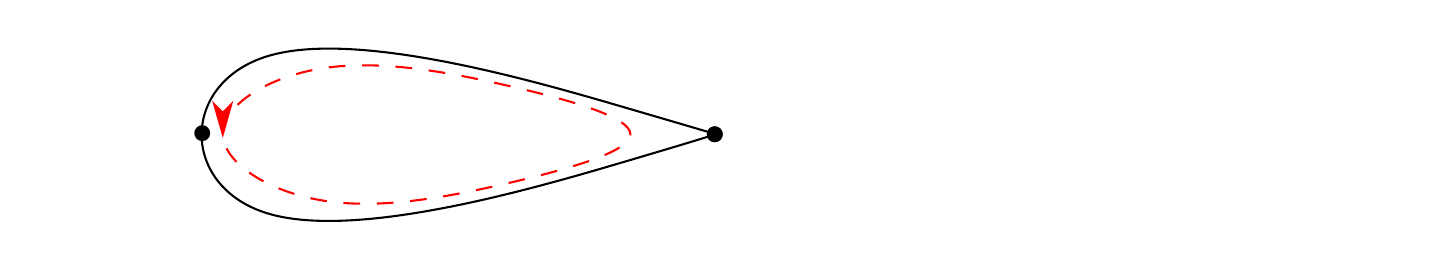}}%
    \put(0.10748477,0.08490032){\color[rgb]{0,0,0.50196078}\makebox(0,0)[t]{\lineheight{1.25}\smash{\begin{tabular}[t]{c}$\langle a \rangle$\end{tabular}}}}%
    \put(0.4965513,0.10789552){\color[rgb]{0,0,0.50196078}\makebox(0,0)[t]{\lineheight{1.25}\smash{\begin{tabular}[t]{c}$\langle b,c \rangle$\end{tabular}}}}%
    \put(0.18507187,0.08682146){\color[rgb]{1,0,0}\makebox(0,0)[t]{\lineheight{1.25}\smash{\begin{tabular}[t]{c}$h_1$\end{tabular}}}}%
    \put(-0.40547742,2.13428267){\color[rgb]{0,0,0}\makebox(0,0)[lt]{\begin{minipage}{2.54478703\unitlength}\centering \end{minipage}}}%
    \put(0,0){\includegraphics[width=\unitlength,page=2]{SharpExampleWTorsionStep2.pdf}}%
    \put(0.21444995,0.16415196){\color[rgb]{0,0,0.50196078}\makebox(0,0)[t]{\lineheight{1.25}\smash{\begin{tabular}[t]{c}$\langle a \rangle$\end{tabular}}}}%
    \put(0.21565467,0.00390123){\color[rgb]{0,0,0.50196078}\makebox(0,0)[t]{\lineheight{1.25}\smash{\begin{tabular}[t]{c}$\langle a \rangle$\end{tabular}}}}%
    \put(0,0){\includegraphics[width=\unitlength,page=3]{SharpExampleWTorsionStep2.pdf}}%
    \put(0.3517901,0.14808885){\color[rgb]{0,0,0.50196078}\makebox(0,0)[t]{\lineheight{1.25}\smash{\begin{tabular}[t]{c}$\langle a^p \rangle$\end{tabular}}}}%
    \put(0.35219169,0.01878043){\color[rgb]{0,0,0.50196078}\makebox(0,0)[t]{\lineheight{1.25}\smash{\begin{tabular}[t]{c}$\langle a^q \rangle$\end{tabular}}}}%
    \put(0,0){\includegraphics[width=\unitlength,page=4]{SharpExampleWTorsionStep2.pdf}}%
    \put(0.80969813,0.08682324){\color[rgb]{1,0,0}\makebox(0,0)[t]{\lineheight{1.25}\smash{\begin{tabular}[t]{c}$h_2$\end{tabular}}}}%
    \put(0.87605793,0.08563619){\color[rgb]{0,0.50196078,0}\makebox(0,0)[t]{\lineheight{1.25}\smash{\begin{tabular}[t]{c}$1$\end{tabular}}}}%
  \end{picture}%
\endgroup%

For $i \leq 2(r-2)$ we define $g_i = b^{x+1} c^{y+1}$ where $i = (p-1)x + y$ for $x,y \in \mathbb{Z}$ with $0 \leq y \leq p-2$. Since $(p-1)(q-1) \geq 2(r-2)$ we see that these $g_i$ represent pairwise non-conjugate elements of $H$. Subdivide the loop representing each $h_j$ (for $2 \leq j \leq r-1$) into $2k$ sub-edges, then apply folds of type II which ``pulls'' $g_{2j-3}$ along $k$ edges starting at one end and ``pulls'' $g_{2j-2}$ along $k$ edges starting at the other. 

\noindent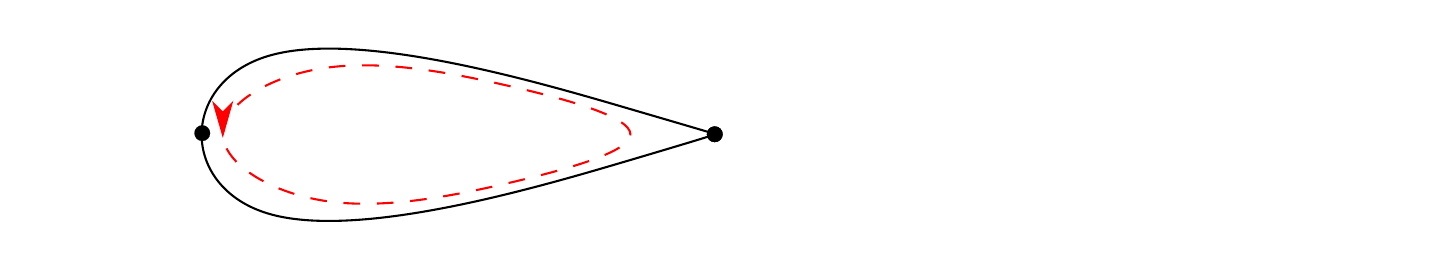 

Observe that this decomposition has $\left\lfloor (2r - \frac{5}{2})k \right\rfloor$ edges. (It's not reduced in general, but recall that this isn't a condition for this result.) It remains to check that the corresponding Bass-Serre tree is $k$--acylindrical. The elements of $G$ which act elliptically (upto conjugacy) are powers of $a$ and elements of $H$; so these are the ones we need to check fix a region of bounded diameter.  

First consider elements of $H$. The elements which fix an edge of our tree are (powers of) the $g_i$, $b$ and $c$ (upto conjugacy). As $b$ and $c$ are conjugate to powers of $a$ we'll leave these for now. Now observe that each $g_i$ has a different image in the ablieanisation of $H$; hence distinct $g_i$ are in different conjugacy classes. Moreover each cyclic root-closed subgroup of $H$ is malnormal in it. Hence each (power of) $g_i$ only fixes $k$ edges. 

\begin{figure}[h!]
\centering
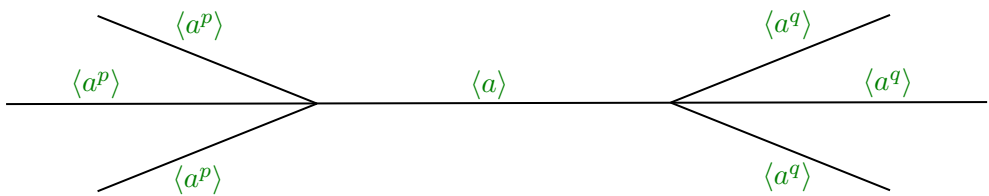
\caption{The region of the Bass-Serre tree which is fixed by some power of $a$, together with their stabilisers. The central arc with stabiliser $\left\langle a \right\rangle$ has length $\left\lceil \frac{k}{2} \right\rceil$ while each ``offshoot'' with stabiliser either $\left\langle a^p \right\rangle$ or $\left\langle a^q \right\rangle$ has length $\left\lfloor \frac{k}{2} \right\rfloor$.}
\label{fig:FixatorOfAPowers}
\end{figure}

We now need to consider powers of $a$. Let $1 \leq m < pq$ and look at the fixator of $a^m$. If $p \: | \: m$ then the fixator of $a^m$ consists of a central vertex with $q+1$ ``offshoots'', one of length $\left\lceil \frac{k}{2} \right\rceil$ and the rest of length $\left\lfloor \frac{k}{2} \right\rfloor$. In other words the fixator consists of the left and the centre parts of Figure~\ref{fig:FixatorOfAPowers}. This region has diameter $k$ and so we are fine. Likewise for the case where $q \: | \: m$. We cannot have $pq \: | \: m$ as $m < pq$. Finally if $m$ is coprime to $pq$ then $p^m$ just fixes a path of length $\left\lceil \frac{k}{2} \right\rceil$; the middle section of Figure~\ref{fig:FixatorOfAPowers}. 

Building an example which is maximal for torsion-free groups is similar. First we need $a$ to have infinite order and so $G \cong F_r$. The initial splitting is defined in the same way as before. Next we subdivide the loop representing $h_1$ into $k$ subedges and apply folds of type II so that each edge in this loop has label $\left\langle a \right\rangle$. (If $k=1$ then we collapse either of the initial edges of the loop instead.) The central vertex now has label isomorphic to the free group of rank $2$ which is generated by $b := a$ and $c := a^{h_1}$. We now subdivide and fold onto the loops representing the rest of the $h_j$ as before. $\square$

\noindent%% Creator: Inkscape inkscape 0.92.4, www.inkscape.org
%% PDF/EPS/PS + LaTeX output extension by Johan Engelen, 2010
%% Accompanies image file 'SharpExampleTF.pdf' (pdf, eps, ps)
%%
%% To include the image in your LaTeX document, write
%%   \input{<filename>.pdf_tex}
%%  instead of
%%   \includegraphics{<filename>.pdf}
%% To scale the image, write
%%   \def\svgwidth{<desired width>}
%%   \input{<filename>.pdf_tex}
%%  instead of
%%   \includegraphics[width=<desired width>]{<filename>.pdf}
%%
%% Images with a different path to the parent latex file can
%% be accessed with the `import' package (which may need to be
%% installed) using
%%   \usepackage{import}
%% in the preamble, and then including the image with
%%   \import{<path to file>}{<filename>.pdf_tex}
%% Alternatively, one can specify
%%   \graphicspath{{<path to file>/}}
%% 
%% For more information, please see info/svg-inkscape on CTAN:
%%   http://tug.ctan.org/tex-archive/info/svg-inkscape
%%
\begingroup%
  \makeatletter%
  \providecommand\color[2][]{%
    \errmessage{(Inkscape) Color is used for the text in Inkscape, but the package 'color.sty' is not loaded}%
    \renewcommand\color[2][]{}%
  }%
  \providecommand\transparent[1]{%
    \errmessage{(Inkscape) Transparency is used (non-zero) for the text in Inkscape, but the package 'transparent.sty' is not loaded}%
    \renewcommand\transparent[1]{}%
  }%
  \providecommand\rotatebox[2]{#2}%
  \newcommand*\fsize{\dimexpr\f@size pt\relax}%
  \newcommand*\lineheight[1]{\fontsize{\fsize}{#1\fsize}\selectfont}%
  \ifx\svgwidth\undefined%
    \setlength{\unitlength}{413.85826772bp}%
    \ifx\svgscale\undefined%
      \relax%
    \else%
      \setlength{\unitlength}{\unitlength * \real{\svgscale}}%
    \fi%
  \else%
    \setlength{\unitlength}{\svgwidth}%
  \fi%
  \global\let\svgwidth\undefined%
  \global\let\svgscale\undefined%
  \makeatother%
  \begin{picture}(1,0.18493151)%
    \lineheight{1}%
    \setlength\tabcolsep{0pt}%
    \put(0.03662987,0.3576232){\color[rgb]{0,0,0}\makebox(0,0)[lt]{\begin{minipage}{0.93116981\unitlength}\centering \end{minipage}}}%
    \put(0,0){\includegraphics[width=\unitlength,page=1]{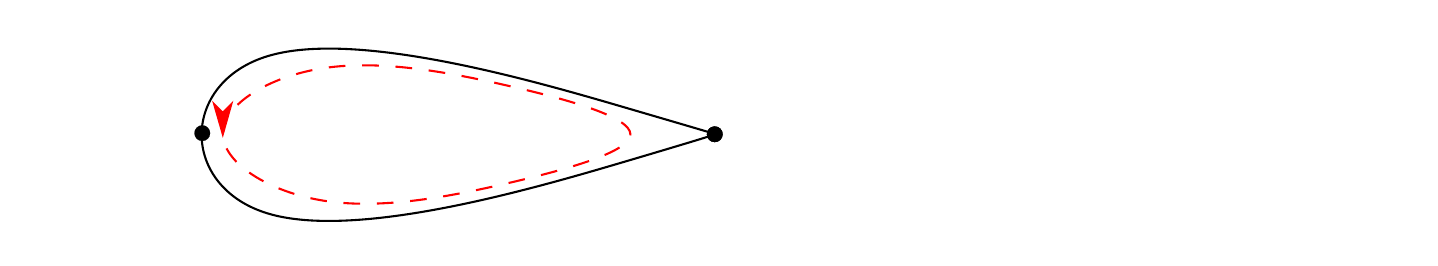}}%
    \put(0.10748477,0.08490032){\color[rgb]{0,0,0.50196078}\makebox(0,0)[t]{\lineheight{1.25}\smash{\begin{tabular}[t]{c}$\langle a \rangle$\end{tabular}}}}%
    \put(0.4965513,0.10789552){\color[rgb]{0,0,0.50196078}\makebox(0,0)[t]{\lineheight{1.25}\smash{\begin{tabular}[t]{c}$\langle b,c \rangle$\end{tabular}}}}%
    \put(0.18507187,0.08682146){\color[rgb]{1,0,0}\makebox(0,0)[t]{\lineheight{1.25}\smash{\begin{tabular}[t]{c}$h_1$\end{tabular}}}}%
    \put(-0.40547742,2.13428267){\color[rgb]{0,0,0}\makebox(0,0)[lt]{\begin{minipage}{2.54478703\unitlength}\centering \end{minipage}}}%
    \put(0,0){\includegraphics[width=\unitlength,page=2]{SharpExampleTF.pdf}}%
    \put(0.25069427,0.16415196){\color[rgb]{0,0,0.50196078}\makebox(0,0)[t]{\lineheight{1.25}\smash{\begin{tabular}[t]{c}$\langle a \rangle$\end{tabular}}}}%
    \put(0.25189899,0.00390123){\color[rgb]{0,0,0.50196078}\makebox(0,0)[t]{\lineheight{1.25}\smash{\begin{tabular}[t]{c}$\langle a \rangle$\end{tabular}}}}%
    \put(0,0){\includegraphics[width=\unitlength,page=3]{SharpExampleTF.pdf}}%
    \put(0.80969813,0.08682324){\color[rgb]{1,0,0}\makebox(0,0)[t]{\lineheight{1.25}\smash{\begin{tabular}[t]{c}$h_2$\end{tabular}}}}%
    \put(0.90142864,0.08563619){\color[rgb]{0,0,0.50196078}\makebox(0,0)[t]{\lineheight{1.25}\smash{\begin{tabular}[t]{c}$\langle g_1,g_2 \rangle$\end{tabular}}}}%
    \put(0.71559676,0.16135054){\color[rgb]{0,0,0.50196078}\makebox(0,0)[t]{\lineheight{1.25}\smash{\begin{tabular}[t]{c}$\langle g_1 \rangle$\end{tabular}}}}%
    \put(0.61567317,0.14171235){\color[rgb]{0,0,0.50196078}\makebox(0,0)[t]{\lineheight{1.25}\smash{\begin{tabular}[t]{c}$\langle g_1 \rangle$\end{tabular}}}}%
    \put(0.81763516,0.15847773){\color[rgb]{0,0,0.50196078}\makebox(0,0)[t]{\lineheight{1.25}\smash{\begin{tabular}[t]{c}$\langle g_1 \rangle$\end{tabular}}}}%
    \put(0.6172593,0.02475901){\color[rgb]{0,0,0.50196078}\makebox(0,0)[t]{\lineheight{1.25}\smash{\begin{tabular}[t]{c}$\langle g_2 \rangle$\end{tabular}}}}%
    \put(0.71559676,0.00900958){\color[rgb]{0,0,0.50196078}\makebox(0,0)[t]{\lineheight{1.25}\smash{\begin{tabular}[t]{c}$\langle g_2 \rangle$\end{tabular}}}}%
    \put(0.81552035,0.01150659){\color[rgb]{0,0,0.50196078}\makebox(0,0)[t]{\lineheight{1.25}\smash{\begin{tabular}[t]{c}$\langle g_2 \rangle$\end{tabular}}}}%
  \end{picture}%
\endgroup%

\end{proof}

Before proving \thref{res:sharp_bound}, which will show the above examples are the best possible, it's useful to compare their constructions to the proof of \thref{res:key_lemma_P_closed}. We start with a single orbit of seed vertices with representative stabiliser $\left\langle a \right\rangle$. Our initial folds induce a new orbit of seed vertices on the central vertex. Moreover the two connecting edges on the loop representing $h_1$ are now non-trivial and so we collapse it. In doing this we'll reduce the $\mathtt{1}$--weight by two but only collapse either $\left\lfloor \frac{3k}{2} \right\rfloor$ or $k$ edges, depending on which construction we're talking about. This is less than the $2k$ edges theoretically allowed by the lemma. Continuing we then successively collapse each loop; getting rid of the maximally possible $2k$ edges each time. 

With this comparison in mind we will now show that such an inefficiency must occur at a particular point in \thref{res:key_lemma_P_closed}. Specifically whenever a vertex first obtains a non-cyclic stabiliser.

\begin{lemma}\thlabel{res:cyclic_inefficiency}
Let $\alpha: R \rightarrow \tilde{R}$ be a fold which factors through $\Psi: T' \rightarrow T$ and let $S$ be a set of seed vertices for $R$. Suppose that the action on $T$ is $k$--acylindrical. Suppose also that both every vertex stabiliser of $R$ is cyclic and every connecting group of $S$ is trivial. Then one of the following holds 
\begin{itemize}
\item Every vertex stabiliser of $R$ is cyclic and we can find a set of seed vertices $\tilde{S}$ for $\tilde{R}$ such that every connecting group of $\tilde{S}$ is trivial and $W_{\mathtt{1},\tilde{S}} \leq W_{\mathtt{1},S}$. 
\item There is a simplicial map $\rho: \tilde{R} \rightarrow \tilde{R}'$ which factors through $\Psi$ and we can collapse at most $\left\lfloor \frac{3k}{2} \right\rfloor$ orbits of edges of $T$ to get a new tree $\overline{T}$ such that the following holds. Let $\overline{R}$ be the tree obtained by collapsing the edges of $\tilde{R}'$ corresponding to $T \rightarrow \overline{T}$. There is a set of seed vertices $\overline{S}$ for $\overline{R}$ with $W_{\mathtt{1},\overline{S}} \leq W_{\mathtt{1},S} - 2$. Moreover if $G$ is torsion-free then we can obtain $\overline{T}$ by collapsing at most $k$ edges of $T$. 
\end{itemize}
\end{lemma}

\begin{proof}
Suppose that $\alpha$ folds together the edges $e_1 = [x,y_1]$ and $e_2 = [x,y_2]$ to an edge $e' = [x',y']$. If there's a forest of influence containing both $e_1$ and $e_2$ then we can just take $\tilde{S} := \alpha(S)$ and we end up in the first outcome listed in the statement. The same applies the fold is of type I or II and either of the $y_i$ have trivial stabiliser. Similarly we can take $\tilde{S} := \alpha(S) \cup G \left\lbrace y' \right\rbrace$ if $\alpha$ is a fold of type III and the $y_i$ have trivial stabilisers. 

Consider the case where $\alpha$ is a fold of type III, the stabiliser of the $y_i$ are non-trivial and $hy_1 = y_2$. Suppose that $y_1$ is influenced by $u$ and $u' := \alpha (u)$. Both $u'$ and $y'$ are fixed by the stabiliser of $y_1$; hence as in the proof of \thref{res:main_simple} we can apply a series of folds $\rho: \tilde{R} \rightarrow \tilde{R}'$ which factors through $\Psi$ such that the reduced edge path $f'$ from $u'$ to $y'$ consists of at most $k$ edges and is injective under $\delta: \tilde{R}' \rightarrow T$. Now we define $\overline{T}$ by collapsing the image of $f'$ in $T$ and $\overline{R}$ as in the statement. Note that the image of $S$ in $\overline{T}$ is a set of seed vertices with $W_{\mathtt{1},\overline{S}} \leq W_{\mathtt{1},S} - 2$ as every connecting edge is trivial and $\chi(\overline{R}/G) = \chi(R/G) + 2$. 
 
\begin{figure}[h!]
\centering
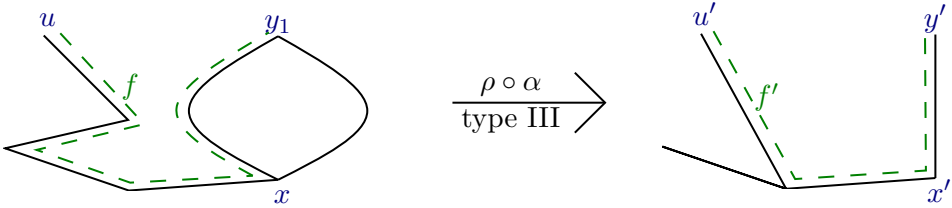
\caption{The case where the fold $\alpha$ is of type III. The edge path labelled $f$ is fixed by some non-trivial member of $G$, hence its image in $\tilde{R}$ has diameter at most $k$. In particular if we replace $f$ with its image in $\tilde{R}$ we see that the path $f'$ has length at most $k$.}
\end{figure}

So now we can assume that the fold is of type I or II where the stabiliser of both $y_1$ and $y_2$ are non-trivial. If $\alpha$ is a fold of type I we say that $y_i$ is influenced by $u_i$ and $f_i$ be the branch of $y_i$. If instead $\alpha$ is a fold of type II we say that $x$ is influenced by $u_1$ and $y_2$ is influenced by $u_2$. We also let $f_1$ be the union of the branch of $x$ together with $e_1$ and let $f_2$ be the branch of $y_2$. 

First consider the case where $u_1$ is inequivalent to $u_2$. Let $\rho: \tilde{R} \rightarrow \tilde{R}'$ be the composition of folds on $f_1$ and $f_2$ (separately) which causes $\gamma: \overline{R'} \rightarrow T$ to be locally injective on $f_1$ and $f_2$. Let $u'_i$ be the vertex closest to $y'' := \rho(y')$ with stabiliser equal to $\stab y_i$ and let $f'_i$ be the reduced edge path from $u'_i$ to $y''$. We now define $\overline{T}$ by collapsing the (orbits of the) images of $f'_1$ and $f'_2$ in $T$ and $\overline{R}$ as in the statement. Observe that we have a set of seed vertices $\overline{S}$ for $\overline{R}$ defined to be the union of the image of $S \setminus G \left\lbrace u_1, u_2 \right\rbrace$ together with the image of $G \left\lbrace y'' \right\rbrace$ and observe that $W_{\mathtt{1},\overline{S}} \leq W_{\mathtt{1},S} - 2$. It remains to bound the number of edges we've collapsed. If $f'_i$ consists of more than a single vertex let $g_i$ be a group element which fixes $u'_i$ but no edge of $f'_i$. Then since $\stab u'_i = \stab u_i$ is cyclic we see that $\stab y_i$ fixes the reduced edge path $f'_i \cup g_i f'_i$. So since the action on $T$ is $k$--acylindrical we see that $f'_1$ and $f'_2$ each consist of at most $\frac{k}{2}$ edges each and so we have collapsed at most $k$ edges total.

\begin{figure}[h!]
\centering
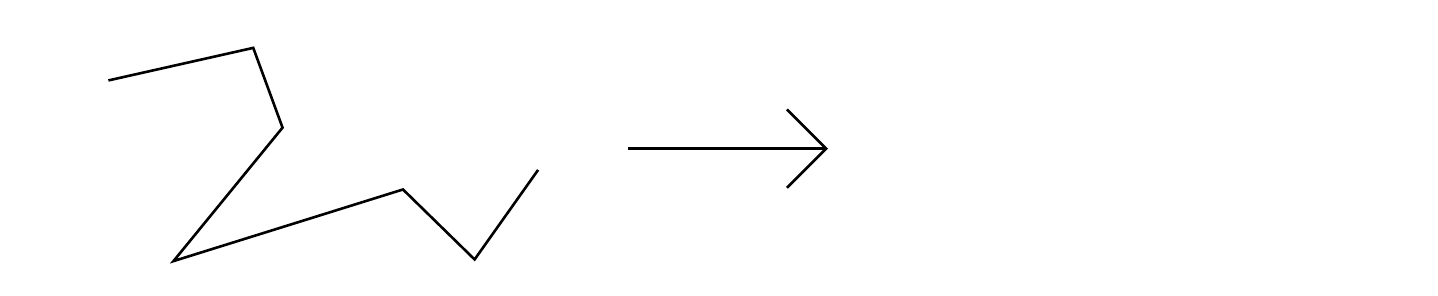 \\
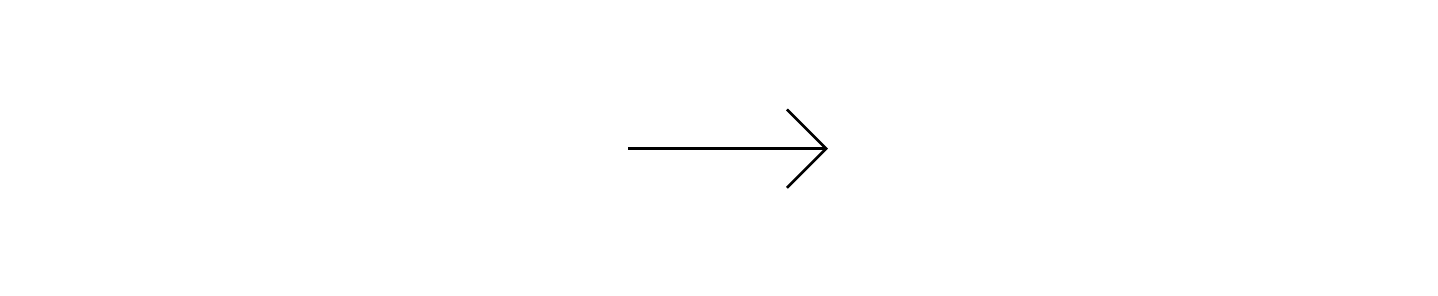
\caption{The graph of groups in the cases where the $u_i$ aren't equivalent.}
\end{figure}

Now consider the case $u_2 = h u_1$ for some $h \in G$. Define $\rho: \tilde{R} \rightarrow \tilde{R}'$, $u'_i$, $y''$ and $f'_i$ as before. Let $\tilde{f}$ be the path from $u'_1$ to $h^{-1}u'_2$. Define $\overline{T}$ by collapsing the images of $f'_1$, $f'_2$ and $\tilde{f}$ in $T$. We have a set of seed vertices $\overline{S}$ for $\overline{R}$ defined to be the union of the image of $S \setminus G \left\lbrace u_1 \right\rbrace$ together with the image of $G \left\lbrace y'' \right\rbrace$ and again we have $W_{\mathtt{1},\overline{S}} \leq W_{\mathtt{1},S} - 2$. It now remains to bound the number of edges collapsed. As before the paths $f'_i$ have at most $\frac{k}{2}$ edges. If $\tilde{f}$ consists of just a single vertex $u'_1$ then we are done as before. If not then observe that $f'_1 \cup \tilde{f} \cup h^{-1}f'_2$ is a reduced edge path from $y''$ to $h^{-1}y''$. If $G$ is torsion-free then as $\stab u_1$ is cyclic then there is some non-trivial subgroup which fixes $f'_1 \cup \tilde{f} \cup h^{-1}f'_2$ and so we've collapsed at most $k$ edges. If $G$ isn't torsion-free then we are only guaranteed to have a non-trivial subgroup which fixes $f'_1 \cup \tilde{f}$. Thus $f'_1 \cup \tilde{f} \cup h^{-1}f'_2$ has at most $k + \frac{k}{2} = \frac{3k}{2}$ edges. $\square$

\begin{figure}[h!]
\centering
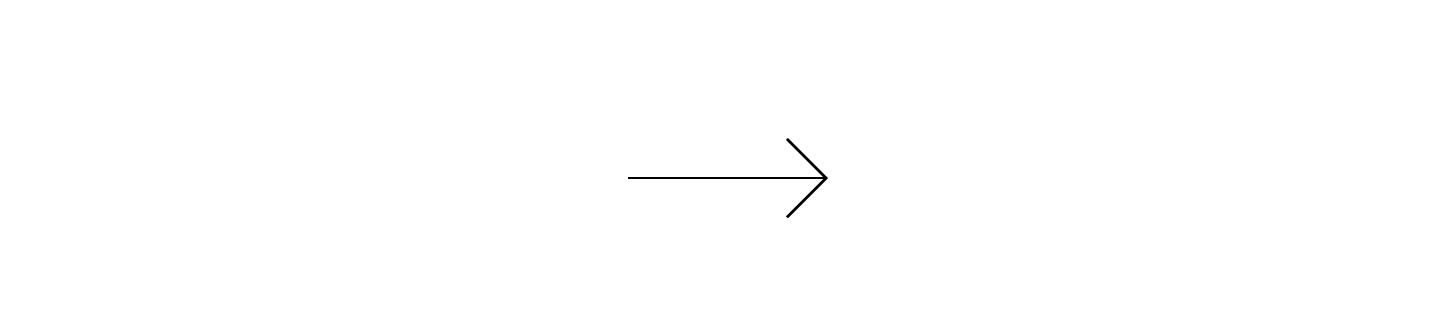 \\
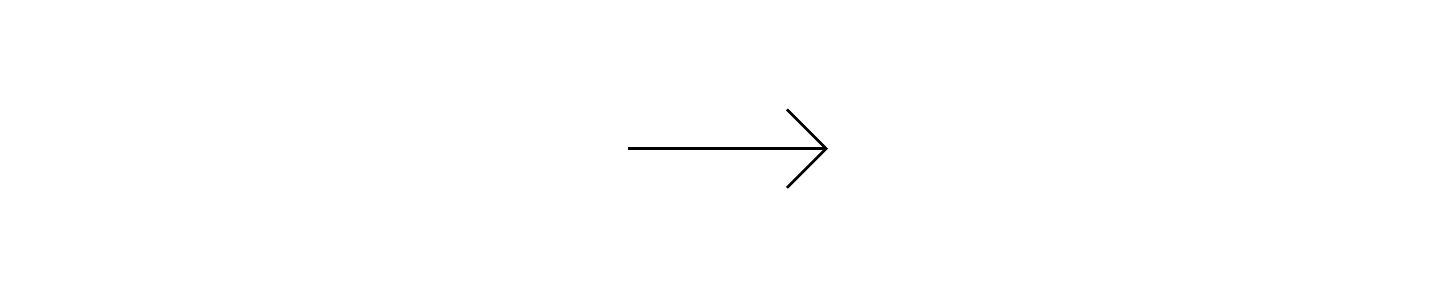 
\caption{The graph of groups in the cases where $u_1$ and $u_2$ are equivalent.}
\end{figure}
\end{proof}

\begin{proof}[Proof of \thref{res:sharp_bound}]
Proceed as in the proof of \thref{res:main_result} \ref{pt:main_P_closed}. We have a homomorphism $\phi: F(X) \rightarrow G$ where $X$ is a minimal generating set for $G$ and combinatorial map $\Psi: T' \rightarrow T$. As in the proof of \thref{res:key_lemma_P_closed} we now decompose $\Psi$ into folds $\alpha_i: T_{i-1} \rightarrow T_i$. We then apply \thref{res:cyclic_inefficiency} to each $\alpha_i$ in turn until one of them causes us to collapse edges. (This must happen eventually as all the edges of $T$ have non-trivial stabiliser.) Then apply \thref{res:main_result} \ref{pt:main_P_closed} to the collapsed tree to get the desired bound. $\square$
\end{proof}

\end{document}